\documentclass[11pt]{article}

\usepackage{amsfonts}
\usepackage{amscd}
\usepackage{amssymb}
\usepackage{amsthm}
\usepackage{amsmath}
\usepackage{stmaryrd}
\usepackage{graphicx}
\usepackage{color}
\usepackage{verbatim}
\usepackage{mathrsfs}
\usepackage{dsfont}
\usepackage{tikz}
\usepackage{subfigure}
\usetikzlibrary{arrows,decorations.pathmorphing,backgrounds,positioning,fit}
\usepackage{pgfplots}
\usepackage{bbm}

\colorlet{Green}{black!20!green}

 \theoremstyle{plain}
\newtheorem{thm}{Theorem}[section]
\newtheorem{conj}[thm]{Conjecture}
\newtheorem{lemma}[thm]{Lemma}
\newtheorem{prop}[thm]{Proposition}
\newtheorem{cor}[thm]{Corollary}
\theoremstyle{definition}
\newtheorem{example}[thm]{Example}
\newtheorem{defn}[thm]{Definition}
\newtheorem{remark}[thm]{Remark}
\theoremstyle{remark}

\numberwithin{equation}{section}

\def\cD{\mathcal{D}}

\def\cN{\mathcal{N}}

\def\AA{\mathbb{A}}

\def\CC{\mathbb{C}}

\def\EE{\mathbb{E}}
\def\FF{\mathbb{F}}

\def\HH{\mathbb{H}}

\def\KK{\mathbb{K}}

\def\NN{\mathbb{N}}

\def\PP{\mathbb{P}}
\def\QQ{\mathbb{Q}}
\def\RR{\mathbb{R}}

\def\TT{\mathbb{T}}

\def\ZZ{\mathbb{Z}}

\def\fc{\mathfrak{c}}

\def\fr{\mathfrak{r}}
\def\fs{\mathfrak{s}}

\def\dim{\mathrm{dim}}

\def\tr{\mathrm{tr}}

\def\Waff{W_{\mathrm{aff}}}
\def\Saff{S_{\mathrm{aff}}}

\def\scA{\mathscr{A}}

\def\scP{\mathscr{P}}

\def\scH{\mathscr{H}}

\newcommand{\vect}[1]{\boldsymbol{#1}}

\colorlet{lightgray}{black!20}

\tikzset{
    partial ellipse/.style args={#1:#2:#3}{
        insert path={+ (#1:#3) arc (#1:#2:#3)}
    }
}

\makeatletter
\renewcommand{\@makefnmark}{\mbox{\textsuperscript{}}}
\makeatother

\title{Buildings, groups of Lie type, and random walks}
\author{
J. Parkinson\footnote{Research partly supported under the Australian Research Council (ARC) discovery grant DP110103205.}
}
\date{}

\begin{document}

\maketitle

\begin{abstract}
 In this paper we survey the theory of random walks on buildings and associated groups of Lie type and Kac-Moody groups. We begin with an introduction to the theory of Coxeter systems and buildings, taking a largely combinatorial perspective. We then survey the theory of random walks on buildings, and show how this theory leads to limit theorems for random walks on the associated groups.
\end{abstract}


\begin{center}
\textit{To Professor Woess on the occasion of his $\mathit{60}^{th}$ birthday}
\end{center}

\section*{Introduction}\label{sect:1}

Probability theory on real Lie groups is a classical area, with beautiful results obtained in the 1960s, 1970s, and 1980s (see, for example \cite{Bou:81,Gui:77,Gui:80,Kai:87,SV:73,Vir:70,Weh:62}).  It is the purpose of the current paper to survey more recent results dealing with probability theory on groups of Lie type defined over other fields, and extensions of these results into the setting of Kac-Moody groups. A unifying feature of these works is the use of a combinatorial/geometric object called the \textit{building} of the group, which in some ways plays a role analogous to the symmetric space of a real Lie group. In fact the building becomes the main object of interest, and so this survey is really about random walks on buildings, with applications to random walks on the associated groups.

Buildings were invented by Jacques Tits in the 1950s in an attempt to give a uniform geometric interpretation of semi-simple Lie groups. He achieved this goal spectacularly by classifying the class of irreducible thick \textit{spherical} buildings of rank at least~$3$, showing that this class of buildings is essentially equivalent to the class of simple linear algebraic groups of relative rank at least~$3$, simple classical linear groups, and certain related groups called groups of mixed type (see \cite{Tit:74}). Since their invention the scope of building theory has expanded immensely, with \textit{affine} buildings playing an important role in the study of Lie groups over $p$-adic fields, and \textit{twin} buildings utilised extensively in the theory of Kac-Moody groups.

For the purpose of this introduction a \textit{building} consists of a set $\Delta$ (whose elements are called \textit{chambers}) and a way of measuring distance between chambers. This measurement is not simply a numerical distance, instead the distance between two chambers is an element of a \textit{Coxeter group}~$W$ associated to the building. Thus there is a ``$W$-valued distance function'' $\delta:\Delta\times\Delta\to W$, satisfying various axioms making $\Delta$ into a kind of ``$W$-metric space'' (see Definition~\ref{defn:building} and Remark~\ref{rem:another} for more details). 

Buildings arise naturally in connection with groups originating in Lie theory, and the axioms satisfied by $\delta$ are in essence capturing the combinatorics of the ``Bruhat decomposition'' in these groups. More precisely, if $G$ is a group with a \textit{Tits system} $(B,N,W,S)$ then setting $\Delta=G/B$ and $\delta(gB,hB)=w$ if and only if $g^{-1}h\in BwB$ produces a building (see Section~\ref{sec:buildingsandgroups} for details). While this connection to group theory is the raison d'\^{e}tre for buildings, there are many buildings that are not associated in any nice way to groups (see, for example, \cite{Ron:86}). This motivates the philosophy of treating the building, rather than the group, as the primary object of interest. 

We will give an introduction to the theory of buildings in the first two sections of the paper, with Section~\ref{sec:1} devoted to the theory of Coxeter groups, a necessary prerequisite to the theory of buildings. Section~\ref{sec:2} is devoted to the buildings themselves and to the related group theoretic notion of a Tits system in a group. We will focus on the classes of buildings on which random walks have been studied, including:
\begin{enumerate}
\item[(1)] The \textit{spherical buildings}, where $W$ is a finite reflection group. By Tits' classification~\cite{Tit:74} these buildings are closely related to groups of Lie type such as $SL_n(\FF)$ where $\FF$ is a field.
\item[(2)] The \textit{affine buildings}, where $W$ is an affine reflection group. By the Tits-Weiss classification~\cite{Tit:86,Wei:09} these buildings are closely related to groups of Lie type defined over fields with discrete valuation, such as $SL_n(\QQ_p)$ or $SL_n(\KK(\!(t)\!))$. The simplest affine buildings are trees with no leaves, for example, homogeneous trees. 
\item[(3)] The \textit{Fuchsian buildings}, where $W$ is generated by reflections in the hyperbolic disc~$\HH^2$. These buildings do not admit a classification (see Section~\ref{sec:classification}), however some of them are related to certain ``Kac-Moody groups''. 
\end{enumerate}

In Section~\ref{sec:3} we survey results on random walks on buildings and associated groups. There are various types of random walks that we will consider. A particularly neat class consists of the \textit{isotropic random walks} on the chambers of a building. These are the random walks $(X_n)_{n\geq 0}$ on the set $\Delta$ of chambers such that the transition probabilities $p(x,y)$ depend only on the $W$-distance $\delta(x,y)$. These random walks arise naturally from bi-$B$-invariant probability measures on groups admitting Tits systems, and any limit theorems established for the random walks on the building imply limit theorems for these measures. 

In Section~\ref{sec:hecke} we outline the beautiful algebraic theory of isotropic random walks on~$\Delta$. Put briefly, the transition operator of an isotropic random walk is an element of an algebra called a \textit{Hecke algebra}. These algebras have been extensively studied, (largely due to their connections with groups of Lie type and $p$-adic Lie groups), and their representation theory plays a key role in the theory of random walks on buildings. 

In Section~\ref{sec:spherical} we specialise to the case of finite spherical buildings. In this case the Hecke algebra is finite dimensional, and we give an overview of how the representation theory of this algebra can be applied to investigate isotropic random walks. In particular we provide tractable upper bounds for mixing times for isotropic walks on finite spherical buildings. The analysis follows, in spirit, the work of Diaconis and Ram~\cite{DR:00} where the representation theory of finite dimensional Hecke algebras is applied to investigate the systematic scan Metropolis algorithm. It turns out that this theory is related to random walks on spherical buildings, and so Section~\ref{sec:spherical} is really a translation of~\cite{DR:00} into the language of buildings. Other works related to random walks on spherical buildings can be found in Brown~\cite{Bro:00,Bro:04}, and Brown and Diaconis~\cite{BD:98}.

Next we consider random walks on affine buildings. In this context it is also natural to consider random walks on the `vertices' of the building (these walks arise from bi-$K$-invariant measures on $p$-adic Lie groups, where $K$ is a maximal compact subgroup). We will survey results on these random walks, and random walks on associated groups, drawing from the works of Cartwright and Woess~\cite{CW:04}, Lindlbauer and Voit~\cite{LV:02}, Parkinson~\cite{Par:07}, Parkinson and Schapira~\cite{PS:11}, Parkinson and Woess~\cite{PW:14}, Schapira~\cite{Sch:09}, Tolli~\cite{Tol:01}, and Trojan~\cite{Tro:13}. These works include precise limit theorems for isotropic random walks on the vertices and chambers of affine buildings, as well as theorems for random walks on groups associated to these buildings. Hecke algebras again play a key role in the analysis. Homogeneous trees are the simplest examples of affine buildings (the ``rank~$2$ case''), and in this direction we mention the fundamental works of Cartwright, Kaimanovich and Woess~\cite{CKW:94}, Lalley~\cite{Lal:93}, and Sawyer~\cite{Saw:78}. The literature relating to probability theory and harmonic analysis on homogeneous trees is extensive, and here we will focus on the higher rank cases. 

The study of random walks on non-spherical, non-affine buildings is very open territory. In Section~\ref{sec:fuchsian} we survey recent results of Gilch, M\"{u}ller and Parkinson~\cite{GMP:14} concerning isotropic  random walks on Fuchsian buildings. In this context a law of large numbers and a central limit theorem are available, with interesting formulae for the speed and variance in terms of an underlying automatic structure related to the building. 

We conclude our survey by listing some future directions in the theory, and providing some appendices. In the first appendix we carry through a `by-hand' computation outlining the general theory of isotropic random walks on the vertices of affine buildings in the special case of $\widetilde{C}_2$ buildings. In this basic case we can minimise some of the heavy (although beautiful) machinery used for the general case, thus making the analysis more accessible. In the second appendix we outline the representation theory of rank~$2$ spherical Hecke algebras, and show how a precise knowledge of the representation theory allows for accurate mixing time estimates for random walks on generalised polygons (that is, rank~$2$ spherical buildings). As a byproduct we recover a proof of the celebrated Feit-Higman Theorem (this approach is due to Kilmoyer and Solomon~\cite{KS:73}, and is in turn an adaptation of Feit and Higman's original proof from 1964~\cite{FH:64}).

On a personal note, it is an absolute pleasure to dedicate this paper to my friend and collaborator Wolfgang Woess on the occasion of his 60th birthday. Wolfgang's tireless support of young mathematicians has been a true gift to the mathematical community, a gift from which I have greatly benefited. 

\section{Coxeter systems}\label{sec:1}

\textit{Coxeter systems} form the backbone of the higher objects of buildings and groups of Lie type. In this section we recall some basic theory of Coxeter systems, focussing on examples and important classes. Standard references include \cite{AB:08,Bou:02,Hum:90}.

\subsection{Definitions}

\begin{defn}
A \textit{Coxeter system} $(W,S)$ is a group $W$ generated by a finite set~$S$ with relations
\begin{align*}
s^2=1\quad\textrm{and}\quad (st)^{m_{st}}=1\quad\textrm{for all $s,t\in S$ with $s\neq t$},
\end{align*}
where $m_{st}=m_{ts}\in\ZZ_{\geq 2}\cup\{\infty\}$ for all $s\neq t$ (if $m_{st}=\infty$ then it is understood that there is no relation between $s$ and $t$). We sometimes say that $W$ is a \textit{Coxeter group} when the generating set~$S$ is implied.
\end{defn}

Let $(W,S)$ be a Coxeter system. The \textit{rank} of $(W,S)$ is $|S|$. The \textit{length} of $w\in W$ is
$$
\ell(w)=\min\{n\geq 0\mid w=s_1\cdots s_n\textrm{ with }s_1,\ldots,s_n\in S\},
$$
and an expression $w=s_1\cdots s_n$ with $n$ minimal (that is, $n=\ell(w)$) is called a \textit{reduced expression} for~$w$. It is useful to note that if $w\in W$ and $s\in S$ then $\ell(ws)=\ell(w)\pm 1$. In particular, $\ell(ws)=\ell(w)$ is not possible.

For each $I\subseteq S$ the \textit{standard $I$-parabolic subgroup} of $W$ is the subgroup
$W_I=\langle\{s\mid s\in I\}\rangle$. Then $(W_I,I)$ is a Coxeter system, and hence Coxeter systems `contain' other Coxeter systems of lower rank. This fact is very important in the theory, facilitating inductive arguments on the rank of the group. This makes the rank~$2$ systems particularly important as the base case. The rank~$2$ Coxeter group $W=\langle s,t\mid s^2=t^2=(st)^m=1\rangle$ is just the dihedral group of order $2m$ (or the infinite dihedral group if $m=\infty$), and is denoted by $I_2(m)$.

The data required to define a Coxeter system is conveniently encoded in a graph $\Gamma(W,S)$ with labelled edges called the \textit{Coxeter graph}. This graph has vertex set~$S$, and vertices $s,t\in S$ are joined by an edge if and only if $m_{st}\geq 3$. If $m_{st}\geq 4$ then the corresponding edge is given the label $m_{st}$ (thus edges with no label have $m_{st}=3$, and if $s$ and $t$ are not joined by an edge then $m_{st}=2$). A Coxeter system $(W,S)$ is called \textit{irreducible} if the Coxeter graph~$\Gamma(W,S)$ is connected. Note that if $\Gamma(W,S)$ is not connected, and if $S=S_1\cup\cdots \cup S_k$ is the decomposition into connected components, then $W$ is the direct product of parabolic subgroups $W=W_{S_1}\times\cdots \times W_{S_k}$. Thus irreducibility is a natural assumption to make in the theory of Coxeter systems.

\subsection{The Coxeter complex and examples}\label{sec:cox}

The Coxeter complex of a Coxeter system is a natural simplicial complex on which the Coxeter group acts, and plays an important role in the general theory.

Recall that a \textit{simplicial complex} with vertex set~$V$ is a collection $\Sigma$ of finite subsets of $V$ (called \textit{simplices}) such that for every $v\in V$, the singleton $\{v\}$ is a simplex (called a \textit{vertex}), and every subset of a simplex $\sigma$ is a simplex (a \textit{face} of $\sigma$). If $\sigma$ is a simplex which is not a proper subset of any other simplex then $\sigma$ is a \textit{chamber} of~$\Sigma$. 

Let $\Sigma$ can be simplicial complex, and let $\leq$ be the face relation (that is, $\sigma'\leq \sigma$ if and only if $\sigma'$ is a face of $\sigma$). Then $(\Sigma,\leq)$ is a partially ordered set satisfying:
\begin{enumerate}
\item[(P1)] For each pair $\sigma,\sigma'\in\Sigma$ there exists a greatest lower bound $\sigma\cap\sigma'$.  
\item[(P2)] For each $\sigma\in \Sigma$ the poset $\{\sigma'\mid \sigma'\leq \sigma\}$ is isomorphic to the poset of subsets of $\{1,2,\ldots,r\}$ for some~$r$.
\end{enumerate}
On the other hand, \textit{any} partially ordered set $(\Sigma,\leq)$ satisfying (P1) and~(P2) can be identified with a simplicial complex $\Sigma$ by taking the vertex to be the set $V$ of all elements $v\in\Sigma$ such that $r=1$ in~(P2), and identifying each element $\sigma\in\Sigma$ with the simplex $\{v\in V\mid v\leq \sigma\}$.

\begin{defn}
Let $(W,S)$ be a Coxeter system. The \textit{Coxeter complex} $\Sigma(W,S)$ is the simplicial complex constructed as above from the poset of all cosets of the form $wW_I$ with $w\in W$ and $I\subseteq S$, ordered by reverse inclusion (we emphasise the reverse inclusion here: $wW_I\leq vW_J$ if and only if $wW_I\supseteq vW_J$). 
\end{defn}

Explicitly, the vertex set of the simplicial complex $\Sigma(W,S)$ is
$$
V=\{wW_{S\backslash\{s\}}\mid w\in W,s\in S\},
$$
and $c_0=\{W_{S\backslash\{s\}}\mid s\in S\}$ is a chamber. The set of all chambers is $\{wc_0\mid w\in W\}$. We have that $wc_0=vc_0$ if and only if $w=v$, and so the set of all chambers can be identified with~$W$ by $wc_0\leftrightarrow w$. Each chamber has exactly $|S|$ vertices (namely, the chamber $w$ has vertices $wW_{S\backslash\{s\}}$ for $s\in S$). The Coxeter complex comes equipped with a natural \textit{type function} $\tau:\Sigma(W,S)\to 2^S$ given by $\tau(wW_I)=S\backslash I$. Thus the vertex $wW_{S\backslash\{s\}}$ has type $s$ (more accurately, type $\{s\}$), and each chamber has exactly one vertex of each type. 

\noindent\begin{minipage}[t]{0.6\linewidth}
\vspace{0pt}
\begin{example}
Let $(W,S)$ be the dihedral group of order~$6$ with $S=\{s,t\}$. Write $W_s=W_{\{s\}}=\{1,s\}$, and similarly for $W_t$. The Coxeter complex $\Sigma(W,S)$ has six vertices, marked in the diagram to the right by $\bullet$ (vertices of type~$s$) and $\circ$ (vertices of type $t$). Each chamber has $|S|=2$ vertices, and thus chambers are represented as edges in the diagram.  Similarly, the Coxeter complex of a dihedral group of order~$2m$ is a $2m$-gon, and the Coxeter complex of the infinite dihedral group is a two sided infinite path with alternating vertex types. 
\end{example}
\end{minipage}\hfill
    \begin{minipage}[t]{0.4\linewidth}
\vspace{0pt}
\begin{center}
\begin{tikzpicture} [scale=1.9]
\path 
({cos(0*180/3)},{sin(0*180/3)}) node (0)[shape=circle,draw,fill=black,scale=0.5]  {}
({cos(180/3)},{sin(180/3)}) node (1) [shape=circle,draw,scale=0.5] {}
 ({cos(2*180/3)},{sin(2*180/3)}) node (2) [shape=circle,draw,fill=black,scale=0.5] {}
  ({cos(3*180/3)},{sin(3*180/3)}) node (3) [shape=circle,draw,scale=0.5] {}
   ({cos(4*180/3)},{sin(4*180/3)}) node (4) [shape=circle,draw,fill=black,scale=0.5] {}
    ({cos(5*180/3)},{sin(5*180/3)}) node (5) [shape=circle,draw,scale=0.5] {};
\draw (0) -- (1) -- (2) --(3)--(4)--(5) --(0);
\node [above] at (1) {$W_s$};
\node [above] at (2) {$W_t$};
\node [left] at (3) {$tW_s$};
\node [below] at (4) {$tsW_t$};
\node [below] at (5) {$stW_s$};
\node [right] at (0) {$sW_t$};
\path 
({0.866*cos(30+0*180/3)},{0.866*sin(30+0*180/3)}) node (6)  {}
({0.866*cos(30+180/3)},{0.866*sin(30+180/3)}) node (7) {}
 ({0.866*cos(30+2*180/3)},{0.866*sin(30+2*180/3)}) node (8) {}
  ({0.866*cos(30+3*180/3)},{0.866*sin(30+3*180/3)}) node (9) {}
   ({0.866*cos(30+4*180/3)},{0.866*sin(30+4*180/3)}) node (10)  {}
    ({0.866*cos(30+5*180/3)},{0.866*sin(30+5*180/3)}) node (11) {};
    \node [above] at (7) {$1$};
    \node [left] at (8) {$t$};
    \node [left] at (9) {$ts$};
    \node [above] at (10) {$sts$};
    \node [right] at (11) {$st$};
    \node [right] at (6) {$s$};
         \end{tikzpicture}
         \end{center}
\end{minipage}

\bigskip

A Coxeter system is called:
\begin{enumerate}
\item[(1)] \textit{spherical} if $|W|<\infty$,
\item[(2)] \textit{affine} if $W$ is infinite and contains a normal abelian subgroup~$Q$ such that $W/Q$ is finite,
\item[(3)] \textit{Fuchsian} if $W$ is generated by the reflections in the sides of a polygon in~$\HH^2$.
\end{enumerate}


If $(W,S)$ is an irreducible spherical Coxeter system then $W$ can be realised as a group generated by linear reflections in $E=\RR^{|S|}$. The action of $W$ decomposes $E$ into $|W|$ geometric cones based at the origin, and by intersecting these cones with the unit sphere we can visualise the Coxeter complex $\Sigma(W,S)$ as a tessellation of the $(|S|-1)$-sphere (some examples are illustrated in Figure~\ref{fig:rank3Coxeter}). There is a well known classification of the irreducible spherical Coxeter systems due to Coxeter \cite{Cox:35}. The nomenclature of this classification has its origins in the Cartan-Killing classification of simple Lie algebras over~$\CC$. The list of spherical Coxeter systems is as follows (in each case the subscript denotes the rank of the system, see Figure~\ref{fig:sphericalclassification} in Appendix~\ref{app:classification} for the Coxeter graphs).
\begin{enumerate}
\item[(1)] \textit{Crystallographic} systems: $A_n$ $(n\geq 1)$, $B_n=C_n$ ($n\geq 2$), $D_n$ ($n\geq 4$), $E_6$, $E_7$, $E_8$, $F_4$, $G_2$.
\item[(2)] \textit{Non-crystallographic} systems: $H_3$, $H_4$, $I_2(m)$ (with $m=5$ or $m\geq 7$).
\end{enumerate}

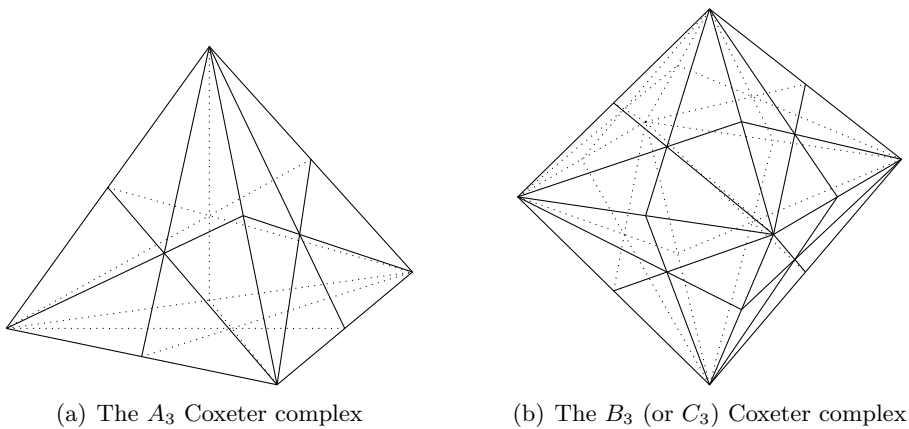
\begin{figure}[!h]
\centering
\subfigure[The $A_3$ Coxeter complex]{
\begin{tikzpicture} [xscale=1.8,yscale=1.5]
\coordinate (A) at (-2,0.5);
\coordinate (B) at (0,0);
\coordinate (C) at (1,1);
\coordinate (D) at (-0.5,3);
\coordinate (BD) at (-0.25,1.5);
\coordinate (AD) at (-1.25,1.75);
\coordinate (CD) at (0.25,2);
\coordinate (BC) at (0.5,0.5);
\coordinate (AB) at (-1,0.25);
\coordinate (AC) at (-0.5,0.75);
\draw (A)--(B)--(C)--(D)--(A);
\draw (B)--(D);
\draw [style=dotted] (A)--(C);
\draw (A)--(BD);
\draw (BD)--(C);
\draw (AD)--(B);
\draw (CD)--(B); 
\draw (BC)--(D);
\draw (AB)--(D);
\draw [style=dotted] (AC)--(D);
\draw [style=dotted] (CD)--(A);
\draw [style=dotted] (BC)--(A);
\draw [style=dotted] (AC)--(B);
\draw [style=dotted] (AB)--(C);
\draw [style=dotted] (AD)--(C);
\end{tikzpicture}
}\hspace{1cm}
\subfigure[The $B_3$ (or $C_3$) Coxeter complex]{
\begin{tikzpicture} [xscale=1.7,yscale=1]
\coordinate (A) at (-2,0.5);
\coordinate (B) at (0,0);
\coordinate (C) at (1,1);
\coordinate (D) at (-0.5,3);
\coordinate (BD) at (-0.25,1.5);
\coordinate (AD) at (-1.25,1.75);
\coordinate (CD) at (0.25,2);
\coordinate (BC) at (0.5,0.5);
\coordinate (AB) at (-1,0.25);
\coordinate (AC) at (-0.5,0.75);
\coordinate (E) at (-1,1.5);
\coordinate (F) at (-0.5,-2);
\coordinate (AE) at (-1.5,1); 
\coordinate (BF) at (-0.25,-1);
\coordinate (AF) at (-1.25,-0.75);
\coordinate (CF) at (0.25,-0.5);
\coordinate (DE) at (-0.75,2.25);
\coordinate (CE) at (0,1.25);
\coordinate (EF) at (-0.75,-0.25);
\draw (A)--(B)--(C)--(D)--(A);
\draw (B)--(D);
\draw [style=dotted] (A)--(E);
\draw (A)--(BD);
\draw (BD)--(C);
\draw (AD)--(B);
\draw (CD)--(B); 
\draw (BC)--(D);
\draw (AB)--(D);
\draw [style=dotted] (AC)--(B);
\draw [style=dotted] (E)--(C);
\draw [style=dotted] (E)--(D);
\draw (A)--(F)--(B);
\draw (F)--(C);
\draw [style=dotted] (AE)--(D);
\draw [style=dotted] (F)--(E);
\draw (B)--(AF);
\draw (F)--(AB);
\draw (A)--(BF);
\draw (C)--(BF);
\draw (F)--(BC);
\draw (B)--(CF);
\draw [style=dotted] (E)--(AD);
\draw [style=dotted] (A)--(DE);
\draw [style=dotted] (D)--(CE);
\draw [style=dotted] (F)--(CE);
\draw [style=dotted] (E)--(CD);
\draw [style=dotted] (C)--(DE);
\draw [style=dotted] (C)--(EF);
\draw [style=dotted] (E)--(CF);
\draw [style=dotted] (A)--(EF);
\draw [style=dotted] (F)--(AE);
\draw [style=dotted] (E)--(AF);
\end{tikzpicture}
}
\caption{Examples of rank~$3$ spherical Coxeter systems}
\label{fig:rank3Coxeter}
\end{figure}


If $(W,S)$ is an irreducible affine Coxeter system then $W$ can be realised as a group generated by affine reflections in~$E=\RR^{|S|-1}$ (see Section~\ref{sec:root}). The action of $W$ decomposes $E$ into geometric simplices, and so the Coxeter complex $\Sigma(W,S)$ may be visualied as a tessellation of $\RR^{|S|-1}$ (some examples are illustrated in Figure~\ref{fig:affinerank3}, the additional information in the figure will be explained in Section~\ref{sec:root}). The classification of irreducible affine Coxeter systems is closely related to the classification of irreducible spherical Coxeter systems. Specifically, to each irreducible crystallographic spherical Coxeter system (of type $X_n$, say) there is an associated affine Coxeter system of type $\widetilde{X}_n$ obtained by adding one additional generator to the spherical system. In the case of spherical systems of type $B_n=C_n$ there are two associated affine systems, called $\widetilde{B}_n$ and $\widetilde{C}_n$, and these are non-isomorphic if $n>2$. See Figure~\ref{fig:affineclassification} in Appendix~\ref{app:classification} for the Coxeter graphs of the irreducible affine Coxeter systems.

\begin{figure}[!h]
\centering
\subfigure[$\widetilde{A}_2$ Coxeter complex]{
\begin{tikzpicture}[scale=0.78]
\path [fill=lightgray!70] (0,0) -- (2.2,3.81) -- (-2.2,3.81) -- (0,0);
\path [fill=gray!90] (0,0) -- (-0.5,0.866) -- (0.5,0.866) -- (0,0);
    \draw (2.8, -3.81)--( 3.7, -2.2516);
    \draw (1.8, -3.81)--( 3.7, -0.52 );
    \draw (0.8, -3.81)--( 3.7, 1.212 );
    \draw (-0.2, -3.81)--( 3.7, 3.044);
    \draw (-1.2, -3.81)--( 3.2, 3.81);
    \draw (-2.2, -3.81)--( 2.2, 3.81);
    \draw (-3.2, -3.81)--( 1.2, 3.81);
    \draw (-3.7, -3.044)--( 0.2, 3.81);
    \draw (-3.7, -1.212)--( -0.8, 3.81 );
      \draw (-3.7, 0.520)--( -1.8, 3.81);
    \draw (-3.7, 2.2516)--( -2.8, 3.81);
    \draw (-2.8, -3.81)--( -3.7, -2.2516);
    \draw (-1.8, -3.81)--( -3.7, -0.52 );
    \draw (-0.8, -3.81)--( -3.7, 1.212);
    \draw (0.2, -3.81)--( -3.7, 3.044 );
    \draw (1.2, -3.81)--( -3.2, 3.81);
    \draw (2.2, -3.81)--( -2.2, 3.81);
    \draw  (3.2, -3.81)--( -1.2, 3.81 );
    \draw (3.7, -3.044)--( -0.2, 3.81 );
    \draw (3.7, -1.212)--( 0.8, 3.81);
    \draw (3.7, 0.520)--( 1.8, 3.81);
    \draw (3.7, 2.2516)-- (2.8, 3.81);
    \draw (-3.7, -3.464)--( 3.7, -3.464);
    \draw (-3.7, -2.598)--( 3.7, -2.598);
    \draw (-3.7, -1.732)--( 3.7, -1.732);
    \draw (-3.7, -0.866)--( 3.7, -0.866);
    \draw (-3.7, 0)--( 3.7, 0);
    \draw (-3.7, 3.464)--( 3.7, 3.464 );
    \draw (-3.7, 2.598)--( 3.7, 2.598);
    \draw (-3.7, 1.732)--( 3.7, 1.732);
    \draw (-3.7, 0.866)--( 3.7, 0.866);
     \node at (0,0) {$\bullet$};
    \node at (3,0) {$\bullet$};
     \node at (-3,0) {$\bullet$};
    \node at (0,1.732) {$\bullet$};
    \node at (3,1.732) {$\bullet$};
     \node at (-3,1.732) {$\bullet$};
     \node at (-1.5,2.598) {$\bullet$};
     \node at (1.5,2.598) {$\bullet$};
     \node at (-1.5,0.866) {$\bullet$};
      \node at (1.5,0.866) {$\bullet$};
      \node at (0,3.464) {$\bullet$};
    \node at (3,3.464) {$\bullet$};
     \node at (-3,3.464) {$\bullet$};
   \node at (0,-1.732) {$\bullet$};
    \node at (3,-1.732) {$\bullet$};
     \node at (-3,-1.732) {$\bullet$};
     \node at (-1.5,-2.598) {$\bullet$};
     \node at (1.5,-2.598) {$\bullet$};
     \node at (-1.5,-0.866) {$\bullet$};
      \node at (1.5,-0.866) {$\bullet$};
      \node at (0,-3.464) {$\bullet$};
    \node at (3,-3.464) {$\bullet$};
     \node at (-3,-3.464) {$\bullet$};
     %
%
    \draw [latex-latex, line width=1pt] (-1.5,-0.866)--(1.5,0.866);
    \draw [latex-latex, line width=1pt] (1.5,-0.866)--(-1.5,0.866);
    \draw [latex-latex, line width=1pt] (0,-1.732)--(0,1.732);
    \draw [-latex, line width=1pt] (0,0) -- (0.5,0.866);
    \draw [-latex,line width=1pt] (0,0)--(-0.5,0.866);
    \node at (-2,1.1) {\small{$\alpha_1^{\vee}$}};
    \node at (2,1.1) {\small{$\alpha_2^{\vee}$}};
    \node at (-0.9,1.1) {\small{$\omega_1$}};
    \node at (1,1.1) {\small{$\omega_2$}};
\end{tikzpicture}
}\hspace{1cm}
\subfigure[$\widetilde{C}_2$ Coxeter complex]{
\begin{tikzpicture}[scale=0.7]
\path [fill=lightgray!70] (0,0) -- (4.25,0) -- (4.25,4.25) -- (0,0);
\path [fill=gray!90] (0,0) -- (1,0) -- (1,1) -- (0,0);
\draw (-4.25,-2) -- (4.25,-2); 
\draw (-4.25,-1) -- (4.25,-1);
\draw (-4.25,0) -- (4.25,0);
\draw (-4.25,1) -- (4.25,1);
\draw (-4.25,2) -- (4.25,2);
\draw (-4.25,3) -- (4.25,3); 
\draw (-4.25,4) -- (4.25,4);
\draw (-4.25,-3) -- (4.25,-3);
\draw (-4.25,-4) -- (4.25,-4);
\draw (-2,-4.25) -- (-2,4.25);
\draw (-1,-4.25) -- (-1,4.25);
\draw (0,-4.25) -- (0,4.25);
\draw (1,-4.25) -- (1,4.25);
\draw (2,-4.25) -- (2,4.25);
\draw (-4,-4.25) -- (-4,4.25);
\draw (-3,-4.25) -- (-3,4.25);
\draw (3,-4.25) -- (3,4.25);
\draw (4,-4.25) -- (4,4.25);
\draw (-4.25,3.75)--(-3.75,4.25);
\draw (-4.25,1.75)--(-1.75,4.25);
\draw (-4.25,-0.25) -- (0.25,4.25);
\draw (-4.25,-2.25) -- (2.25,4.25);
\draw (-4.25,-4.25) -- (4.25,4.25);
\draw (-2.25,-4.25) -- (4.25,2.25);
\draw (-0.25,-4.25) -- (4.25,0.25);
\draw (1.75,-4.25)--(4.25,-1.75);
\draw (3.75,-4.25)--(4.25,-3.75);
\draw (4.25,3.75)--(3.75,4.25);
\draw (4.25,1.75)--(1.75,4.25);
\draw (4.25,-0.25) -- (-0.25,4.25);
\draw (4.25,-2.25) -- (-2.25,4.25);
\draw (4.25,-4.25) -- (-4.25,4.25);
\draw (2.25,-4.25) -- (-4.25,2.25);
\draw (0.25,-4.25) -- (-4.25,0.25);
\draw (-1.75,-4.25)--(-4.25,-1.75);
\draw (-3.75,-4.25)--(-4.25,-3.75);
\node at (0,0) {$\bullet$};
\node at (-2,0) {$\bullet$};
\node at (2,0) {$\bullet$};
\node at (-2,-2) {$\bullet$};
\node at (0,-2) {$\bullet$};
\node at (2,-2) {$\bullet$};
\node at (0,2) {$\bullet$};
\node at (-2,2) {$\bullet$};
\node at (2,2) {$\bullet$};
\node at (-4,-4) {$\bullet$};
\node at (-4,-2) {$\bullet$};
\node at (-4,0) {$\bullet$};
\node at (-4,2) {$\bullet$};
\node at (-4,4) {$\bullet$};
\node at (-2,-4) {$\bullet$};
\node at (-2,4) {$\bullet$};
\node at (0,-4) {$\bullet$};
\node at (0,4) {$\bullet$};
\node at (2,-4) {$\bullet$};
\node at (2,4) {$\bullet$};
\node at (4,-4) {$\bullet$};
\node at (4,-2) {$\bullet$};
\node at (4,0) {$\bullet$};
\node at (4,2) {$\bullet$};
\node at (4,4) {$\bullet$};
%
%
%
%
\node at (2.7,-1.7) {\small{$\alpha_1^{\vee}$}};
\node at (0.7,2.3) {\small{$\alpha_2^{\vee}$}};
\node at (2.7,0.2) {\small{$\omega_1$}};
\node at (1.6,1.15) {\small{$\omega_2$}};
\draw [latex-latex,line width=1pt] (0,-2)--(0,2);
\draw [latex-latex,line width=1pt] (-2,0)--(2,0);
\draw [latex-latex,line width=1pt] (-2,-2)--(2,2);
\draw [latex-latex,line width=1pt] (-2,2)--(2,-2);
\draw [-latex, line width=1pt] (0,0)--(1,1);
\end{tikzpicture}
}
\caption{Rank~$3$ affine Coxeter systems, and associated (dual) root systems}\label{fig:affinerank3}
\end{figure}
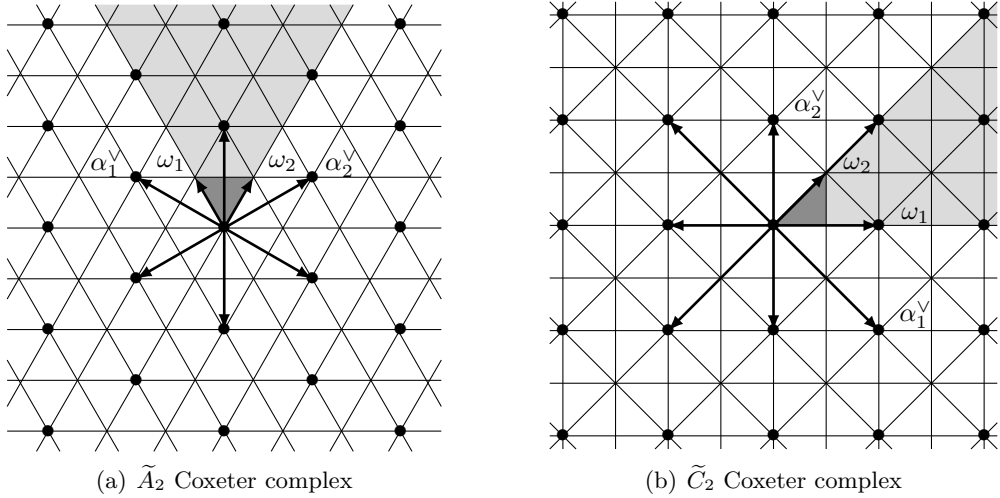


If $(W,S)$ is a Fuchsian Coxeter system then $W$ may be realised as a group generated by reflections in the sides of a polygon in the hyperbolic disc~$\HH^2$, and thus $W$ is a cocompact discrete subgroup of $PGL_2(\RR)$ (hence the term `Fuchsian'). Specifically, let $n\geq 3$ be an integer, and let $k_1,\ldots,k_n\geq 2$ be integers satisfying 
\begin{align}\label{eq:hyperbolic}
\sum_{i=1}^n\frac{1}{k_i}<n-2.
\end{align}
Assign the angles $\pi/k_i$ to the vertices of a combinatorial $n$-gon~$F$. There is a convex realisation of $F$ (which we also call~$F$) in the hyperbolic disc~$\mathbb{H}^2$. Let $W$ be the subgroup of $PGL_2(\mathbb{R})$ generated by the set $S$ of reflections in the sides of~$F$. Then $(W,S)$ is a Coxeter system (see \cite[Example~6.5.3]{Dav:08}), and if $s_1,\ldots,s_{n}$ are the reflections in the sides of~$F$ (arranged cyclically), then the order of $s_is_j$ is 
\begin{align}\label{eq:hyperbolic2}
\begin{aligned}
m_{ij}=\begin{cases}k_i&\textrm{if $j=i+1$}\\
\infty&\textrm{if $|i-j|>1$},
\end{cases}
\end{aligned}
\end{align}
where the indices are read cyclically. We denote this Coxeter system by
$
F(k_1,\ldots,k_n).
$ The group $W$ acts on $\mathbb{H}^2$ with fundamental domain~$F$, and thus induces a tessellation of $\mathbb{H}^2$ by isometric polygons $wF$, $w\in W$. Examples are shown in Figure~\ref{fig:triangle}. We note that this is \textit{not} a depiction of the Coxeter complex of these groups unless $|S|=3$. For example, each chamber of the Coxeter complex of the group represented in Figure~\ref{fig:triangle}(b) is a $|S|-1=5$ dimensional simplex. Instead the pictures are (essentially) the \textit{Davis complex} of the group (see \cite[Example~12.43]{AB:08}). We will not go into further details, however we simply remark that this is a much more convenient way to visualise Fuchsian Coxeter systems (and their buildings). 

\begin{figure}[!h]
\centering
\subfigure[Fuchsian Coxeter system $F(3,3,4)$]{
\includegraphics[totalheight=6cm]{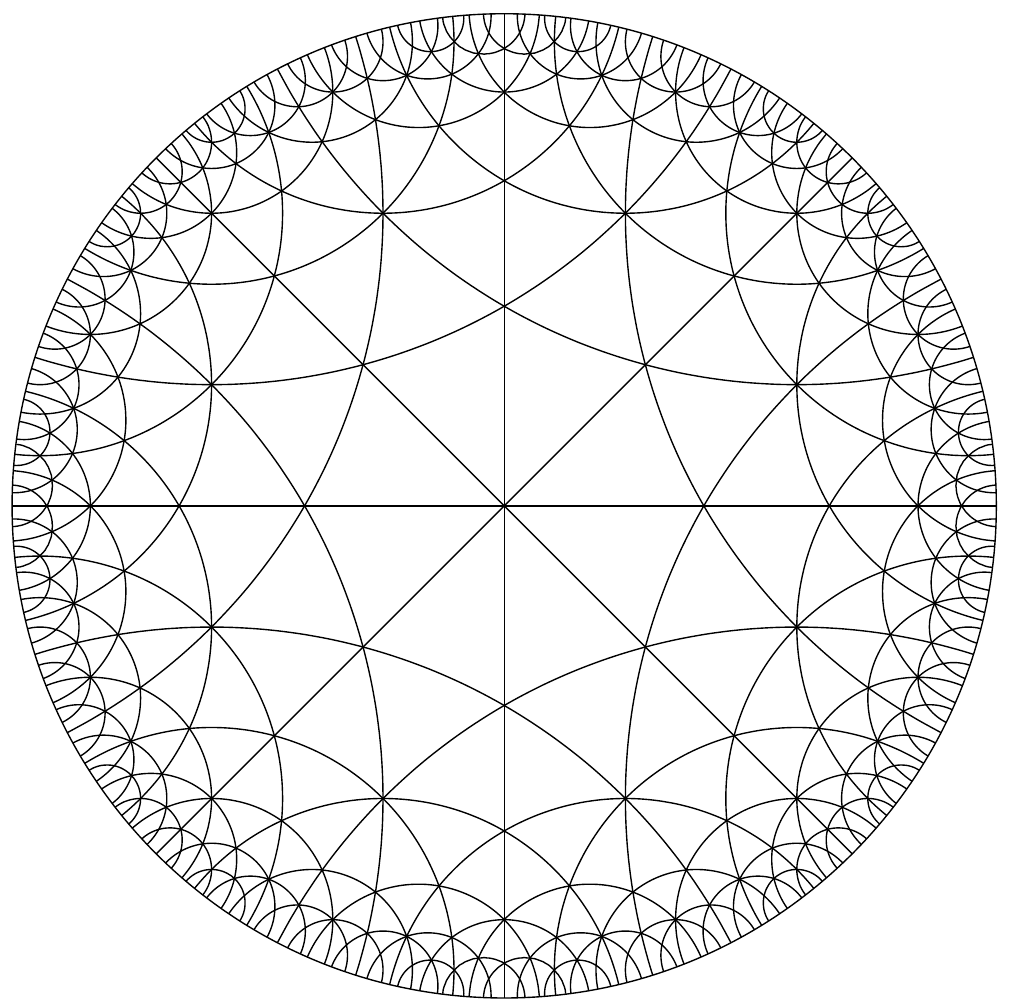}}\hspace{1.5cm}
\subfigure[\mbox{Fuchsian Coxeter system $F(2,2,2,2,2,2)$}]{
\includegraphics[totalheight=6cm]{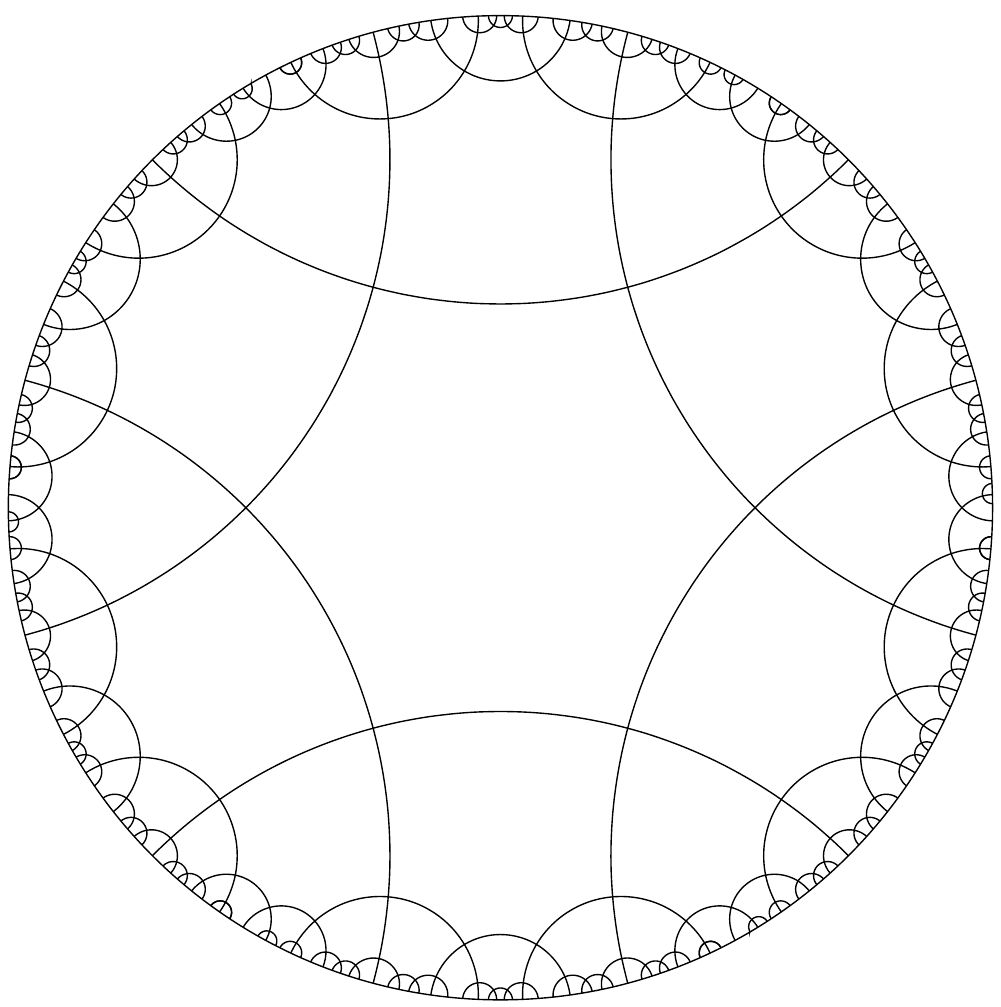}}
\caption{Fuchsian Coxeter systems}\label{fig:triangle}
\end{figure}

\subsection{The Coxeter complex of an affine Coxeter system}\label{sec:root}

For our later discussions it is necessary to have a concrete description of the Coxeter complex of an affine Coxeter system in terms of \textit{root systems}. The standard reference for this theory is \cite{Bou:02}. Let $E$ be a $d$-dimensional real vector space with inner product $\langle\cdot,\cdot\rangle$. The hyperplane orthogonal to the vector $\alpha\in E\backslash\{0\}$ is $H_{\alpha}=\{x\in E\mid \langle x,\alpha\rangle=0\}$, and the \textit{reflection} in $H_{\alpha}$ is given by
$
s_{\alpha}(x)=x-\langle x,\alpha\rangle\alpha^{\vee}$ where $\alpha^{\vee}=2\alpha/\langle\alpha,\alpha\rangle$.

A \textit{root system} in $E$ is a finite set $R$ of non-zero vectors (called \textit{roots}) such that: (1) $R$ spans~$E$, (2) if $\alpha\in R$ and $k\alpha\in R$ then $k=\pm 1$, (3) if $\alpha,\beta\in R$ then $s_{\alpha}(\beta)\in R$, and (4) if $\alpha,\beta\in R$ then $\langle\alpha,\beta^{\vee}\rangle\in\mathbb{Z}$. The \textit{rank} of $R$ is $d=\dim(E)$. Figure~\ref{fig:roots} illustrates three rank~$2$ root systems.

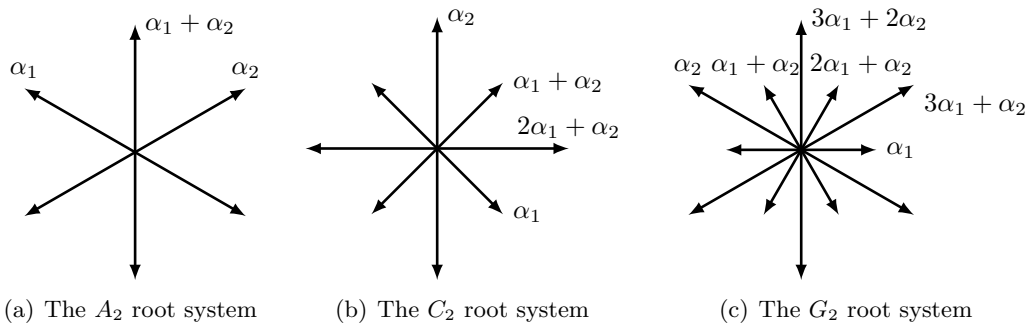
\begin{figure}[!h]
\centering
\subfigure[The $A_2$ root system]{
\begin{tikzpicture}[scale=1.7]
\draw[latex-latex,line width=1pt] (0,-1)--(0,1);
\draw [latex-latex,line width=1pt] (-0.866,0.5)--(0.866,-0.5);
\draw  [latex-latex,line width=1pt] (0.866,0.5)--(-0.866,-0.5);
\node [above] at (-0.866,0.5) {\small{$\alpha_1$}};
\node [above] at (0.866,0.5) {\small{$\alpha_2$}};
\node [right] at (0,1) {\small{$\alpha_1+\alpha_2$}};
\end{tikzpicture}}\hspace{0.2cm}
\subfigure[The $C_2$ root system]{
\begin{tikzpicture}[scale=1.75]
\draw [latex-latex,line width=1pt] (0,-1)--(0,1);
\draw [latex-latex,line width=1pt] (-0.5,-0.5)--(0.5,0.5);
\draw [latex-latex,line width=1pt] (0.5,-0.5)--(-0.5,0.5);
\draw [latex-latex,line width=1pt] (-1,0)--(1,0);
\node [right] at (0.5,-0.5) {\small{$\alpha_1$}};
\node [right] at (0,1) {\small{$\alpha_2$}};
\node [right] at (0.5,0.5) {\small{$\alpha_1+\alpha_2$}};
\node [above] at (1,0) {\small{$2\alpha_1+\alpha_2$}};
\end{tikzpicture}}\hspace{0.2cm}
\subfigure[The $G_2$ root system]{
\begin{tikzpicture}[scale=1]
\draw [latex-latex,line width=1pt] (-0.5,-0.866)--(0.5,0.866);
\draw [latex-latex,line width=1pt] (0.5,-0.866)--(-0.5,0.866);
\draw [latex-latex,line width=1pt] (-1,0)--(1,0);
\draw [latex-latex,line width=1pt] (0,-1.732)--(0,1.732);
\draw [latex-latex,line width=1pt] [rotate around={60:(0,0)}] (0,-1.732)--(0,1.732);
\draw [latex-latex,line width=1pt] [rotate around={-60:(0,0)}] (0,-1.732)--(0,1.732);
\node [right] at (1,0) {\small{$\alpha_1$}};
\node [above] at (-1.5,0.866) {\small{$\alpha_2$}};
\node [above] at (-0.6,0.866) {\small{$\alpha_1+\alpha_2$}};
\node [above] at (0.8,0.866) {\small{$2\alpha_1+\alpha_2$}};
\node [right] at (1.5,0.6) {\small{$3\alpha_1+\alpha_2$}}; 
\node [right] at (0,1.732) {\small{$3\alpha_1+2\alpha_2$}};
\end{tikzpicture}}
\caption{The irreducible rank~$2$ root systems}\label{fig:roots}
\end{figure}

Let $R$ be a rank~$d$ root system. There exists a subset $\{\alpha_1,\ldots,\alpha_d\}\subseteq R$ of \textit{simple roots} with the property that every $\alpha\in R$ can be written as a linear combination of $\alpha_1,\ldots,\alpha_d$ with integer coefficients which are either all nonpositive, or all nonnegative. Those roots whose coefficients are all nonnegative are called \textit{positive roots} (with respect to the fixed chosen set of simple roots), and the set of all positive roots is denoted~$R^+$. Then $R=R^+\cup(-R^+)$.

The \textit{Weyl group} of $R$ is the finite subgroup $W_0$ of $GL(E)$ generated by the reflections~$s_{\alpha}$ with $\alpha\in R$. For each $i=1,\ldots,d$ write $s_i=s_{\alpha_i}$, and let $S_0=\{s_1,\ldots,s_d\}$. Then $(W_0,S_0)$ is a spherical Coxeter system, with the order of $s_is_j$ being $m_{ij}$, where $\pi-\pi/m_{ij}$ is the angle between $\alpha_i$ and $\alpha_j$. The Coxeter system $(W_0,S_0)$ is irreducible if and only if the root system $R$ is irreducible (where the latter means that there is no partition $R=R_1\cup R_2$ with $R_1$ and $R_2$ nonempty such that $\langle\alpha,\beta\rangle=0$ for all $\alpha\in R_1$ and all $\beta\in R_2$). 

The irreducible root systems admit a complete classification, and explicit descriptions of each system can be found in \cite[Plates I--IV]{Bou:02}. They fall into four infinite families $A_d$ ($d\geq 1$), $B_d$ ($d\geq 2$), $C_d$ ($d\geq 2$) and $D_d$ ($d\geq 4$), and $5$ exceptional types $E_6,E_7,E_8,F_4$ and $G_2$. If $R$ is an irreducible root system of type $X_d$ then the Coxeter system $(W_0,S_0)$ is also of type $X_d$, and hence every irreducible crystallographic spherical Coxeter group can be realised as the Weyl group of an irreducible root system.

For each $\alpha\in R$ and each $k\in\ZZ$ let
$
H_{\alpha,k}=\{x\in E\mid \langle x,\alpha\rangle=k\}. 
$
Thus the affine hyperplane $H_{\alpha,k}$ is a translate of the linear hyperplane~$H_{\alpha}=H_{\alpha,0}$. The orthogonal (affine) reflection in the hyperplane $H_{\alpha,k}$ is given by the formula
$
s_{\alpha,k}(x)=x-(\langle x,\alpha\rangle-k)\alpha^{\vee}$ for $x\in E$, and the \textit{affine Weyl group} of~$R$ is the subgroup $W_{\mathrm{aff}}$ of $\mathrm{Aff}(E)$ generated by all reflections $s_{\alpha,k}$ with $\alpha\in R$ and $k\in\mathbb{Z}$. The root system $R$ has a unique \textit{highest root} $\varphi$ (the \textit{height} of the root $\alpha=a_1\alpha_1+\cdots+a_d\alpha_d$ is $a_1+\cdots+a_d$, and $\varphi$ is the unique root of greatest height). Let $s_0=s_{\varphi,1}$. Then $W_{\mathrm{aff}}$ is generated by $\Saff=\{s_0,s_1,\ldots,s_d\}$ and the order of $s_0s_j$ is $m_{0j}$, where $\pi-\pi/m_{0j}$ is the angle between $-\varphi$ and $\alpha_j$. Thus $(\Waff,\Saff)$ is a Coxeter system, and $(W_0,S_0)$ is a parabolic subsystem of $(W_{\mathrm{aff}},S_{\mathrm{aff}})$. For $\lambda\in E$ let $t_{\lambda}\in\mathrm{Aff}(E)$ be the translation $t_{\lambda}(x)=x+\lambda$. Since $s_{\alpha,k}=t_{k\alpha^{\vee}}s_{\alpha}$ we have 
\begin{align*}
\Waff=Q\rtimes W_0,\quad \text{where $Q=\ZZ\alpha_1^{\vee}+\cdots+\ZZ\alpha_d^{\vee}$ is the \textit{coroot lattice}}.
\end{align*}
Thus $(W_{\mathrm{aff}},S_{\mathrm{aff}})$ is an affine Coxeter system. All irreducible affine Coxeter systems arise in this way.

Let $\omega_1,\ldots,\omega_d\in E$ be the dual basis to $\alpha_1,\ldots,\alpha_d$, given by $\langle\omega_i,\alpha_j\rangle=\delta_{i,j}$. The \textit{coweight lattice} $P$ of $R$ is 
\begin{align*}
P&=\{\lambda\in E\mid \langle\lambda,\alpha\rangle\in\mathbb{Z}\text{ for all $\alpha\in R$}\}=\ZZ\omega_1+\cdots+\ZZ\omega_d.
\end{align*}
Note that $Q\subseteq P$, and that $Q$ and $P$ are both $W_{\mathrm{aff}}$-invariant lattices. The set of \textit{dominant coweights} is
$$
P^+=\NN\omega_1+\cdots+\NN\omega_d.
$$

The family of hyperplanes $H_{\alpha,k}$, $\alpha\in R$, $k\in\ZZ$, tessellates $E$ into $d$-dimensional geometric simplices (the \textit{chambers}). The extreme points of the chambers are \textit{vertices}, and each chamber has exactly $d+1$ vertices. The resulting simplicial complex is isomorphic to the Coxeter complex $\Sigma(W_{\mathrm{aff}},S_{\mathrm{aff}})$. The affine Weyl group $W_{\mathrm{aff}}$ acts simply transitively on the set of all chambers. The \textit{fundamental chamber} is 
$$
\fc_0=\{x\in E\mid \langle x,\alpha_i\rangle\geq 0\text{ for $1\leq i\leq d$, and }\langle x,\varphi\rangle\leq 1\}
$$
and we often identify $W_{\mathrm{aff}}$ with the set of chambers by $w\leftrightarrow w\fc_0$. The set $P$ of coweights is a subset of the set of all vertices of $\Sigma(W_{\mathrm{aff}},S_{\mathrm{aff}})$, called the \textit{special vertices}. The set $Q$ is the set of type~$0$ vertices, and every chamber has exactly one type $0$ vertex. 

The action of $W_0$ decomposes $E$ into $|W_0|$ geometric cones based at the origin, and the translates of these cones by elements of~$P$ are called \textit{sectors}. The \textit{fundamental sector} is 
$$
\fs_0=\{x\in E\mid \langle x,\alpha_i\rangle\geq 0\text{ for all $1\leq i\leq d$}\}.
$$
We also write $E^+=\fs_0$ (roughly speaking, we write $\fs_0$ when we are interested in the simplicial structure, and $E^+$ when we are interested in the metric structure). 

Examples of the above construction of affine Coxeter systems are illustrated in Figure~\ref{fig:affinerank3}, where the sector $\fs_0$ and the chamber $\fc_0$ are shaded (light and dark, respectively). The coroot lattice $Q$ is indicted with heavy dots. In Figure~\ref{fig:affinerank3}(a) all vertices are special, while in Figure~\ref{fig:affinerank3}(b) only the vertices of valency $8$ are special.





\section{Buildings}\label{sec:2}

\textit{Buildings} were introduced by Jacques Tits in the 1950s. ``The origin of the notions of buildings and $BN$-pairs lies in an attempt to give a systematic procedure for geometric interpretation of the semi-simple Lie groups and, in particular, the exceptional groups'' \cite[Introduction]{Tit:74}. Over the past~$60$ years the theory has grown immensely, and has had diverse applications in geometry, group theory, representation theory, and geometric group theory. In this section we give a brief introduction to the theory, with our main references being~\cite{AB:08,Ron:09,Tit:74}.

\subsection{Definitions and basic properties}\label{sec:defn}

Buildings are defined axiomatically, and historically there have been two main approaches to the theory, both due to Jacques Tits. The initial approach was via simplicial complexes, and later an approach was developed using `chamber systems' (see \cite{Tit:81} for an enlightening historical discussion). Both approaches are relevant and useful, however here we have chosen to adopt Tits' original simplicial complex definition (see Remark~\ref{rem:another} for the other approach). 

\begin{defn}\label{defn:building}
Let $(W,S)$ be a Coxeter system with Coxeter complex $\Sigma(W,S)$. A \textit{building of type $(W,S)$} is a nonempty simplicial complex $\Sigma$ with a family $\mathbf{A}$ of subcomplexes (called \textit{apartments}) such that 
\begin{enumerate}
\item[(B1)] each apartment $A\in\mathbf{A}$ is isomorphic to the Coxeter complex $\Sigma(W,S)$,
\item[(B2)] given any two simplices of $\Sigma$ there is an apartment $A\in\mathbf{A}$ containing both of them, and
\item[(B3)] if $A,A'\in\mathbf{A}$ are apartments containing a common chamber then there is a unique simplicial complex isomorphism $\psi:A'\to A$ fixing each simplex of the intersection $A\cap A'$. 
\end{enumerate}
\end{defn}

Let $\Sigma$ be a building of type $(W,S)$. The \textit{rank} of $\Sigma$ is $|S|$. Fix, once and for all, an apartment of $\Sigma$ and identify it with $\Sigma(W,S)$. Thus we regard $\Sigma(W,S)$ as an apartment of $\Sigma$, the ``standard'' (or ``base'') apartment. The type function on the Coxeter complex $\Sigma(W,S)$ extends uniquely to a type function $\tau:\Sigma\to 2^S$ on the building making $\Sigma$ into a labelled simplicial complex. The isomorphism in (B3) is then necessarily type preserving. Let $\Delta$ be the set of all chambers of~$\Sigma$. 

\noindent\begin{minipage}[t]{0.6\linewidth}
\vspace{0pt}
\begin{example}\label{ex:tree}
Buildings of type $\widetilde{A}_1=I_2(\infty)$ are equivalent to trees in which every vertex has valency at least~$2$. The apartments are two sided infinite geodesics in the tree, and the chambers are the edges of the tree. There are two types of vertices, indicated by $\bullet$ and $\circ$ in the picture. Buildings of higher rank are considerably more sophisticated objects, although the tree example is very instructive. 
\end{example}
\end{minipage}\hfill
    \begin{minipage}[t]{0.4\linewidth}
\vspace{0pt}
\begin{center}
\begin{tikzpicture} [xscale=0.25, yscale=0.25]
\path 
(2,0) node (1) [shape=circle,draw,fill=black,scale=0.5]  {}
(5,2) node (2) [shape=circle,draw,scale=0.5] {}
(5,-2) node (3) [shape=circle,draw,scale=0.5] {}
(6,5) node (4) [shape=circle,draw,fill=black,scale=0.5]  {}
(8,3) node (5) [shape=circle,draw,fill=black,scale=0.5]  {}
(8,-3) node (6) [shape=circle,draw,fill=black,scale=0.5]  {}
(6,-5) node (7) [shape=circle,draw,fill=black,scale=0.5]  {}
(5,7) node (8) {}
(7,7) node (9) {}
(10,4) node (10) {}
(10,2) node (11) {}
(10,-2) node (12) {}
(10,-4) node (13) {}
(7,-7) node (14) {}
(5,-7) node (15) {}
(-2,0) node (-1) [shape=circle,draw,scale=0.5]  {}
(-5,2) node (-2) [shape=circle,draw,fill=black,scale=0.5] {}
(-5,-2) node (-3) [shape=circle,draw,fill=black,scale=0.5] {}
(-6,5) node (-4) [shape=circle,draw,scale=0.5]  {}
(-8,3) node (-5) [shape=circle,draw,scale=0.5]  {}
(-8,-3) node (-6) [shape=circle,draw,scale=0.5]  {}
(-6,-5) node (-7) [shape=circle,draw,scale=0.5]  {}
(-5,7) node (-8) {}
(-7,7) node (-9) {}
(-10,4) node (-10) {}
(-10,2) node (-11) {}
(-10,-2) node (-12) {}
(-10,-4) node (-13) {}
(-7,-7) node (-14) {}
(-5,-7) node (-15) {};
\draw (1) -- (2);
\draw (1)--(3);
\draw (2)--(4);
\draw (2)--(5);
\draw (3)--(6);
\draw (3)--(7);
\draw [style=dashed](4)--(8);
\draw [style=dashed](4)--(9);
\draw [style=dashed](5)--(10);
\draw [style=dashed](5)--(11);
\draw [style=dashed](6)--(12);
\draw [style=dashed](6)--(13);
\draw [style=dashed](7)--(14);
\draw [style=dashed](7)--(15);
\draw (-1) -- (-2);
\draw (-1)--(-3);
\draw (-2)--(-4);
\draw (-2)--(-5);
\draw (-3)--(-6);
\draw (-3)--(-7);
\draw [style=dashed](-4)--(-8);
\draw [style=dashed](-4)--(-9);
\draw [style=dashed](-5)--(-10);
\draw [style=dashed](-5)--(-11);
\draw [style=dashed](-6)--(-12);
\draw [style=dashed](-6)--(-13);
\draw [style=dashed](-7)--(-14);
\draw [style=dashed] (-7)--(-15);
\draw (1)--(-1);
         \end{tikzpicture}
         \end{center}
\end{minipage}

\bigskip

The \textit{dimension} of $\sigma\in\Sigma$ is $|\sigma|-1$, and the \textit{codimension} of $\sigma\in\Sigma$ is $|S|-|\sigma|$. Each chamber of $\Sigma$ has dimension $|S|-1$ (that is, has $|S|$ vertices). A \textit{panel} of $\Sigma$ is a codimension~$1$ simplex. Chambers $x,y\in\Delta$ are \textit{$s$-adjacent} (written $x\sim_s y$) if and only if they share a panel of type $S\backslash\{s\}$ (that is, if either $x=y$ or $x\cap y$ is a panel of type $S\backslash\{s\}$, or equivalently, if either $x=y$ or $x\backslash y$ is a vertex of type~$s$).

\noindent\begin{minipage}{0.6\linewidth}
\qquad For example, in a rank~$3$ building the chambers are triangles with the three edges (panels) of the triangle corresponding to the $3$ types of adjacency. Chambers are `glued together' along their $s$-edges if and only if they are $s$-adjacent. In a rank~$d$ building the chambers are $(d-1)$-simplices, and are glued together along their panels (codimension~$1$ faces). An alternate way to visualise a higher rank building is to imagine each chamber as a $d$-gon, with the sides in bijection with~$S$, and these polygons are glued together along their edges according adjacency. This is particularly useful when $(W,S)$ is Fuchsian.  
\end{minipage}
\begin{minipage}{0.4\linewidth}
\begin{center}
\begin{tikzpicture}[scale=0.4]
\path [fill=lightgray!70] (0,5)--(-5,4) -- (0,-2) -- (0,5);
\path [fill=lightgray!70] (0,-2)--(5,1) -- (0,5) -- (0,-2);
\path [fill=lightgray!70] (2.917, 2.666)-- (5, 6)--(0,5)--(2.917, 2.666);
\path [fill=lightgray!70] (1.321, 5.264)-- (2, 9)--(0,5)--(1.321, 5.264);
\path [fill=lightgray!70] (-2.5926, 4.4815)-- (-4, 8)--(0,5)--(-2.5926, 4.4815);
    \draw (-5, 4)--( 0, -2);
   \draw (0, 5)-- (-5, 4);
    \draw (0 ,-2) --(0, 5);
    \draw (0, -2)-- ( 5, 1);
    \draw (5, 1) --(0, 5);
    \draw (0, 5)-- (2, 9);
   \draw (2.917, 2.666)-- (5, 6);
    \draw  (2, 9) -- (1.321, 5.264);
    \draw ( 0, 5)--  (5, 6);
    \draw (-2.5926, 4.4815) --( -4, 8);
    \draw  (-4, 8) -- (0, 5);
     \draw[style=dashed]  (0, -2)--  (2.917, 2.666 );
    \draw[style=dashed] (1.321, 5.264)-- ( 0, -2);
    \draw[style=dashed] (0, -2)-- (-2.5926, 4.4815);
\end{tikzpicture}
\end{center}
\end{minipage}
\smallskip

A \textit{gallery of type $(s_1,\ldots,s_n)\in S^n$} joining $x\in\Delta$ to $y\in\Delta$ is a sequence $x_0,x_1,\ldots,x_n\in\Delta$ of chambers such that
$$
x=x_0\sim_{s_1}x_1\sim_{s_2}\cdots\sim_{s_n}x_n=y\quad\textrm{with $x_{j-1}\neq x_j$ for all $1\leq j\leq n$}.
$$
This gallery has \textit{length $n$}. 

The \textit{$W$-distance function} $\delta:\Delta\times \Delta\to W$ on $\Delta$ is defined as follows: If $x,y\in \Delta$ and if there is a minimal length gallery of type $(s_1,\ldots,s_n)$ from $x$ to $y$, then let
$$
\delta(x,y)=s_1s_2\cdots s_n.
$$
This does not depend on the particular minimal length gallery chosen.


A building is called \textit{thick} if $|\{y\in\Delta\mid x\sim_s y\}|\geq 3$ for all chambers $x\in\Delta$ and all $s\in S$, and \textit{thin} if $|\{y\in\Delta\mid x\sim_s y\}|=2$ for all chambers $x\in\Delta$ and all $s\in S$. It is clear that the Coxeter complex $\Sigma(W,S)$ is a thin building of type $(W,S)$, and that all thin buildings are Coxeter complexes. Typically we are interested in thick buildings.

A building is \textit{regular} if for each $s\in S$ the cardinality
$$
q_s+1=|\{y\in\Delta\mid x\sim_s y\}|\quad\text{is finite and does not depend on $x\in\Delta$}.
$$
All locally finite thick buildings whose Coxeter group $(W,S)$ has $m_{st}<\infty$ for all $s,t\in S$ are necessarily regular (see \cite[Theorem~2.4]{Par:06a}). Here \textit{locally finite} means that $|\{y\in\Delta\mid x\sim_s y\}|<\infty$ for all $x\in\Delta$ and $s\in S$. The numbers $(q_s)_{s\in S}$ are called the \textit{thickness parameters} (or just the \textit{parameters}) of the (regular) building.

For each $x\in\Delta$ and each $w\in W$ let
$$
\Delta_w(x)=\{y\in\Delta\mid \delta(x,y)=w\}\quad\text{be the \textit{sphere of radius $w$ centred at $x$}.}
$$
If $(\Delta,\delta)$ is regular, then by \cite[Proposition~2.1]{Par:06a} the cardinality $q_w=|\Delta_w(x)|$ does not depend on $x\in \Delta$, and is given by
$$
q_w=q_{s_1}\cdots q_{s_{k}}\quad\textrm{whenever $w=s_1\cdots s_{k}$ is a reduced expression.}
$$

The adjectives `spherical', `affine' and `Fuchsian' from Coxeter systems carry over to buildings. Thus the apartments of spherical, affine, or Fuchsian buildings are tessellations of a sphere, Euclidean space, or hyperbolic disc, respectively.

\begin{remark}\label{rem:another}
Let $\Sigma$ be a building of type $(W,S)$ with chamber set $\Delta$ and Weyl distance function $\delta:\Delta\times\Delta\to W$. It is not hard to see that the pair $(\Delta,\delta)$ satisfies the following:
\begin{enumerate}
\item[$(\text{B1})'$] $\delta(x,y)=1$ if and only if $x=y$.
\item[$(\text{B2})'$] If $\delta(x,y)=w$ and $z\in\Delta$ satisfies $\delta(y,z)=s$ with $s\in S$, then $\delta(x,z)\in \{w,ws\}$. If, in addition, $\ell(ws)=\ell(w)+1$, then $\delta(x,z)=ws$.
\item[$(\text{B3})'$] If $\delta(x,y)=w$ and $s\in S$, then there is a chamber $z\in\Delta$ with $\delta(y,z)=s$ and $\delta(x,z)=ws$.
\end{enumerate}
Conversely, suppose that we are given a set $\Delta$ and a function $\delta:\Delta\times\Delta\to W$ satisfying $(\text{B1})'$, $(\text{B2})'$, and $(\text{B3})'$. For each $I\subseteq S$ and each $x\in \Delta$ let $R_I(x)=\{y\in\Delta\mid \delta(x,y)\in W_I\}$. The poset $(\Sigma,\leq)$ of all sets of the form $R_I(x)$ with $I\subseteq S$ and $x\in\Delta$ (ordered by reverse inclusion) satisfies conditions~(P1) and (P2) from Section~\ref{sec:cox}, and hence we may regard $\Sigma$ as a simplicial complex. It turns out that this simplicial complex is a building of type $(W,S)$ (the most challenging thing to check is the existence of apartments). This gives a second approach to buildings: Specifically one can take a building of type $(W,S)$ to be a pair $(\Delta,\delta)$ where $\Delta$ is a set and $\delta:\Delta\times\Delta\to W$ is a function satisfying $(\text{B1})'$, $(\text{B2})'$ and $(\text{B3})'$. See \cite{AB:08,Tit:81} for further details. A certain fluency in both approaches is useful when working with buildings.
\end{remark}

\subsection{Buildings and groups}\label{sec:buildingsandgroups}

The group theoretic counterpart to a building is the notion of a \textit{Tits system} in a group. This concept has been very influential in group theory due to the existence of Tits systems in many ``Lie theoretic'' groups, facilitating a uniform treatment of these groups. We will see that every Tits system gives rise to a building, however not every building results from a Tits system. 

\begin{defn}
A \textit{Tits system} in a group $G$ is a quadruple $(B,N,W,S)$ where $B$ and $N$ are subgroups of $G$, and $(W,S)$ is a Coxeter system, and the following axioms are satisfied:
\begin{enumerate}
\item[(T1)] The group $G$ is generated by $B\cup N$. 
\item[(T2)] The group $H=B\cap N$ is a normal subgroup of~$N$, and $N/H\cong W$.
\item[(T3)] If $n_s\in N$ maps to $s\in S$ under the natural homomorphism of $N$ onto $W$ then for all $n\in N$ we have
$
BnBn_sB\subseteq BnB\cup Bnn_sB.
$
\item[(T4)] With $n_s$ as above, $n_sBn_s^{-1}\neq B$ for all $s\in S$.
\end{enumerate}
Since $H$ is a subgroup of $B$ there is no harm in writing $wB$ in place of $nB$ whenever $n\in N$ maps to $w\in W$ under the homomorphism of $N$ onto~$W$, and we will do so throughout.
\end{defn}

The axioms of a Tits system may appear as foreign to the reader as the axioms of a building! Thus we pause to mention some important classes of groups that admit Tits systems. To begin with, the Chevalley groups and twisted Chevalley groups admit natural Tits systems, with the associated Coxeter systems being of spherical type. For excellent treatments of this theory, see \cite{Car:89,Ste:67}. For readers familiar with Chevalley groups, the Tits system is as follows. Let $R$ be an irreducible root system, and let $G(\FF)$ be the Chevalley group of type~$R$ over the field~$\FF$. Recall that $G(\FF)$ is generated by elements $x_{\alpha}(t)$ with $\alpha\in R$ and $t\in\FF$. Let
$$
n_{\alpha}(t)=x_{\alpha}(t)x_{-\alpha}(-t^{-1})x_{\alpha}(t)\quad\text{and}\quad h_{\alpha^{\vee}}(t)=n_{\alpha}(t)n_{\alpha}(-1)\quad\text{for $\alpha\in R$ and $t\in\FF^{\times}$}.
$$
Let $N$ (respectively $H$) be the subgroup of $G(\FF)$ generated by the elements $n_{\alpha}(t)$ (respectively $h_{\alpha^{\vee}}(t)$) with $\alpha\in R$ and $t\in\FF^{\times}$. Let $U$ be the subgroup of $G(\FF)$ generated by the elements $x_{\alpha}(t)$ with $\alpha\in R^+$ and $t\in\FF$, and let $B=\langle U,H\rangle$. Then $(B,N,W_0,S_0)$ is a Tits system in~$G(\FF)$.

When the field $\FF$ has a discrete valuation Iwahori and Matsumoto~\cite{IM:65} discovered that the Chevalley group $G(\FF)$ admits another Tits system, this time with Coxeter group being the affine Weyl group~$W_{\mathrm{aff}}$. Examples of fields with discrete valuation include the $p$-adic numbers $\QQ_p$, and the field of Laurent series $\KK(\!(t)\!)$ with $\KK$ any field. For concreteness, suppose that $\FF=\KK(\!(t)\!)$. Let $K=G(\KK[[t]])$ be the Chevalley group defined over the ring of power series with coefficients in~$\KK$ (the \textit{valuation ring} of $\FF$). The evaluation map $\theta:\KK[[t]]\to\KK$, $t\mapsto 0$, induces a group homomorphism $\theta:K\to G(\KK)$. Let $N=N(\FF)$ and $B=B(\KK)$ be the groups from the previous paragraph (for the groups $G(\FF)$ and $G(\KK)$ respectively). The \textit{Iwahori subgroup} of $G(\FF)$ is the inverse image of $B$ under $\theta$. That is, $I=\theta^{-1}(B)$. Then $(I,N,W_{\mathrm{aff}},S_{\mathrm{aff}})$ is a Tits system in~$G(\FF)$. See \cite{IM:65,BT:72} for details.

There is a vast generalisation of the notation of a Chevalley group. Recall that Chevalley groups are constructed as automorphism groups of finite dimensional Lie algebras associated to Cartan matrices. In a similar (although highly non-trivial) way there is a construction of groups using infinite dimensional Lie algebras associated to generalised Cartan matrices (so called \textit{Kac-Moody algebras}, see \cite{Kac:90}). The associated \textit{Kac-Moody groups} admit Tits systems with more general Coxeter systems (see \cite{Tit:87}). In fact every Coxeter system with $m_{st}\in\{2,3,4,6,\infty\}$ arises as the Coxeter group of a Tits system in a Kac-Moody group. 


We will now describe the connection between Tits systems and buildings. Let $\Sigma$ be a building with system of apartments $\mathbf{A}$. Suppose that $G$ is a group acting on $\Sigma$ by type preserving simplicial complex automorphisms, and that $G$ preserves the apartment system~$\mathbf{A}$. We say that $G$ acts \textit{strongly transitively relative to $\mathbf{A}$} if it is transitive on pairs $(A,x)$ with $A$ an apartment in~$\mathbf{A}$ and $x$ a chamber of~$A$. For the statement of the following theorem it is convenient to adopt the approach to buildings from Remark~\ref{rem:another}.

\begin{thm}\label{thm:Titsbuilding}
\emph{(1)} Let $(B,N,W,S)$ be a Tits system in a group~$G$. Let $\Delta=G/B$, and define $\delta:\Delta\times\Delta\to W$ by 
$$
\delta(gB,hB)=w\quad\text{if and only if}\quad g^{-1}h\in BwB.
$$
Then $(\Delta,\delta)$ is a thick building of type~$(W,S)$. The set $A=\{wB\mid w\in W\}$ is an apartment, $\mathbf{A}=\{gA\mid g\in G\}$ is a system of apartments, and $G$ acts strongly transitively with respect to~$\mathbf{A}$.

\emph{(2)} Let $(\Delta,\delta)$ be a thick building of type $(W,S)$ and suppose that a group $G$ acts strongly transitively with respect to a $G$-invariant apartment system~$\mathbf{A}$. Let $o\in\Delta$ be a chamber, and let $A\in\mathbf{A}$ be an apartment containing~$o$. Let
$$
B=\{g\in G\mid go=o\}\quad\text{and}\quad N=\{g\in G\mid gx\in A\text{ for all $x\in A$}\}.
$$
Then $(B,N,W,S)$ is a Tits system in~$G$. 
\end{thm}

Thus we have a wealth of examples of thick buildings. In particular, since every Coxeter system $(W,S)$ with $m_{st}\in \{2,3,4,6,\infty\}$ can occur as the Coxeter system of a Kac-Moody group, there are thick buildings of type $(W,S)$ for every such $(W,S)$. However we should emphasise that not all buildings arise from this type of construction (see, for example, Theorem~\ref{thm:free} below).

\subsection{Affine buildings}

In this section we discuss some additional structural theory for affine buildings. Let $\Sigma$ be an affine building of type $(W_{\mathrm{aff}},S_{\mathrm{aff}})$. The \text{special vertices} of the Coxeter complex of $(W_{\mathrm{aff}},S_{\mathrm{aff}})$ are the elements of the coweight lattice $P$, and the \textit{special vertices} of $\Sigma$ are the vertices which are special vertices in some apartment. Let $V$ be the set of all special vertices of~$\Sigma$. A \textit{sector} in $\Sigma$ is a subset $\fs$ which is isomorphic to a sector in some apartment of the building (where sectors in apartments are as in Section~\ref{sec:root}). Thus sectors are always based at special vertices. A fundamental fact concerning sectors in affine buildings is (see \cite[Chapter~11]{AB:08}): 
\begin{enumerate}
\item[(S1)] If $x$ is a chamber of $\Sigma$, and if $\fs$ is a sector of $\Sigma$, then there is a subsector $\fs'$ of 
$\fs$ such that $\fs'\cup x$ is contained in an apartment.
\end{enumerate}

Figure~\ref{fig:piece} shows a simplified picture of an affine building of type $\widetilde{A}_2$, however note that if the building is thick then the `branching' actually occurs along \textit{every} 
wall, and so the picture is rather incomplete. All vertices in this building are special, and a sector is shaded.
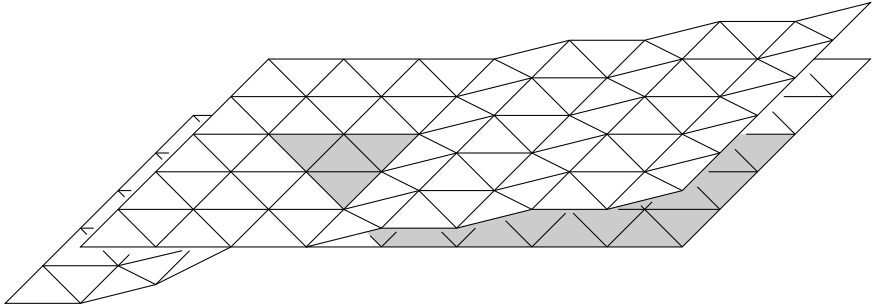
\begin{figure}[h]
\begin{center}
\begin{tikzpicture}[scale=0.5]
\path [fill=lightgray] (-6,0) -- (-4,-2) -- (-2,0);
\path [fill=lightgray] (5,-3) -- (-3,-3) -- (-3.4,-2.6) -- (-3,-2.5) -- (-1,-2.5) -- (1,-2) -- (3,-2) -- (5,-1.5) -- (6.5,0) -- (8,0) -- (5,-3);
\draw (-5,-3) -- (-3,-2.5) -- (-1,-2.5) -- (1,-2) -- (3,-2) -- (5,-1.5) -- (10,3.5) -- (8,3) -- (6,3) -- (4,2.5) -- (2,2.5) -- (0,2) -- (-6,2) -- (-11,-3) -- (5,-3) -- (10,2) -- (8.5,2);
\draw (-7,-3) -- (-9,-4) -- (-11,-4.5) -- (-13,-4.5) -- (-8,0.5) -- (-7.5,0.5);
\draw (-9,-4) -- (-8.1,-3.1);
\draw (-5,-3) -- (0,2);
\draw (-3,-2.5) -- (2,2.5);
\draw (-1,-2.5) -- (4,2.5);
\draw (1,-2) -- (6,3);
\draw (-9,-3) -- (-4,2);
\draw (-7,-3) -- (-2,2);
\draw (-10,-2) -- (-4,-2) -- (-2,-1.5) -- (0,-1.5) -- (2,-1) -- (4,-1) -- (6,-0.5);
\draw (-9,-1) -- (-3,-1) -- (-1,-0.5) -- (1,-0.5) -- (3,0) -- (5,0) -- (7,0.5);
\draw (-8,0) -- (-2,0) -- (0,0.5) -- (2,0.5) -- (4,1) -- (6,1) -- (8,1.5);
\draw (-7,1) -- (-1,1) -- (1,1.5) -- (3,1.5) -- (5,2) -- (7,2) -- (9,2.5);
\draw (-10,-2) -- (-9,-3);
\draw (-9,-1) -- (-7,-3);
\draw (-8,0) -- (-5,-3);
\draw (-7,1) -- (-4,-2);
\draw (-6,2) -- (-3,-1);
\draw (-4,2) -- (-2,0);
\draw (-2,2) -- (-1,1);
\draw (8,3) -- (9,2.5);
\draw (6,3) -- (7,2) -- (8,1.5);
\draw (4,2.5) -- (5,2) -- (6,1) -- (7,0.5);
\draw (2,2.5) -- (3,1.5) -- (4,1) -- (5,0) -- (6,-0.5);
\draw (0,2) -- (1,1.5) -- (2,0.5) -- (3,0) -- (4,-1) -- (5,-1.5);
\draw (3,-2) -- (8,3);
\draw (-1,1) -- (0,0.5) -- (1,-0.5) -- (2,-1) -- (3,-2);
\draw (-2,0) -- (-1,-0.5) -- (0,-1.5) -- (1,-2);
\draw (-3,-1) -- (-2,-1.5) -- (-1,-2.5);
\draw (-4,-2) -- (-3,-2.5);
\draw (7.7,1) -- (9,1);
\draw (6.7,0) -- (8,0);
\draw (5.7,-1) -- (7,-1);
\draw (3.3,-2) -- (6,-2);
\draw (3,-3) -- (4.2,-1.8);
\draw (1,-3) -- (1.9,-2.1);
\draw (-1,-3) -- (-0.5,-2.5);
\draw (-3,-3) -- (-2.6,-2.6);
\draw (-3,-3) -- (-3.3,-2.7);
\draw (-1,-3) -- (-1.4,-2.6);
\draw (1,-3) -- (0.3,-2.3);
\draw (3,-3) -- (2.1,-2.1);
\draw (5,-3) -- (3.9,-1.9);
\draw (6,-2) -- (5.4,-1.4);
\draw (7,-1) -- (6.4,-0.4);
\draw (8,0) -- (7.4,0.6);
\draw (9,1) -- (8.4,1.6);
\draw (-9,-4) -- (-10,-3.5) -- (-8.3,-3.1);
\draw (-11,-4.5) -- (-9.7,-3.2);
\draw (-11,-4.5) -- (-12,-3.5) -- (-10,-3.5) -- (-10.4,-3.1);
\draw (-10.83,-2.67) -- (-11,-2.5);
\draw (-9.83,-1.67) -- (-10,-1.5);
\draw (-8.83,-0.67) -- (-9,-0.5);
\draw (-7.83,0.33) -- (-8,0.5);
\draw (-11,-2.5) -- (-10.65,-2.5);
\draw (-10,-1.5) -- (-9.65,-1.5);
\draw (-9,-0.5) -- (-8.65,-0.5);
\end{tikzpicture}
\end{center}
\caption{A small piece of an $\widetilde{A}_2$ building}
\label{fig:piece}
\end{figure}

The following notion of `vector distance' gives a refined way of measuring the distance between special vertices in affine buildings. 

\begin{defn}\label{defn:vectdist} Let $x,y\in V$ be special vertices of $\Sigma$. The \textit{vector distance} 
$\vect{d}(x,y)\in P^+$ from $x$ to $y$ is defined as follows. By (B2) there is an apartment $A$ containing 
$x$ and $y$, and let $\psi:A\to \Sigma$ be a type preserving isomorphism. Then we define
$$
\vect{d}(x,y)=\bigl(\psi(y)-\psi(x)\bigr)^+,
$$
where for $\mu \in P$, we denote by $\mu^+$ the unique element in $W_0\mu\cap P^+$. This value is independent of choice of apartment $A$ and the isomorphism $\psi:A\to\Sigma$ (see \cite[Proposition~5.6]{Par:06a}).
\end{defn}

More intuitively, to compute $\vect{d}(x,y)$ one looks at the vector from $x$ to $y$ (in any apartment containing $x$ and $y$) and takes the dominant representative of this vector under the $W_0$-action.

If $A$ is an apartment and $\fs\subset A$ is a sector of $A$, then the \textit{retraction of $\Sigma$ onto $A$ with centre~$\fs$} is the map $\rho_{A,\fs}:\Sigma\to A$ computed as follows: If $x\in \Sigma$, choose (by (S1)) an apartment~$A'$ containing $x$ and a subsector of~$\fs$. Let $\psi:A'\to A$ be the isomorphism from~(B3), and let $\rho_{A,\fs}(x)=\psi(x)$. It is easy to check that this value is independent of the apartment $A'$ chosen. Intuitively speaking, the retraction $\rho_{A,\fs}$ flattens the building onto the apartment $A$ with the centre of this flattening being `deep' in the sector $\fs$ (this is illustrated in Figure~\ref{fig:busemann} for the rank~$2$ case of trees).

If $\fs$ is a sector in an apartment $A$, then there is a unique isomorphism $\psi_{A,\fs}:A\to\Sigma(W_{\mathrm{aff}},S_{\mathrm{aff}})$ mapping $\fs$ to $\fs_0$ and preserving vector distances (cf. \cite[Lemma~3.2]{Par:06b}). This gives a canonical way of fixing a Euclidean coordinate system on the apartment $A$ with respect to the sector $\fs$. 

\begin{defn}\label{defn:buse}
The \textit{vector Busemann function} associated to the sector $\fs$ is the function
$$
\vect{h}_{\fs}:\Sigma\to P\quad\text{given by}\quad \vect{h}_{\fs}(x)=\psi_{A,\fs}(\rho_{A,\fs}(x)),
$$
where $A$ is any apartment containing $\fs$ (this does not depend on~$A$). We write $\vect{h}=\vect{h}_{\fs_0}$. 
\end{defn}

\begin{example}
Let $\Sigma$ be a homogeneous tree of degree $q+1$. Thus $\Sigma$ is an affine building of type~$\widetilde{A}_1$ (see Example~\ref{ex:tree}). In this case the vector distance $\vect{d}(x,y)$ is simply the graph distance $d(x,y)$, and the vector Busemann function $\vect{h}_{\fs}:\Sigma\to P$ is the familiar `horocycle function' (perhaps with an additional superficial minus sign; see \cite[Figure~1]{BNW:08}). This is illustrated in Figure~\ref{fig:busemann} -- to compute $\vect{h}_{\fs}(v)$ for a vertex~$v$ one reads off the `level' of~$v$. For example, $\vect{h}_{\fs}(x)=-1$.  Note that there are infinitely many vertices on each horizontal level. More generally, the vector Busemann functions for an affine building are vector analogues of the usual Busemann functions for a $\mathrm{CAT}(0)$ space (see \cite[Proposition~2.8]{PW:14}).
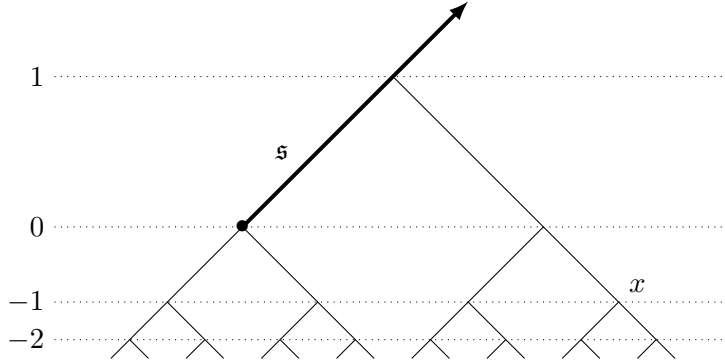
\begin{figure}[h]
\begin{center}
\begin{tikzpicture}[xscale=0.5, yscale=0.25]
\node at (-4,4) {$\bullet$};
\node at (-9,-2) [left] {$-2$};
\node at (-9,0) [left] {$-1$};
\node at (-9,4) [left] {$0$};
\node at (-9,12) [left] {$1$};
\node at (-2.5,8) [left] {$\fs$};
\node at (6.5,0) [right, above] {$x$};
\draw[-latex,line width=1.5pt] (-4,4)--(2,16);
\draw (-7.5,-3)--(-4,4);
\draw (-5.5,-3)--(-5,-2);
\draw (-3.5,-3)--(-2,0);
\draw (-1.5,-3)--(-1,-2);
\draw (-6.5,-3)--(-7,-2);
\draw (-4.5,-3)--(-6,0);
\draw (-2.5,-3)--(-3,-2);
\draw (-0.5,-3)--(-4,4);
\draw (7.5,-3)--(0,12);
\draw (5.5,-3)--(5,-2);
\draw (3.5,-3)--(2,0);
\draw (1.5,-3)--(1,-2);
\draw (6.5,-3)--(7,-2);
\draw (4.5,-3)--(6,0);
\draw (2.5,-3)--(3,-2);
\draw (0.5,-3)--(4,4);
\draw [style=dotted] (-9,-2)--(9,-2);
\draw [style=dotted] (-9,0)--(9,0);
\draw [style=dotted] (-9,4)--(9,4);
\draw [style=dotted] (-9,12)--(9,12);
\end{tikzpicture}
\end{center}
\caption{Vector Buesmann function for a tree}
\label{fig:busemann}
\end{figure}
\end{example}

\begin{example}\label{ex:padic} Let $R$ be an irreducible root system, let $G=G(\KK(\!(t)\!))$ be a Chevalley group over the field $\KK(\!(t)\!)$, and let $K=G(\KK[[t]])$. Following the work of Iwahori and Matsumoto~\cite{IM:65} and Bruhat and Tits~\cite{BT:72}, $G/K$ is the set of type zero vertices of an affine building (thus is a subset of the set of special vertices). The \textit{Cartan} and \textit{Iwasawa} 
decompositions of $G$ are (respectively): 
\begin{align}\label{eq:CIdecomp}
G=\bigsqcup_{\lambda\in Q\cap P^+}Kt_{\lambda}K\qquad\textrm{and}\qquad G=\bigsqcup_{\mu\in Q}Ut_{\mu}K,
\end{align}
(with $U$ being the subgroup of $G$ generated by the elements $x_{\alpha}(f)$ with $\alpha\in R^+$ and $f\in\mathbb{K}(\!(t)\!)$, and with $t_{\lambda}$ given by $t_{\lambda}=h_{\alpha_1^{\vee}}(t^{-a_1})\cdots h_{\alpha_d^{\vee}}(t^{-a_d})$ if $\lambda=a_1\alpha_1^{\vee}+\cdots+a_d\alpha_d^{\vee}$). The vector distance between vertices $gK$ and $hK$ in $\Delta$ is 
\begin{align}\label{eq:vbuild}
\vect{d}(gK,hK)=\lambda\qquad\textrm{if and only if}\qquad g^{-1}hK\subseteq Kt_{\lambda}K,
\end{align}
and since each $u\in U$ stabilises a subsector of the fundamental sector of $\Sigma$, 
the Busemann function $\vect{h}$ is given by 
\begin{align}\label{eq:hbuild}
\vect{h}(gK)=\mu\qquad\textrm{if and only if}\qquad gK\subseteq Ut_{\mu}K.
\end{align}
Thus the vector distance $\vect{d}$ and the vector Busemann function $\vect{h}$ are natural statistics from the group theoretic context, encoding the Cartan and Iwasawa decompositions (respectively). 
\end{example}

\subsection{Generalised polygons}\label{sec:gpoly}

Let $\Sigma$ be a building of type $(W,S)$ with chamber set~$\Delta$ and Weyl distance $\delta:\Delta\times\Delta\to W$. Let $I\subseteq S$, and let $W_I$ be the parabolic subgroup of~$W$ generated by~$I$. The \textit{$I$-residue} of a chamber $x\in\Delta$ is 
$$
R_I(x)=\{y\in\Delta\mid \delta(x,y)\in W_I\}.
$$
This residue is a building of type~$(W_I,I)$ (it is easiest to check this using the formulation of buildings from Remark~\ref{rem:another}). Thus general buildings are `made up of' many other buildings of lower rank. The rank~$1$ residues have no interesting structure, and so the important case is rank~$2$. 

Hence buildings of type $I_2(m)$ play a critical role in the theory. We have already discussed buildings of type $I_2(\infty)$ in Example~\ref{ex:tree}, and so here we consider the building of type $I_2(m)$ with $2\leq m<\infty$. It is easy to verify that these buildings are equivalent to bipartite graphs with:
\begin{center}
diameter $m$, and girth $2m$ (where girth is the length of the shortest cycle).
\end{center}
Such graphs are also known in the literature by another name: \textit{generalised $m$-gons}. Thus 
buildings of type $I_2(m)$ are equivalent to generalised $m$-gons. The chambers of the building are the edges of the generalised $m$-gon, and the apartments of the building are the cycles of length $2m$ in the generalised $m$-gon. Some examples are shown in Figure~\ref{fig:further}.

\begin{figure}[!h]
\centering
\subfigure[Building of type $I_2(3)$]{
\begin{tikzpicture} [scale=2.3]
\path 
({cos(0*180/7)},{sin(0*180/7)}) node (0)[shape=circle,draw,fill=black,scale=0.5]  {}
({cos(180/7)},{sin(180/7)}) node (1) [shape=circle,draw,scale=0.5] {}
 ({cos(2*180/7)},{sin(2*180/7)}) node (2) [shape=circle,draw,fill=black,scale=0.5] {}
  ({cos(3*180/7)},{sin(3*180/7)}) node (3) [shape=circle,draw,scale=0.5] {}
   ({cos(4*180/7)},{sin(4*180/7)}) node (4) [shape=circle,draw,fill=black,scale=0.5] {}
    ({cos(5*180/7)},{sin(5*180/7)}) node (5) [shape=circle,draw,scale=0.5] {}
     ({cos(6*180/7)},{sin(6*180/7)}) node (6) [shape=circle,draw,fill=black,scale=0.5] {}
      ({cos(7*180/7)},{sin(7*180/7)}) node (7) [shape=circle,draw,scale=0.5] {}
      ({cos(8*180/7)},{sin(8*180/7)}) node (8) [shape=circle,draw,fill=black,scale=0.5] {}
      ({cos(9*180/7)},{sin(9*180/7)}) node (9) [shape=circle,draw,scale=0.5] {}
      ({cos(10*180/7)},{sin(10*180/7)}) node (10) [shape=circle,draw,fill=black,scale=0.5] {}
      ({cos(11*180/7)},{sin(11*180/7)}) node (11) [shape=circle,draw,scale=0.5] {}
      ({cos(12*180/7)},{sin(12*180/7)}) node (12) [shape=circle,draw,fill=black,scale=0.5] {}
      ({cos(13*180/7)},{sin(13*180/7)}) node (13) [shape=circle,draw,scale=0.5] {};
\draw (0) -- (1) -- (2) --(3)--(4)--(5)--(6)--(7)--(8)--(9)--(10)--(11)--(12)--(13)--(0);
\draw (4)--(9);
\draw (3)--(12);
\draw (5)--(0);
\draw (13)--(8);
\draw (1)--(10);
\draw (6)--(11);
\draw (2)--(7);
         \end{tikzpicture}}\hspace{2cm}
         \subfigure[Building of type $I_2(4)$]{
         \begin{tikzpicture} [scale=1.7]
\path 
(-0.398806,1.28975) node (1) [shape=circle,draw,fill=black,scale=0.5] {}
(-0.112204,1.34533) node (2) [shape=circle,draw,scale=0.5] {}
(0.112204,1.34533) node (3) [shape=circle,draw,fill=black,scale=0.5] {}
(0.398806,1.28975) node (4) [shape=circle,draw,scale=0.5] {}
(0.670803,1.17155) node (5) [shape=circle,draw,fill=black,scale=0.5] {}
(0.906918,1) node (6) [shape=circle,draw,scale=0.5] {}
(1.10339, 0.777842) node (7) [shape=circle,draw,fill=black,scale=0.5] {}
(1.24481, 0.522442) node (8) [shape=circle,draw,scale=0.5] {}
(1.31416,0.309017) node (9) [shape=circle,draw,fill=black,scale=0.5] {}
(1.34986,0.0192676) node (10) [shape=circle,draw,scale=0.5] {}
(1.3215, -0.275943) node (11) [shape=circle,draw,fill=black,scale=0.5] {}
(1.23131, -0.553513) node (12) [shape=circle,draw,scale=0.5] {}
(1.08074, -0.809017) node (13) [shape=circle,draw,fill=black,scale=0.5] {}
(0.88154, -1.02244) node (14) [shape=circle,draw,scale=0.5] {}
(0.69999, -1.15435) node (15) [shape=circle,draw,fill=black,scale=0.5] {}
(0.435455, -1.27784) node (16) [shape=circle,draw,scale=0.5] {}
(0.145929, -1.34209) node (17) [shape=circle,draw,fill=black,scale=0.5] {}
(-0.145927, -1.34209) node (18) [shape=circle,draw,scale=0.5] {}
(-0.435455, -1.27784) node (19) [shape=circle,draw,fill=black,scale=0.5] {}
(-0.69999, -1.15435) node (20) [shape=circle,draw,scale=0.5] {}
(-0.88154, -1.02244) node (21) [shape=circle,draw,fill=black,scale=0.5] {}
(-1.08074, -0.809017) node (22) [shape=circle,draw,scale=0.5] {}
(-1.23131, -0.553516) node (23) [shape=circle,draw,fill=black,scale=0.5] {}
(-1.3215, -0.275944) node (24) [shape=circle,draw,scale=0.5] {}
(-1.34986, 0.0192676) node (25) [shape=circle,draw,fill=black,scale=0.5] {}
(-1.31416, 0.309017) node (26) [shape=circle,draw,scale=0.5] {}
(-1.24481, 0.522442) node (27) [shape=circle,draw,fill=black,scale=0.5] {}
(-1.10339, 0.777842) node (28) [shape=circle,draw,scale=0.5] {}
(-0.906921, 1.) node (29) [shape=circle,draw,fill=black,scale=0.5] {}
(-0.670803,1.17155) node (30) [shape=circle,draw,scale=0.5] {}
;
\draw (1)--(2)--(3)--(4)--(5)--(6)--(7)--(8)--(9)--(10)--(11)--(12)--(13)--(14)--(15)--(16)--(17)--(18)--(19)--(20)--(21)--(22)--(23)--(24)--(25)--(26)--(27)--(28)--(29)--(30)--(1);
\draw (1)--(10);
\draw (2)--(15);
\draw (3)--(20);
\draw (4)--(25);
\draw (5)--(12);
\draw (6)--(29);
\draw (7)--(16);
\draw (8)--(21);
\draw (9)--(26);
\draw (11)--(18);
\draw (13)--(22);
\draw (14)--(27);
\draw (17)--(24);
\draw (19)--(28);
\draw (23)--(30);
\end{tikzpicture}}
         \caption{Examples of finite thick generalised $m$-gons}\label{fig:further}
\end{figure}
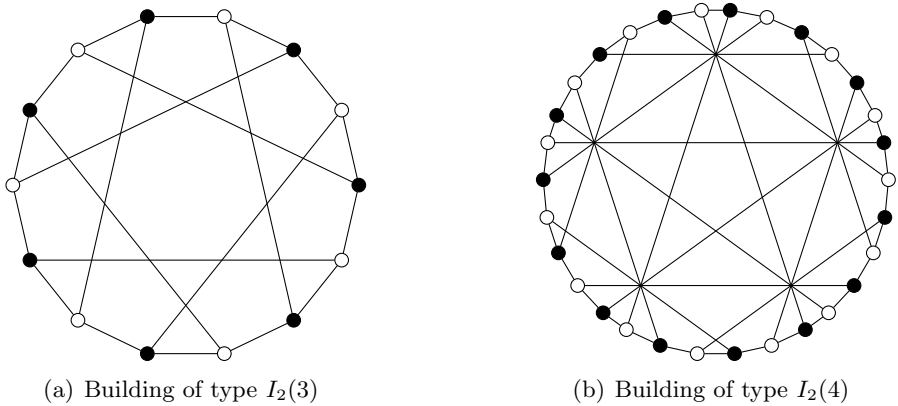

It is not hard to see that finite thick generalised $m$-gons are necesarily biregular (with alternating valencies $q+1$ and $r+1$, say), and that if $m$ is odd then necessarily $q=r$.  There is a lot to say about generalised $m$-gons (see, for example, \cite{HVM:98}). Let us simply recall some of the key results. Since we are rarely interested in `ordinary' $m$-gons, we will usually omit the adjective `generalised'. 

\begin{thm}\label{thm:existmgon}\emph{\cite{Tit:77}} For each $m\geq 2$ there exist thick $m$-gons.
\end{thm}

Theorem~\ref{thm:existmgon} shows that there are many examples of thick $m$-gons. The proof of the theorem is via a `free construction' which produces $m$-gons in which every vertex has infinitely many neighbours. Finite thick $m$-gons are much more restricted, as shown by the following beautiful and unexpected theorem due to Feit and Higman.

\begin{thm}\label{thm:FH}\emph{\cite{FH:64}}
Finite thick generalsed $m$-gons exist if and only if $m\in\{2,3,4,6,8\}$.
\end{thm}

The `only if' part of the Feit-Higman Theorem follows from rather involved character theory (see Appendix~\ref{app:rank2}). For the `if' part of the theorem, the existence of finite thick $2$-gons is clear (they are just complete bipartite graphs), and examples of finite thick $3$-gons, $4$-gons, $6$-gons, and $8$-gons come from the Chevalley groups $A_2(q)$, $B_2(q)$, $G_2(q)$, and the twisted Chevalley group $^2F_4(2^{2m+1})$ via Theorem~\ref{thm:Titsbuilding}.

Thus in the finite theory the $m$-gons with $m=2,3,4,6,8$ play a special role. We often call $4$-gons, $6$-gons, and $8$-gons \textit{quadrangles}, \textit{hexagons}, and \textit{octagons} respectively. Generalised $3$-gons are called \textit{projective planes} (this comes from the older language of incidence geometry). There are severe restrictions on the possible thickness parameters of projective planes, quadrangles, hexagons, and octagons. Most of these restrictions have character theoretic proofs (we prove some of these in Appendix~\ref{app:rank2}). References to the original works can be found in Kantor~\cite{Kan:84}.

\begin{thm}\label{thm:param}
Let $\Gamma$ be a finite thick $m$-gon with parameters $(q,r)$. 
\begin{enumerate}
\item[\emph{(1)}] If $m=3$ then $q=r$, and if $q\equiv 1,2 \mod 4$ then $q$ is a sum of two squares.
\item[\emph{(2)}] If $m=4$ then $q\leq r^2$, $r\leq q^2$, and $q^2(qr+1)/(q+r)\in\ZZ$.
\item[\emph{(3)}] If $m=6$ then $q\leq r^3$, $r\leq q^3$, $q^3(q^2r^2+qr+1)/(q^2+qr+r^2)\in\ZZ $, and $\sqrt{qr}\in\ZZ$.
\item[\emph{(4)}] If $m=8$ then $q\leq r^2$, $r\leq q^2$, $q^4(qr+1)(q^2r^2+1)/(q+r)(q^2+r^2)\in\ZZ$, and $\sqrt{2qr}\in\ZZ$. 
\end{enumerate} 
\end{thm}

The number theoretic part of statement~(1) is called the \textit{Bruck-Ryser-Chowla Theorem}. For example, it prohibits the existence of a projective plane with parameter~$q=6$. 

Together, Theorems~\ref{thm:FH} and~\ref{thm:param} place a lot of conditions on the structure of general locally finite thick buildings. In particular we immediately have the following corollary by looking at rank~$2$ residues.

\begin{cor}\label{cor:restriction} Let $\Sigma$ be a locally finite thick building of type $(W,S)$ with parameters~$(q_s)_{s\in S}$. For each $s,t\in S$ with $s\neq t$ let $m_{st}$ be the order of $st$. Then:
\begin{enumerate}
\item[\emph{(1)}] $m_{st}\in\{2,3,4,6,8,\infty\}$ for all $s,t\in S$ with $s\neq t$, and 
\item[\emph{(2)}] if $m_{st}\in \{3,4,6,8\}$ then the pair $(q_s,q_t)$ satisfies the constraints from Theorem~\ref{thm:param}. 
\end{enumerate}
\end{cor}

It is an open problem to determine the possible parameters of thick projective planes, quadrangles, hexagons, and octagons. The \textit{known} examples have the following parameters (where we arrange the parameters $(q,r)$ so that $q\leq r$). Projective planes: $(q,q)$ with $q$ a prime power. Quadrangles: $(q,q)$, $(q,q^2)$, $(q^2,q^3)$, $(q-1,q+1)$ with $q$ a prime power. Hexagons: $(q,q)$, $(q,q^3)$ with $q$ a prime power. Octagons: $(q,q^2)$ with $q=2^{2k+1}$ an odd power of~$2$. Constructing a thick generalised $m$-gon with parameters other than these, or proving further restrictions on the possible parameters, would be revolutionary. Finally we note that for a given value of the parameters there may be multiple non-isomorphic generalised $m$-gons. For example there are $4$ distinct projective planes with parameters $(9,9)$. 

\subsection{Classification and free constructions}\label{sec:classification}

Thick irreducible spherical (respectively affine) buildings of rank at least $3$ (respectively~$4$) have been classified by Tits (respectively, Tits and Weiss) (see \cite{Tit:74} and \cite{BT:72,Tit:86,Wei:09}). Put very roughly, this classification says that all irreducible thick spherical buildings of rank at least~$3$ arise from groups of Lie origin via Tits systems, and that all irreducible thick affine buildings of rank at least~$4$ arise from groups of Lie origin defined over fields (or skew-fields) with discrete valuation via affine Tits systems. The precise statement of these classification theorems is involved (see the above references for details).

On the other hand, the following essentially free construction shows that the situation is very different for Coxeter systems which contain no irreducible spherical rank~$3$ parabolic subgroups.

\begin{thm}\emph{\cite{Ron:86}}\label{thm:free}
Let $(W,S)$ be a Coxeter system such that every irreducible rank~$3$ parabolic subgroup of $W$ is infinite. Suppose that $(q_s)_{s\in S}$ is a sequence of integers such that for each pair $s,t\in S$ there exists a generalised $m_{st}$-gon with parameters $(q_s,q_t)$. Then there exists a locally finite thick regular building of type $(W,S)$ whose rank~$2$ residues of type $W_{\{s,t\}}$ range through any desired set of generalised $m_{st}$-gons having parameters $(q_s,q_t)$. 
\end{thm}

The above theorem tells us that rank~$3$ irreducible affine buildings (that is, those of type $\widetilde{A}_2$, $\widetilde{B}_2$, and $\widetilde{G}_2$) cannot be classified (at least not in the spirit of the higher rank classification; for example, one can make a $\widetilde{A}_2$ building with thickness parameter~$q=9$ whose rank~$2$ residues can be chosen freely from the $4$ non-isomorphic projective planes with $q=9$). The theorem also applies to all Fuchsian Coxeter systems, and so these buildings are also unclassifiable. In other words, there are many of these buildings that are not related in any nice way to groups. Using Corollary~\ref{cor:restriction} and Theorem~\ref{thm:free} we have the following existence result for Fuchsian buildings:

\begin{thm}\label{thm:fuchsian}
Let $(W,S)$ be a Fuchsian Coxeter system of type $F(k_1,\ldots,k_n)$. There exists a locally finite thick building of type $(W,S)$ if and only if $k_i\in\{2,3,4,6,8\}$ for all $i=1,\ldots,n$ and either $k_i\in\{2,4\}$ for some $i=1,\ldots,n$ or $|\{i\mid k_i=8\}|$ is even. 
\end{thm}

\begin{proof}
Suppose that there is a locally finite thick building $\Sigma$ of type $F(k_1,\ldots,k_n)$. By Corollary~\ref{cor:restriction} we have $k_i\in\{2,3,4,6,8\}$ (since $k_i\neq\infty$ by definition for Fuchsian systems). Let $(q_1,\ldots,q_n)$ be the parameters of $\Sigma$ (arranged cyclically so that the parameters of the rank~$2$ residue corresponding to $k_i$ are $q_i$ and $q_{i+1}$). If $k_i\neq 2,4$ for all $i$ then $k_i\in\{3,6,8\}$ for all $i$. If $k_i\in\{3,6\}$ then $\sqrt{q_iq_{i+1}}\in\ZZ$, and if $k_i=8$ then $\sqrt{2q_iq_{i+1}}\in\ZZ$ (see Theorem~\ref{thm:param}). Multiplying these conditions together gives $2^{\alpha/2}q_1\cdots q_n\in\ZZ$ where $\alpha=|\{i\mid k_i=8\}|$,
and hence $\alpha$ is even. This proves the `only if' part of the theorem. We leave the `if' part of the theorem as an exercise (using Theorem~\ref{thm:free} and the known examples of generalised $m$-gons listed after Corollary~\ref{cor:restriction}).
\end{proof}


\section{Random walks on buildings}\label{sec:3}

A \textit{random walk} on a finite or countable space $X$ is a sequence $(X_n)_{n\geq 0}$ of $X$-valued random variables governed by a stochastic \textit{transition matrix} (or \textit{transition operator}) $P=(p(x,y))_{x,y\in X}$. That is, 
$$
p(x,y)=\PP[X_{n+1}=y\mid X_n=x]\quad \text{for all $x,y\in X$ and all $n\geq 0$},
$$
and the transition operator $P$ acts on $\ell^1(X)$ by
$$
Pf(x)=\sum_{y\in X}p(x,y)f(y)\quad\text{for all $x\in X$}.
$$
The \textit{$n$-step transition probabilities} of the walk are
$$
p^{(n)}(x,y)=\PP[X_n=y\mid X_0=x],
$$
and we have $P^n=(p^{(n)}(x,y))_{x,y\in X}$. The random walk $(X_n)_{n\geq 0}$ is \textit{irreducible} if for each pair $x,y\in X$ there exists $n\geq 0$ such that $p^{(n)}(x,y)>0$. For the general theory of random walks we refer to \cite{Woe:00}.

Here we are interested in random walks on buildings $\Sigma$ (and associated groups). There are a few variations; for example, one might consider random walks on the set $\Delta$ of chambers of $\Sigma$, or one might consider random walks on the vertices of~$\Sigma$. This latter case is particularly natural for affine buildings. To begin with we will consider random walks on the chambers of a general (locally finite) building, mainly following the setup from~\cite{Par:06a,GMP:14}.

\subsection{Random walks on chambers and the Hecke algebra}\label{sec:hecke}

Let $\Sigma$ be a locally finite building of type $(W,S)$ with chamber set $\Delta$ and parameters $(q_s)_{s\in S}$. Let $(X_n)_{n\geq 0}$ be a random walk on the chamber set. Without some additional assumptions on the walk there is not so much that one can say. A natural assumption is to assume that the walk is \textit{isotropic}, meaning that the transition probabilities depend only on the Weyl distance:

\begin{defn}
A random walk $(X_n)_{n\geq 0}$ on $\Delta$ is \textit{isotropic} if the transition probabilities of the walk satisfy $p(x,y)=p(x',y')$ whenever $\delta(x,y)=\delta(x',y')$.
\end{defn}

Isotropic random walks have a beautiful algebraic structure. For each $w\in W$ let $P_w=(p_w(x,y))_{x,y\in\Delta}$ be the transition operator of the isotropic random walk with
$$
p_w(x,y)=\begin{cases}
q_w^{-1}&\text{if $\delta(x,y)=w$}\\
0&\text{otherwise}.
\end{cases} 
$$
The following elementary proposition shows that every isotropic random walk on $\Delta$ is a convex combination of the random walks $P_w$, $w\in W$. 

\begin{prop}\label{prop:transition}
A random walk on $\Delta$ with transition operator $P$ is isotropic if and only if
$$
P=\sum_{w\in W}a_wP_w\quad\text{where $a_w\geq 0$ and $\sum_{w\in W}a_w=1$},
$$
in which case $p(x,y)=a_wq_w^{-1}$ if $\delta(x,y)=w$. 
\end{prop}

Therefore we are naturally lead to consider linear combinations of the (linearly independent) operators $P_w$, $w\in W$. Let $\scP$ be the vector space over $\CC$ with basis $\{P_w\mid w\in W\}$. The key facts about $\scP$ are summarised below (see \cite[\S~3]{Par:06a}).

\begin{thm}\label{thm:product}\emph{\cite{Par:06a}}
The vector space $\scP$ is an associative unital algebra under composition of linear operators, and the multiplication table with respect to the vector space basis $\{P_w\mid w\in W\}$ is given by
$$
P_uP_v=\sum_{w\in W}c_{u,v}^wP_w\quad\text{where } c_{u,v}^w=\frac{q_w}{q_uq_v}|\Delta_u(x)\cap \Delta_{v^{-1}}(y)|\quad\text{for any $x,y\in\Delta$ with $\delta(x,y)=w$}.
$$
In particular, the intersection cardinalities $|\Delta_u(x)\cap \Delta_{v^{-1}}(y)|$ depend only on $u,v$ and $\delta(x,y)$, and for $w\in W$ and $s\in S$ we have
\begin{align}\label{eq:presentation}
P_wP_s=\begin{cases}
P_{ws}&\text{if $\ell(ws)=\ell(w)+1$}\\
q_s^{-1}P_{ws}+(1-q_s^{-1})P_w&\text{if $\ell(ws)=\ell(w)-1$.}
\end{cases}
\end{align}
\end{thm}

The algebra $\scP$ is called the \textit{Hecke algebra} of the building. If $\delta(x,y)=w$ then the $n$-step transition probability $p^{(n)}(x,y)$ is given by
\begin{align*}
p^{(n)}(x,y)=q_w^{-1}a_w^{(n)},\quad\text{where}\quad P^n=\sum_{w\in W}a_w^{(n)}P_w.
\end{align*}
Thus finding $p^{(n)}(x,y)$ when $\delta(x,y)=w$ is equivalent to finding the coefficient of $P_w$ in~$P^n$.

Before surveying known results on isotropic random walks, let us briefly indicate how isotropic random walks arise from bi-invariant measures on groups acting on buildings (for example, groups of Lie type, or Kac-Moody groups). Specifically we have the following (see \cite[Lemma~8.1]{CW:04} for a proof in a similar context).

\begin{prop}\label{prop:walkgp1} Let $G$ be a locally compact group acting transitively on a regular building $\Sigma$, and let $B$ be the stabiliser of a fixed base chamber~$o$. Normalise the Haar measure on $G$ so that $B$ has measure~$1$. Let $\varphi$ be the density function of a bi-$B$-invariant probability measure on~$G$. If the group $B$ acts transitively on each set $\Delta_w(o)$ with $w\in W$, then the assignment
$$
p(go,ho)=\varphi(g^{-1}h)
$$
for $g,h\in G$ defines an isotropic random walk on the chambers of~$\Sigma$. 
\end{prop}


\subsection{Isotropic random walks on spherical buildings}\label{sec:spherical}

Let $\Sigma$ be a locally finite thick spherical building of type~$(W,S)$ and let $(X_n)_{n\geq 0}$ be an isotropic random walk on the set $\Delta$ of chambers with transition operator~$P=(p(x,y))_{x,y\in\Delta}$. Since the set $\Delta$ of chambers is finite, natural questions to ask include:
\begin{enumerate}
\item[(1)] What is the limiting distribution of the walk?
\item[(2)] What is the value of $p^{(n)}(x,y)$?
\item[(3)] What is the mixing time for the walk?
\end{enumerate}

The first question is very easy to answer:
\begin{prop}
Let $(X_n)_{n\geq 0}$ be an irreducible isotropic random walk on the set $\Delta$ of chambers of a locally finite thick spherical building. Then the uniform distribution is the unique invariant measure, and 
$$
\lim_{n\to\infty}\mu^{(n)}(x)=\frac{1}{|\Delta|}\quad\textrm{for all $x\in\Delta$}.
$$ 
\end{prop}

\begin{proof}
Using the thickness of the building it is not difficult to see that irreducible isotropic random walks on $\Delta$ are necessarily aperiodic (see \cite[Lemma~4.3]{GMP:14}). Thus an irreducible isotropic walk has a unique stationary measure. To see that this stationary measure is the uniform measure $u:\Delta\to[0,1]$, note that for each $y\in\Delta$,
$$
\sum_{x\in \Delta}u(x)p(x,y)=\frac{1}{|\Delta|}\sum_{w\in W}\sum_{x\in \Delta_w(y)}p(x,y)=\frac{1}{|\Delta|}\sum_{w\in W}\frac{q_w}{q_{w^{-1}}}a_{w^{-1}}=\frac{1}{|\Delta|}\sum_{w\in W}a_{w^{-1}}=\frac{1}{|\Delta|}=u(y),
$$
where we have used the fact that if $\delta(x,y)=w$ then $\delta(y,x)=w^{-1}$, and that $q_{w^{-1}}=q_w$. 
\end{proof}

To address questions~(2) and~(3) we can apply the techniques from~\cite{DR:00}, where the representation theory of Hecke algebras is used to analyse convergence of systematic scan Metropolis algorithms. The aims and context of \cite{DR:00} are quite different from our setting here, and so we will give an overview of the basic setup of~\cite{DR:00}, translating into our building theoretic point of view.

We first recall some basic representation theory of finite dimensional associative unital algebras over~$\CC$. This theory generalises the more familiar representation theory of finite groups (see, for example, \cite{GP:00}). 

\begin{defn} Let $\scA$ be a finite dimensional unital associative algebra over~$\CC$.  A \textit{representation} of $\scA$ is a pair $(\rho,V)$ where $V$ is a $\CC$-vector space, and $\rho:\scA\to\mathrm{End}(V)$ is an algebra homomorphism. The \textit{character} of the representation $(\rho,V)$ is the function $\chi_{\rho}:\scA\to\CC$ given by
$$
\chi_{\rho}(A)=\tr(\rho(A))\quad\text{for all $A\in\scA$}.
$$
The \textit{dimension} of $(\rho,V)$ is $\dim(V)$ (assumed to be finite throughout this section). Sometimes it is convenient to simply denote a representation $(\rho,V)$ by $\rho$, and to write $V=V_{\rho}$. 
\end{defn}

Since $\mathrm{End}(V)\cong M_d(\CC)$ (the algebra of $d\times d$ matrices with entries in~$\CC$, where $d=\dim(V)$), a representation of $\scA$ amounts to ``representing the algebra elements by matrices''. Familiar notions of direct sum, subrepresentations, and irreducibility carry over from the group setting. 

An algebra is \textit{semisimple} if every finite dimensional representation decomposes as a direct sum of irreducible representations. In this case the irreducible representations are the `atomic building blocks' of the representation theory in the sense that every representation can be written as a direct sum of these atoms. Note that the group algebra of a finite group is necessarily semisimple (this is Maschke's Theorem), however general algebras need not be.

\subsubsection{Hecke algebras and the geometric representation}

Let $(W,S)$ be a spherical Coxeter system, and let $(q_s)_{s\in S}$ be a sequence of numbers with $q_s>0$ and $q_s=q_t$ if $s$ and $t$ are conjugate in~$W$. Let $\scH$ be the algebra over $\CC$ generated by symbols $T_w$ (with $w\in W$) with relations
\begin{align}\label{eq:presentation2}
T_wT_s=\begin{cases}
T_{ws}&\text{if $\ell(ws)>\ell(w)$}\\
q_s^{-1}T_{ws}+(1-q_s^{-1})T_{w}&\text{if $\ell(ws)<\ell(w)$}.
\end{cases}
\end{align}
The algebra $\scH$ is unital (with identity $T_e$) and associative, and is called an \textit{abstract Hecke algebra} (see \cite{Hum:90}).

If $(q_s)_{s\in S}$ are the parameters of a locally finite spherical building of type $(W,S)$ then, comparing (\ref{eq:presentation}) and (\ref{eq:presentation2}) we see that the Hecke algebra $\scP$ of~$\Sigma$ gives a representation of the abstract Hecke algebra $\scH$. More specifically, let $V_{\Delta}=\bigoplus_{x\in\Delta}\CC x$ be a vector space with basis indexed by the the chambers of the building, and let $\rho_{\Delta}:\scH\to\mathrm{End}(V_{\Delta})$ be the linear map with $\rho_{\Delta}(T_w)=P_w$ for all $w\in W$. Then
$$
(\rho_{\Delta},V_{\Delta})\quad\text{is a $|\Delta|$-dimensional representation of $\scH$}.
$$
We call this representation the \textit{geometric representation} of the Hecke algebra~$\scH$. In fact the map $\rho_{\Delta}:\scH\to\scP$ is bijective, and so $\scP\cong \scH$. We emphasise that the geometric representation of the abstract Hecke algebra $\scH$ only exists when there is a building with parameters~$(q_s)_{s\in S}$.

It is well known that the algebra $\scH$ is semisimple (see \cite{GU:89}). Thus the geometric representation decomposes into a direct sum of irreducible representations. Writing $\mathrm{Irrep}(\scH)$ for the set of irreducible representations of $\scH$ (more formally, isomorphism classes of irreducible representations) we have 
$$
V_{\Delta}=\bigoplus_{\rho\in\mathrm{Irrep}(\scH)}V_{\rho}^{\oplus m_{\rho}}
$$
where for each $\rho\in\mathrm{Irrep}(\scH)$ the integer $m_{\rho}\geq 0$ is the multiplicity of $(\rho,V_{\rho})$ in $(\rho_{\Delta},V_{\Delta})$.  Thus the character $\chi_{\Delta}$ of the geometric representation is given by
\begin{align}\label{eq:character}
\chi_{\Delta}=\sum_{\rho\in\mathrm{Irrep}(\scH)}m_{\rho}\chi_{\rho}.
\end{align}

The first fundamental task is to compute the multiplicities~$m_{\rho}$. 

\begin{thm}\label{thm:character}
The multiplicity $m_{\rho}$ of the irreducible representation $\rho$ in $(\rho_{\Delta},V_{\Delta})$ is
$$
m_{\rho}=\frac{\mathrm{dim}(\rho)}{\langle\chi_{\rho},\chi_{\rho}\rangle},\quad\text{where}\quad\langle f,g\rangle=\frac{1}{|\Delta|}\sum_{w\in W}q_wf(T_w)g(T_{w^{-1}})
$$
for functions $f,g:\scH\to\CC$.
\end{thm}

\begin{proof} It follows from general results on the representation theory of symmetric algebras (see \cite[Corollary~7.2.4 and \S8.1.8]{GP:00}) that if $\chi$ and $\chi'$ are irreducible characters of $\scH$ then
\begin{align}\label{eq:innerprod}
\langle\chi,\chi'\rangle=0\quad\text{if and only if $\chi$ and $\chi'$ are non-isomorphic}.
\end{align}
Thus taking inner products with $\chi_{\rho}$ in (\ref{eq:character}), and using the facts that $\chi_{\Delta}(T_w)=\tr(P_w)=\delta_{w,e}|\Delta|$ and $\chi_{\rho}(T_e)=\dim(\rho)$, the result follows.
\end{proof}

\subsubsection{The transition probabilities $p^{(n)}(x,y)$}

We now return to question (2), seeking a formula for the $n$-step return probabilities $p^{(n)}(x,y)$. In principal one could compute this probability by noting that it is the $(x,y)^{\text{th}}$ entry of the matrix~$P^n$. Of course this is not a practical method because $P$ is a very large matrix (for example, for the smallest thick $F_4$ building the matrix $P$ has approximately $2\times 10^8$ rows and columns!). The following theorem gives a more practical solution to the problem. 

\begin{thm}\label{thm:spherical1}
Let $P=\sum a_wP_w\in\scP$ be the transition matrix of an isotropic random walk on a regular spherical building, and let $T=\sum a_wT_w\in\scH$. Then
$$
p^{(n)}(x,y)=\frac{1}{|\Delta|}\sum_{\rho\in\mathrm{Irrep}(\scH)}m_{\rho}\chi_{\rho}(T^nT_{w^{-1}})\quad\text{if $\delta(x,y)=w$}.
$$
\end{thm}

\begin{proof} We claim that
\begin{align}\label{eq:1obs}
\chi_{\Delta}(T_uT_{v^{-1}})=q_u^{-1}|\Delta|\delta_{u,v}\quad\text{for all $u,v\in W$}.
\end{align}
To see this, note that $\chi_{\Delta}(T_w)=\tr(\rho_{\Delta}(T_w))=\tr(P_w)=|\Delta|\delta_{w,e}$ (because each $P_w$ is a $|\Delta|\times|\Delta|$ matrix, with $(x,y)^{\mathrm{th}}$ entry equal to $1$ if $\delta(x,y)=w$ and $0$ otherwise). Thus 
$$
\chi_{\Delta}(T_uT_{v^{-1}})=\tr(P_uP_{v^{-1}})=\sum_{w\in W}c_{u,v^{-1}}^{w}\tr(P_w)=c_{u,v^{-1}}^e|\Delta|=\frac{1}{q_uq_v}|\Delta_u(o)\cap \Delta_v(o)|,
$$
where we have used Theorem~\ref{thm:product}, and (\ref{eq:1obs}) follows.

Now, if $P^n=\sum a_w^{(n)}P_w$ then $T^n=\sum a_w^{(n)}T_w$, and so
$$
\frac{1}{|\Delta|}\chi_{\Delta}(T^nT_{w^{-1}})=\frac{1}{|\Delta|}\sum_{v\in W}a_w^{(n)}\chi_{\Delta}(T_vT_{w^{-1}})=q_w^{-1}a_w^{(n)}.
$$
If $\delta(x,y)=w$ then $p^{(n)}(x,y)=q_w^{-1}a_w^{(n)}$, and thus
$
p^{(n)}(x,y)=\chi_{\Delta}(T^nT_{w^{-1}})/|\Delta|$.
The result follows from~(\ref{eq:character}). 
\end{proof}

\subsubsection{Mixing times}

Now we move to the more sophisticated question of mixing times. To begin with we need to define a notion of ``distance'' between two measures. In the literature the \textit{total variation distance} has become a standard choice, popularised by Persi Diaconis (see, for example, \cite{Dia:88}). This distance is defined as follows: If $\mu$ and $\nu$ are probability measures on $\Delta$ then
$$
\|\mu-\nu\|_{\mathrm{tv}}=\max_{A\subseteq \Delta}|\mu(A)-\nu(A)|.
$$
To work with the total variation distance it is helpful to observe the following elementary fact. 

\begin{lemma}
Let $\mu$ be probability measure on~$\Delta$ and let $u:\Delta\to[0,1]$ be the uniform distribution. Then
$$
\|\mu-u\|_{\mathrm{tv}}=\frac{1}{2}\|\mu-u\|_1\leq\frac{\sqrt{|\Delta|}}{2}\|\mu-u\|_2
$$
where $\|\cdot\|_1$ and $\|\cdot\|_2$ are the $\ell^1$ and $\ell^2$ norms respectively. 
\end{lemma}

\begin{proof}
Clearly $\|\mu-u\|_{\mathrm{tv}}$ equals either $|\mu(A)-u(A)|$ or $|\mu(B)-u(B)|$, where $A$ and $B$ are the sets $A=\{x\in \Delta\mid \mu(x)>u(x)\}$ and $B=\{x\in \Delta\mid \mu(x)<u(x)\}$. In fact
$$
|\mu(A)-u(A)|=|1-\mu(B)-1+u(B)|=|\mu(B)-u(B)|
$$
and so $\|\mu-u\|_{\mathrm{tv}}=\frac{1}{2}\left(|\mu(A)-u(A)|+|\mu(B)-u(B)|\right)=\frac{1}{2}\|\mu-u\|_1$. The final inequality follows from Cauchy-Schwarz. 
\end{proof}

Let $\mu^{(n)}:\Delta\to[0,1]$ be the measure $\mu^{(n)}(x)=p^{(n)}(o,x)$, where $o\in\Delta$ is a fixed chamber of~$\Sigma$. The following theorem gives a mathematically tractable upper bound estimate for the total variation distance $\|\mu^{(n)}-u\|_{\mathrm{tv}}$, and thus can be used to give upper bounds for mixing times. Define an involution $*:\scH\to\scH$ by $(\sum a_wT_w)^*=\sum\overline{a}_wT_{w^{-1}}$. Note that $\scH$ has a $1$-dimensional representation $\rho_{\mathrm{triv}}$ (the \textit{trivial representation}) given by $\rho_{\mathrm{triv}}(T_s)=1$ for all $s\in S$ (to check this, simply verify that the defining relations (\ref{eq:presentation2}) are satisfied).

\begin{thm}\label{thm:spherical2}\emph{(cf. \cite{DR:00})}
Let $P=\sum a_wP_w$ be the transition operator of an isotropic random walk on the chambers of a regular spherical buildings, and let $T=\sum a_wT_w\in\scH$. Then
$$
\|\mu^{(n)}-u\|_{\mathrm{tv}}^2\leq\frac{1}{4}\sum_{\rho\neq \rho_{\mathrm{triv}}}m_{\rho}\chi_{\rho}\left(T^{n}(T^*)^n\right).
$$
\end{thm}

\begin{proof}
We have
\begin{align*}
\|\mu^{(n)}-u\|_{\mathrm{tv}}^2&\leq \frac{|\Delta|}{4}\|\mu^{(n)}-u\|_2^2=\frac{|\Delta|}{4}\langle\mu^{(n)}-u,\mu^{(n)}-u\rangle_2=\frac{1}{4}\left(|\Delta|\langle\mu^{(n)},\mu^{(n)}\rangle_2-1\right).
\end{align*}
Now
$$
\langle\mu^{(n)},\mu^{(n)}\rangle_2=\sum_{w\in W}q_w^{-1}\left(a_w^{(n)}\right)^2=\frac{1}{|\Delta|}\chi_{\Delta}(T^n(T^*)^n),
$$
and so 
$$
\|\mu^{(n)}-u\|_{\mathrm{tv}}^2\leq \frac{1}{4}\Big(\chi_{\Delta}(T^n(T^*)^n)-1\Big).\qedhere
$$
\end{proof}

Theorems~\ref{thm:spherical1} and~\ref{thm:spherical2}, together with the multiplicity formula from Theorem~\ref{thm:character}, provide some basic theory for studying isotropic random walks on spherical buildings. To make more practical estimates in given examples one needs to work harder with the representation theory (for example, to give meaningful bounds on the right hand side of the inequality in Theorem~\ref{thm:spherical2}). Some calculations are made in~\cite{DR:00}, in a different context, that can be translated to give estimates for certain random walks on buildings $\Sigma$ of type $A_{n}$ (see \cite[Proposition~7.4 and Theorem~7.5]{DR:00}). However in the building theoretic context the walks covered in \cite{DR:00} are perhaps not the most natural (for example, the simple random walk is not covered). Thus there is still a lot to do in this direction, and we hope that the setup provided above might stimulate some future research. In Appendix~\ref{app:rank2} we outline the details of the representation theory in the rank~$2$ case (that is, when the building is a generalised polygon).

\subsection{Random walks on affine buildings}

Let $(X_n)_{n\geq 0}$ be a random walk on an infinite graph with transition probabilities $p(x,y)$. Natural questions to ask in this setting include:
\begin{enumerate}
\item[(1)] At what velocity does the random walk move to infinity?
\item[(2)] What is the distribution of the fluctuations away from expected distance?
\item[(3)] What are the asymptotics of $p^{(n)}(x,y)$?
\end{enumerate}
Appropriate solutions to these problems come in the form of a law of large numbers, a central limit theorem, and a local limit theorem (respectively).

For random walks on affine buildings it is natural to consider both random walks on the chambers of the building, and random walks on the vertices of the building. The latter case now has a rather complete theory. We will consider both cases below.

If the affine building has rank~$2$ then we are dealing with a random walk on a trees. In this context there is a huge literature which takes us too far afield to discuss here, and so we will focus on the higher rank case.

\subsubsection{Random walks on the vertices of an affine building}\label{sec:vertices}

Let $R$ be an irreducible root system with coweight lattice~$P$, and let $(W_{\mathrm{aff}},S_{\mathrm{aff}})$ be the associated affine Coxeter system. Let $\Sigma$ be a regular affine building of type $(W_{\mathrm{aff}},S_{\mathrm{aff}})$, and let $V$ be the set of all special vertices of~$\Sigma$. Recall the definitions of the vector distance function $\vect{d}(\cdot,\cdot)$ from Definition~\ref{defn:vectdist}. Some of the formulae of this section become more complicated in the case of $\widetilde{C}_n$ buildings with $q_0\neq q_n$, and so here we will restrict to the case $q_0=q_n$ for $\widetilde{C}_n$ buildings (see \cite{Par:06a,Par:06b,Par:07} for the general case). 

We now define isotropic random walks on the set~$V$ of all special vertices, and outline the algebraic and analytic theory that is used to analyse them. In Appendix~\ref{app:A} we will give more details in the specific case of $\widetilde{C}_2$ buildings, where one can carry out the calculations `by hand'.

\begin{defn}
A random walk $(X_n)_{n\geq 0}$ on $V$ is \textit{isotropic} if its transition probabilities satisfy
$$
p(x,y)=p(x',y')\quad\text{whenever $\vect{d}(x,y)=\vect{d}(x',y')$}.
$$
\end{defn}

For each $x\in V$ and $\lambda\in P^+$ Let
$$
V_{\lambda}(x)=\{y\in V\mid \vect{d}(x,y)=\lambda\}\quad\text{be the sphere of `radius' $\lambda$ centred at $x$}.
$$
The cardinality $N_{\lambda}=|V_{\lambda}(x)|$ does not depend on $x\in V$
(see \cite[Proposition~1.5]{Par:06b}). If $(X_n)_{n\geq 0}$ is an isotropic random walk on $V$ then there are numbers $a_{\lambda}\geq 0$ with $\sum_{\lambda\in P^+}a_{\lambda}=1$ such that
\begin{align}\label{eq:a}
p(x,y)=\frac{a_{\lambda}}{N_{\lambda}}\quad\text{for all $y\in V_{\lambda}(x)$}.
\end{align}
In an analogous way to the case of isotropic random walks on chambers (see Proposition~\ref{prop:transition}) the transition operator $A$ of an isotropic random walk on the vertices of a regular affine building is of the form
\begin{align}\label{eq:atrans}
A=\sum_{\lambda\in P^+}a_{\lambda}A_{\lambda},
\end{align}
where the numbers $a_{\lambda}$ are as in (\ref{eq:a}) and the operator $A_{\lambda}$ acts on functions $f:V\to\CC$ by
$$
A_{\lambda}f(x)=\frac{1}{N_{\lambda}}\sum_{y\in V_{\lambda}(x)}f(y).
$$

Let $\mathscr{A}$ be the vector space over $\CC$ with basis $\{A_{\lambda}\mid \lambda\in P^+\}$. The following is an analogue of Theorem~\ref{thm:product} (the proof is, however, a little more involved). 

\begin{thm}\label{thm:com}\emph{\cite[Theorem~5.24]{Par:06b}}
The vector space $\scA$ is a commutative associative unital algebra under composition of linear operators.
\end{thm}

The algebra $\scA$ plays an important role in understanding isotropic random walks on the vertices of affine buildings. The key feature of Theorem~\ref{thm:com} is that this algebra is commutative. In fact one can be more precise.

\begin{thm}\label{thm:centre}\emph{\cite[Theorem~6.16]{Par:06b}}
Let $\Sigma$ be a regular affine building. Let $\scP$ the the algebra of chamber set averaging operators on $\Sigma$ (c.f. Theorem~\ref{thm:product}) and let $\scA$ be the algebra of vertex set averaging operators on~$\Sigma$ (c.f. Theorem~\ref{thm:com}). Then $\scA$ is isomorphic to the centre of $\scP$. 
\end{thm}

For each $\alpha\in R$ we write $q_{\alpha}=q_i$ if $\alpha\in W_0\alpha_i$, and let
$$
r^{\lambda}=\prod_{\alpha\in R^+}q_{\alpha}^{\frac{1}{2}\langle\lambda,\alpha\rangle}\quad\text{for all $\lambda\in P$}.
$$
If $q_i=q$ for all $i=0,1,\ldots,d$ then $r^{\lambda}=q^{\ell(t_{\lambda})/2}$ where $t_{\lambda}$ is the translation by~$\lambda$.

If $u\in\mathrm{Hom}(P,\CC^{\times})$ we write $u^{\lambda}=u(\lambda)$, and if $w\in W_0$ and $u\in\mathrm{Hom}(P,\CC^{\times})$ let $wu\in \mathrm{Hom}(P,\CC^{\times})$ be given by $(wu)^{\lambda}=u^{w\lambda}$ for all $\lambda\in P$. For each $\lambda\in P^+$ the \textit{Macdonald spherical function} $P_{\lambda}$ is the function $P_{\lambda}:\mathrm{Hom}(P,\CC^{\times})\to\CC$ given by
\begin{align*}
P_{\lambda}(u)=\frac{r^{-\lambda}}{W_0(q^{-1})}\sum_{w\in W_0}u^{w\lambda}c(wu)\quad\text{where}\quad c(u)=\prod_{\alpha\in R^+}\frac{1-q_{\alpha}^{-1}u^{-\alpha^{\vee}}}{1-u^{-\alpha^{\vee}}},
\end{align*}
where $W_0(q^{-1})=\sum_{w\in W_0}q_w^{-1}$. This formula requires, of course, that the denominators are nonzero, however it turns out that $P_{\lambda}(u)$ is a linear combination of terms $u^{\mu}$ with $\mu\in P$ and so the `singular' cases where a denominator vanishes can be obtained by taking an appropriate limit in the general formula. The Macdonald spherical functions arise in the representation theory of $p$-adic groups (see~\cite{Mac:71}).

Theorem~\ref{thm:centre}, combined with the \textit{Satake isomorphism}, implies the following result, giving a complete description of the irreducible representations of~$\scA$. 

\begin{thm}\emph{\cite[Proposition~2.1]{Par:06b}}
For each $u\in\mathrm{Hom}(P,\CC^{\times})$ there is a $1$-dimensional representation $\pi_u$ of $\scA$ given by $\pi_u(A_{\lambda})=P_{\lambda}(u)$. Moreover, every $1$-dimensional representation $\pi$ of $\scA$ is of the form $\pi=\pi_u$ for some $u\in\mathrm{Hom}(P,\CC^{\times})$, and $\pi_u=\pi_{u'}$ if and only if $u'=wu$ for some $w\in W_0$. 
\end{thm}

Each $A\in\scA$ maps $\ell^2(V)$ into itself, and $\|A_{\lambda}f\|_2\leq\|f\|_2$. Thus we may regard $\scA$ as a subalgebra of the $C^*$-algebra $\mathscr{L}(\ell^2(V))$ of bounded linear operators on $\ell^2(V)$. It is not hard to see that $A_{\lambda}^*=A_{\lambda^*}$ where $\lambda^*=-w_0\lambda$ (with $w_0$ the longest element of $W_0$), and thus $\scA$ is closed under taking adjoints. Let $\scA_2$ be the completion of $\scA$ with respect to the $\ell^2$-operator norm~$\|\cdot\|$. Thus $\scA_2$ is a commutative $C^*$-algebra. By passing to this completion we ensure that the transition operator $A$ of an isotropic random walk on~$V$ is an element of $\scA_2$ (it is an element of the `uncompleted' algebra $\scA$ if and only if the walk has bounded range).

The $1$-dimensional representations of $\scA_2$ are precisely the extensions to $\scA_2$ of the representations $\pi_u:\scA\to\CC$ which are continuous with respect to the $\ell^2$-operator norm, and in \cite[\S5]{Par:06b} it is shown that these are the representations $\pi_u$ with $u\in\mathrm{Hom}(P,\TT)$ where $\TT=\{z\in\CC\mid |z|=1\}$. If $u\in\mathrm{Hom}(P,\TT)$ and $A\in\scA_2$ we write $\widehat{A}(u)=\pi_u(A)$ (the \textit{Galfand transform} of $A$). In particular, we have $\widehat{A}_{\lambda}(u)=P_{\lambda}(u)$, and if $A$ is the transition operator of an isotropic random walk as in~(\ref{eq:atrans}) we have
$$
\widehat{A}(u)=\sum_{\lambda\in P^+}a_{\lambda}P_{\lambda}(u)\quad\text{for all $u\in\mathrm{Hom}(P,\TT)$}. 
$$ 

The final ingredient in the analysis of $\scA_2$ is the calculation of the \textit{Plancherel measure}. 

\begin{thm}\label{thm:plangen}\emph{\cite[Theorem~5.1.5]{Mac:71}, \cite[Theorem~5.2]{Par:06b}} Let $du$ denote normalised Haar measure on $\mathbb{U}=\mathrm{Hom}(P,\TT)$, and let $\mu$ be the measure on $\mathbb{U}$ given by
$$
d\mu(u)=\frac{W_0(q^{-1})}{|W_0|}\frac{1}{|c(u)|^2}\,du.
$$
Then
$$
\frac{1}{N_{\lambda}}\int_{\mathbb{U}}\widehat{A}_{\lambda}(u)\overline{\widehat{A}_{\lambda'}(u)}\,d\mu(u)=\delta_{\lambda,\lambda'}\quad\text{for all $\lambda,\lambda'\in P^+$}.
$$
\end{thm}

Theorem~\ref{thm:plangen} implies that the $n$-step transition probabilities of an isotropic random walk with transition operator~$A$ are given by
\begin{align}\label{eq:rec1}
p^{(n)}(x,y)=\frac{1}{N_{\lambda}^2}\int_{\mathbb{U}}\widehat{A}(u)^n\overline{\widehat{A}_{\lambda}(u)}d\mu(u)\quad\text{if $\vect{d}(x,y)=\lambda$}.
\end{align}
This is the analogue of Theorem~\ref{thm:spherical1}, and is a key result in studying isotropic random walks on the vertices of affine buildings. Indeed the primary limit theorems  (that is, the law of large numbers, the central limit theorem, and the local limit theorem) can all be proven using the above machinery via techniques from classical harmonic analysis. These limit theorems were proved by Lindlebauer and Voit~\cite{LV:02} for the case of $\widetilde{A}_2$ buildings, and Cartwright and Woess~\cite{CW:04} for the case of $\widetilde{A}_n$ buildings. The general case was settled by Parkinson~\cite{Par:07}, and the results are summarised below.

\begin{thm}\label{thm:lln and clt}\emph{\cite{Par:07}}
Let $(X_n)_{n\geq 0}$ be an isotropic random walk on the vertices of a locally finite thick regular affine building. 
\begin{enumerate}
\item[\emph{(1)}] Under the moment assumption $\sum_{\lambda\in P^+}|\lambda|a_{\lambda}<\infty$ there exists $\gamma\in E^+$ such that
$$
\lim_{n\to\infty}\frac{\vect{d}(o,X_n)}{n}=\gamma\quad\text{almost surely}.
$$
\item[\emph{(2)}] Under the moment assumption $\sum_{\lambda\in P^+}|\lambda|^2a_{\lambda}<\infty$, the vector 
$$(\vect{d}(o,X_n)-n\gamma)/\sqrt{n}$$ converges in distribution to the multivariable normal distribution $N(0,\Gamma)$, where $\Gamma$ is a positive definite matrix. 
\item[\emph{(3)}] Let $y\in V_{\lambda}(x)$ and $n\in\NN$. Suppose that $(X_n)_{n\geq 0}$ is irreducible and aperiodic. Then
$$
p^{(n)}(x,y)=CP_{\lambda}(1)\widehat{A}(1)^nn^{-(|R|+d)/2}\left(1+O(n^{-1/2})\right),
$$
where $C>0$ is an explicit constant. 
\end{enumerate}
\end{thm}

\begin{remark}
We remark that the precision of Theorem~\ref{thm:lln and clt} is really quite impressive, with explicit formulae for the speed, variance, radius of convergence, and all asymptotic constants (see \cite{Par:07} for details). In the local limit theorem the assumption of aperiodicity may be removed, see~\cite{Par:07}.
\end{remark}

Heat kernel and Green function estimates for finite range isotropic random walks on affine buildings have been obtained recently by Trojan~\cite{Tro:13} (with earlier results obtained for $\widetilde{A}_n$ buildings by Anker, Schapira and Trojan~\cite{AST:06}). The starting point for this analysis is again formula~(\ref{eq:rec1}). Estimates for the Green function are given within the radius of convergence, and at the radius of convergence. For example, at the radius of convergence Trojan proves:

\begin{thm} \emph{\cite[Theorem~7]{Tro:13}}
The green function of a finite range isotropic random walk on the special vertices of an affine building of rank~$r$, evaluated at the radius of convergence, satisfies
$$
\sum_{n= 0}^{\infty}p^{(n)}(x,y)\rho^{-n}\asymp P_{\lambda}(1)\|\lambda\|^{2-r-2|R^+|}
$$
\end{thm}

Finally, convergence results for isotropic random walks on affine buildings to Brownian motion in a Weyl sector have been studied by Schapira, at least in the context of nearest neighbour random walks on $\widetilde{A}_r$ buildings. We now describe this result. Let $(X_n)_{n\geq 0}$ be a symmetric nearest neighbour random walk on the vertices of an $\widetilde{A}_r$ building with transition probabilities $p(x,y)$. Let $\rho$ be the spectral radius of~$(X_n)_{n\geq 0}$ and let $(Y_n)_{n\geq 0}$ be the random walk with transition probabilities
$$
q(x,y)=p(x,y)\frac{P_{\vect{d}(o,y)}(1)}{P_{\vect{d}(o,x)}(1)}\rho^{-1}
$$
(it is easily seen that this defines a random walk on the building using \cite[Theorem~3.22]{Par:06b}). Schapira proves the following (see \cite[\S~2]{Sch:09} for the relevant definitions of Brownian motion):

\begin{thm}\emph{\cite[Theorem~6.1]{Sch:09}} With the notation as above, the sequence $(Z_t^n)_{t\geq 0}$ with
$$
Z_t^n=\frac{1}{\sqrt{n}}\vect{d}\left(o,Y_{\lceil nt\rceil}\right)
$$
converges in law to Brownian motion $(I_t)_{t\geq 0}$ in the sector $\fs_0$ as $n\to\infty$. 
\end{thm}

\subsubsection{Random walks on the chambers of an affine building}

The literature on isotropic random walks on the chambers of affine buildings is currently less complete. A key reason for this is that the algebra $\mathscr{P}$ of averaging operators on the chambers of an affine building is noncommutative. Thus the Plancherel Theorem for this infinite dimensional noncommutative algebra is rather sophisticated (see \cite{Opd:04} and \cite{Par:14}). The general approach to the primary limit theorems is outlined by Parkinson and Schapira in \cite{PS:11}, and the detailed calculations are carried through for $\widetilde{A}_2$ buildings. The general case is in preparation by the author.

\begin{thm}\emph{\cite[Theorem~3.7]{PS:11}}
For the simple random walk on the chambers of a thick $\widetilde{A}_2$ building with thickness~$q>1$ we have 
$$
p^{(n)}(x,y)=C_w\rho^nn^{-4}\left(1+O(n^{-1/2})\right)\quad\text{if $\delta(x,y)=w$}
$$
where $C_w$ is an explicitly computable constant (depending on $w$ and $q$ only), and where the spectral radius $\rho$ is given by
$
\rho=(3(q-1)+\sqrt{q^2+34q+1})/6q.
$
\end{thm}

\begin{remark}
Assuming a suitably transitive group action, a formula for the spectral radius for an isotropic random walk on the chambers of an affine building can be deduced from results of Saloff-Coste and Woess~\cite{SCW:97}. In particular, see \cite[Example~6]{SCW:97}.
\end{remark}

\subsubsection{Regular sequences in affine buildings}

Recently Parkinson and Woess~\cite{PW:14} proved the ``$p$-adic analogue'' of Kaimanovich's characterisation~\cite{Kai:87} of \textit{regular sequences} in symmetric spaces. This theory has applications to random walks on buildings and associated groups, and we describe this here. Recall the definition of the vector Busemann functions $\vect{h}_{\fs}$ from Definition~\ref{defn:buse}. 

It is convenient to work with a natural `metric realisation' of the affine building $\Sigma$. By the construction in Section~\ref{sec:root} we may regard the apartments of an affine building as tessellations of a Euclidean space, and thus there is a metric on each apartment. Using axioms (B2) and (B3) it can be shown that these metrics may be `glued together' to make $\Sigma$ into a metric space (see \cite[\S11.2]{AB:08}). By \cite[Theorem~11.16]{AB:08} this metric space is a $\mathrm{CAT}(0)$ space. In this section we will regard affine buildings as metric spaces, although we also remember the underlying simplicial complex structure. The vector distance and the Busemann functions (originally only defined for vertices) naturally extend to give a vector distance and Busemann function for any points $x,y$ of the building (see \cite{PW:14} for details).

Let $\lambda\in E^+$. A \textit{$\lambda$-ray} in $\Delta$ is a function $\fr:[0,\infty)\to\Delta$ such that
$$
\vect{d}(\fr(t_1),\fr(t_2))=(t_2-t_1)\lambda\quad\text{for all $t_2\geq t_1\geq 0$}.
$$
Since we are specifying both a speed and direction, the notion of a $\lambda$-ray is a refinement of the usual notion of a ray in a $\mathrm{CAT}(0)$ space. 

\begin{thm}\emph{\cite[Theorem~3.2]{PW:14}}\label{thm:PW14}
Let $(x_n)_{n\geq 0}$ be a sequence in $\Delta$, and let $\lambda\in E^+$. Let $\fs$ be a sector of $\Delta$. The following are equivalent:
\begin{enumerate}
\item[(1)] There is a $\lambda$-ray $\fr:[0,\infty)\to\Delta$ such that $d(x_n,\fr(n))=o(n)$.
\item[(2)] $d(x_n,x_{n+1})=o(n)$ and $\vect{h}_{\fs}(x_n)=n\mu_{\fs}+o(n)$ for some $\mu_{\fs}\in W_0\lambda$ (independent of $n$).
\item[(3)] $d(x_n,x_{n+1})=o(n)$ and $\vect{d}(o,x_n)=n\lambda+o(n)$. 
\end{enumerate}
\end{thm}

A sequence $(x_n)_{n\geq 0}$ satisfying any one of the above equivalent conditions is called a \textit{$\lambda$-regular sequence}. This is a direct analogue of Kaimanovich's results on symmetric spaces~\cite{Kai:87}. We will discuss applications of Theorem~\ref{thm:PW14} to random walks on affine buildings and associated groups in this section and the next.

Since $\Delta$ is $\mathrm{CAT}(0)$ we define the \textit{visibility boundary} $\partial \Delta$ 
in the usual way as the set of equivalence classes of rays (with two rays being 
\textit{equivalent} if the distance between them is bounded). The standard topology makes 
$\overline{\Delta}=\Delta\cup\partial \Delta$ into a compact Hausdorff space 
(see Bridson and Haefliger \cite[\S II.8.5]{BH:99}). Points of the visibility boundary are called 
\textit{ideal points} of~$\Delta$. Given $\xi\in\partial \Delta$ and $x\in \Delta$, 
there is a unique ray in the class $\xi$ with base point $x$ (\cite[Proposition~II.8.2]{BH:99} 
or \cite[Lemma~11.72]{AB:08}). We sometimes denote this ray by $[x,\xi)$. Thus one may think 
of $\partial \Delta$ as ``all rays based at $x$'' for any fixed~$x\in\Delta$.

\begin{defn} A random walk on $V$ is \textit{semi-isotropic} if the transition probabilities of the walk depend only on the vectors $\vect{d}(x,y)$ and $\vect{h}(y)-\vect{h}(x)$. 
\end{defn}

Clearly isotropic random walks are semi-isotropic, but not vice-versa. For each $\lambda\in P$ let $H_{\lambda}=\{x\in V\mid \vect{h}(x)=\lambda\}$. As shown in \cite[Proposition~4.6]{PW:14} semi-isotropic random walks are `factorisable' over~$P$, in the sense that the value of the sum
$$
\overline{p}(\lambda,\mu)=\sum_{y\in H_{\mu}}p(x,y)\quad\text{with $\lambda,\mu\in P$ and $x\in H_{\lambda}$}
$$
does not depend on the particular $x\in H_{\lambda}$ chosen. Moreover, we have
$$
\overline{p}(\lambda+\nu,\mu+\nu)=\overline{p}(\lambda,\mu)\quad\text{for all $\lambda,\mu,\nu\in P$}.
$$
In other words, if $(X_n)_{n\geq 0}$ is semi-isotropic then the sequence $\vect{h}(X_n)\in P$ is a translation invariant random walk on~$P$ with transition probabilities $\overline{p}(\lambda,\mu)$. Since $P\cong \ZZ^d$ the random walk $(\vect{h}(X_n))_{n\geq 0}$ is well understood from the classical theory, and using Theorem~\ref{thm:PW14} we obtain the following result for the original random walk $(X_n)_{n\geq 0}$ on the building.

\begin{thm}\emph{\cite[Corollary 4.8]{PW:14}}
Let $(X_n)_{n\geq 0}$ be a semi-isotropic random walk on $V$. Under the finite first moment assumption $\sum_{\nu\in P}\overline{p}(0,\nu)|\nu|<\infty$ we have
$$
\lim_{n\to\infty}\frac{1}{n}\vect{d}(o,X_n)=\lambda\quad\text{almost surely},
$$
where $\lambda$ is the dominant element in the $W_0$-orbit of $\mu=\sum_{\nu\in P}\overline{p}(0,\nu)\nu$. Moreover, if $\lambda\neq 0$ then $(X_n)_{n\geq 0}$ converges almost surely to an ideal point~$X_{\infty}$.
\end{thm}

The drift-free case (when $\lambda=0$) is more subtle. A weaker form of convergence of the random walk in this case is established in \cite[Theorem~4.15]{PW:14} for nearest neighbour random walks.

\subsubsection{Random walks on groups acting on affine buildings}

Limit theorems for isotropic random walks on the vertices of affine buildings imply limit theorems for bi-$K$-invariant probability measures on groups acting sufficiently transitively on the building, where $K$ is the stabiliser of a fixed (special) vertex of the building. This is completely analogous to the chamber case of Proposition~\ref{prop:walkgp1} (see \cite[Remark~2.19]{Par:07} for some details). For example, If $G=G(\QQ_p)$ is a Chevalley group over the $p$-adic numbers, and if $K=G(\ZZ_p)$ with $\ZZ_p$ the ring of $p$-adic integers, then Theorem~\ref{thm:lln and clt} gives a local limit theorem for the density function of a bi-$K$-invariant probability measure on~$G$ (see Example~\ref{ex:padic}).

In the prototypical example of $G=SL_{d+1}(\QQ_p)$ one can remove the bi-$K$-invariance assumption, at the cost of losing explicit formulae for the spectral radius. We expect the following result of Tolli to hold for more general Lie types, however at present it is only available for~$SL_{d+1}$. 

\begin{thm}\emph{\cite{Tol:01}}
Let $G=SL_{d+1}(\FF)$ where $\FF$ is a local field. Let $f$ be a continuous compactly supported density of a probability measure on $G$ such that
\begin{enumerate}
\item[\emph{(1)}] $f$ is symmetric, that is $f(x)=f(x^{-1})$, and 
\item[\emph{(2)}] the support of $f$ is a neighbourhood of the identity that generates~$G$. 
\end{enumerate}
Then there exists a number $\rho>0$ and a positive function $\psi:G\to\RR_{\geq 0}$ such that
$$
\rho^{-n}n^{d(d+2)/2}f^{(*n)}\to \psi\quad\text{pointwise as $n\to\infty$}.
$$
\end{thm}

Let $\Sigma$ be a regular affine building, and let $G$ be a subgroup of the automorphism group~$\mathrm{Aut}(\Sigma)$. Let $\sigma$ be a Borel probability measure on $G$, such that the support of $\sigma$ generates $G$. We say that $\sigma$ has \textit{finite first moment} if
$$
\int_Gd(o,go)\,d\sigma(g)<\infty.
$$
Let $(g_n)_{n\geq 0}$ be a stationary sequence of $G$-valued random variables with joint distribution $\sigma$ The \textit{right random walk} is the sequence $(X_n)_{n\geq 0}$ with 
$$
X_0=o\quad\text{and}\quad X_n=g_1\cdots g_no\quad\text{for $n\geq 1$}.
$$
The theory of regular sequences (Theorem~\ref{thm:PW14}) implies the following result for the right random walk.

\begin{thm}\emph{\cite[Theorem~4.1]{PW:14}} Let $G$ and $\sigma$ be as above, and suppose that $\sigma$ has finite first moment. Let $(X_n)_{n\geq 0}$ be the associated right random walk on $\Sigma$. There exists $\lambda\in E^+$ such that
$$
\lim_{n\to\infty}\frac{1}{n}\vect{d}(o,X_n)=\lambda\quad\text{almost surely},
$$
and for each sector $\fs$ of $\Sigma$ there exists $\mu_{\fs}\in W_0\lambda$ such that
$$
\lim_{n\to\infty}\frac{1}{n}\vect{h}_{\fs}(X_n)=\mu\quad\text{almost surely}.
$$
If $\lambda\neq 0$ then $(X_n)_{n\geq 0}$ converges almost surely to an ideal point~$X_{\infty}$.  
\end{thm}

\subsection{Random walks on Fuchsian buildings}\label{sec:fuchsian}

Probability theory for buildings and related groups of non-spherical, non-affine type is in its infancy. Recently isotropic random walks on the chambers of Fuchsian buildings have been studied by Gilch, M\"{u}ller and Parkinson, and a law of large numbers and a central limit theorem have been obtained. In this case the Hecke algebra has less controllable representation theory than in the spherical and affine cases, owing partly to the existence of free group subgroups in the Coxeter systems). Thus the representation theoretic techniques that have worked so nicely for the spherical and affine cases do not seem to help.

Instead the arguments rely much more heavily on the underlying hyperbolic geometry of the building and the planarity of its apartments, with the general ideas adapted from the work of Ha\"{i}ssinski, Mathieu and M\"{u}ller \cite{HMM:13}. In this work the planarity of the and hyperbolicity of the Cayley graph of a surface group are exploited to develop a `renewal theory' related to the automata structure of the group. In the setting of Fuchsian buildings the apartments of the building are planar and hyperbolic, and so similar ideas can be applied to the apartments. To lift this to the entire building requires some more work, and in \cite{GMP:14} a theory of cones, cone types, and automata for Fuchsian buildings paralleling the more familiar notions in groups is developed to achieve this goal. The idea is to find a decomposition of the trajectory of the walk into aligned pieces in such a way that these pieces are independent and identically distributed. Roughly speaking, one fixes a recurrent cone type $\mathbf{T}$ and sets $R_1$ to be the first time that the walk visits a cone of type $\mathbf{T}$ and never leaves this cone again. Inductively one defines $R_{n+1}$ to be the first time after $R_n$ that the walk enters a cone of type $\mathbf{T}$ and never leaves it again. The main results of \cite{GMP:14} are as follows.

\begin{thm}\label{thm:llnfuchsian}
Let $\Sigma$ be a regular Fuchsian building and let $(X_n)_{n\geq 0}$
be an isotropic random walk on~$\Delta$ with bounded range. Then,
\begin{align*}
\frac{1}{n} d(o,X_{n}) \stackrel{a.s.}{\longrightarrow} v=  \frac{\EE[ d(X_{ R_{2}}, X_{ R_{1}})]}{\EE[  R_{2}- R_{1}]}>0~\mbox{ as }  n\to\infty.
\end{align*}
\end{thm}

\begin{thm}\label{thm:cltfuchsian}
Let $\Sigma$ be a regular Fuchsian building and let $(X_n)_{n\geq 0}$ be an isotropic random walk on~$\Delta$ with bounded range. Then, with $v$ as in Theorem~\ref{thm:llnfuchsian},
$$ \frac{d(o,X_{n}) -nv }{\sqrt{n}} \stackrel{\cD}{\longrightarrow} \cN (0, \sigma^2),\quad\text{where}\quad \sigma^{2}=\frac{\EE[(d(X_{R_{2}}, X_{R_{1}}) - (R_{2}- R_{1})v)^{2}]}{\EE[R_{2}-R_{1}]}.$$
\end{thm} 

\subsection{Future directions}

We conclude the main body of this paper by listing some future directions and open problems in the theory of random walks on buildings:
\begin{enumerate}
\item[(1)] Provide sharp mixing time estimates and establish cut-off phenomenon for natural random walks on spherical buildings (in particular, for the simple random walk). 
\item[(2)] Prove a law of large numbers, a central limit theorem, and a local limit theorem for isotropic random walks on the chambers of affine buildings (generalising \cite{PS:11}). 
\item[(3)] Establish a local limit theorem for $p$-adic Lie groups of general type (generalising \cite{Tol:01}).
\item[(4)] Prove convergence properties for the right random walk on a group acting on an affine building in the drift free case (c.f. \cite{PW:14}). 
\item[(5)] Give an explicit formula for the spectral radius of a random walk on a Fuchsian building or Coxeter group (or any other non-spherical non-affine building or Coxeter group). There are some trivial `tree-like' examples, although apart from these no explicit formulae are known. Efficient algorithms, or asymptotic formulae in the thickness parameter, would also be interesting in lieu of an explicit formula. 
\item[(6)] Prove a precise and explicit local limit theorem for a non-spherical, non-affine building.
\item[(7)] Derive heat kernel and Green function estimates for random walks on the chambers of affine buildings (extending \cite{Tro:13}). 
\item[(8)] Generalise the Brownian motion convergence results of \cite{Sch:09} to arbitrary type. As a first step one might consider either the rank~$2$ cases, or remove the nearest neighbour restriction from~\cite{Sch:09} for walks on $\widetilde{A}_n$ buildings. 
\end{enumerate}

\begin{appendix}

\section{Isotropic random walks on the vertices of a $\widetilde{C}_2$ building}\label{app:A}

In this appendix we carry out the details of the outline given in Section~\ref{sec:vertices} in the special case of an affine building of type~$\widetilde{C}_2$. These calculations were made by the author some years ago, in collaboration with Donald Cartwright, following the calculations made in the $\widetilde{A}_2$ case by Cartwright and M{\l}otkowski~\cite{CM:94}. The calculations here are very much `hands-on', and do not require as much machinery as the general case. See the author's thesis for some further calculations for $\widetilde{G}_2$ buildings and so called $\widetilde{BC}_2$ buildings.

Let $\Sigma$ be a building of type $\widetilde{C}_2$ with thickness parameters $q_0=q_2=q$ and $q_1=r$. Let $V$ be the set of all special vertices of~$\Sigma$. The root system and fundamental coweights~$\omega_1$ and $\omega_2$ are illustrated in Figure~\ref{fig:affinerank3}(b). If $\lambda=k\omega_1+l\omega_2$ we write $V_{\lambda}(x)=V_{k,l}(x)$. 

\begin{lemma}\label{lem:N} The cardinalities $N_{k,l}=|V_{k,l}(x)|$ do not depend on $x\in V$, and we have
\begin{align*}
N_{k,l}&=(q+1)(r+1)(qr+1)q^2r(q^2r^2)^{k-1}(q^2r)^{l-1}\\
N_{k,0}&=(r+1)(qr+1)q(q^2r^2)^{k-1}\\
N_{0,l}&=(q+1)(qr+1)(q^2r)^{l-1}.
\end{align*}
\end{lemma}

Let us illustrate Lemma~\ref{lem:N} with an examples (the general argument is an induction). Figure~\ref{fig:C2count} shows part of an apartment of $\Sigma$. Panels with thickness $q$ are shown as solid lines, and panels with thickness $r$ are shown as dashed lines. The vertex $y$ is in $V_{1,2}(x)$.

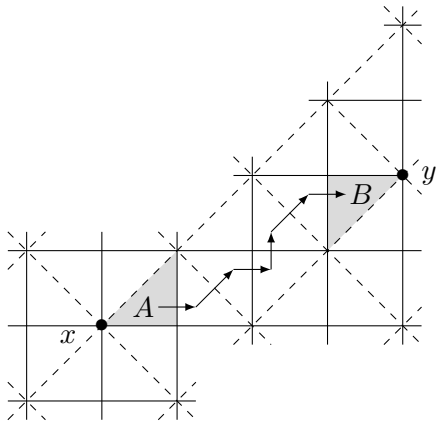
\begin{figure}[!h]
\centering
\begin{tikzpicture}[scale=1]
\path [fill=lightgray!70] (0,0) -- (1,0) -- (1,1) -- (0,0);
\path [fill=lightgray!70] (3,2) -- (4,2) -- (3,1) -- (3,2);
\draw (-1.25,0)--(4.25,0);
\draw (-1.25,1)--(4.25,1);
\draw (1.75,2)--(4,2);
\draw (2.75,3)--(4.25,3);
\draw (3.75,4)--(4.25,4);
\draw (-1.25,-1)--(1.25,-1);
\draw (4,-0.25)--(4,4.25);
\draw (3,-0.25)--(3,3.25);
\draw (2,-0.25)--(2,2.25);
\draw (1,-1.25)--(1,1.25);
\draw (0,-1.25)--(0,1.25);
\draw (-1,-1.25)--(-1,1.25);
\draw [style = dashed] (-1.25,-1.25)--(4.25,4.25);
\draw [style=dashed] (-1.25,0.75)--(-0.75,1.25);
\draw [style=dashed] (0.75,-1.25)--(1.25,-0.75);
\draw [style=dashed] (1.75,-0.25)--(4.25,2.25);
\draw [style=dashed] (3.75,-0.25)--(4.25,0.25);
\draw [style=dashed] (-1.25,-0.75)--(-0.75,-1.25);
\draw [style=dashed] (-1.25,1.25)--(1.25,-1.25);
\draw [style=dashed] (0.75,1.25)--(2.25,-0.25);
\draw [style=dashed] (1.75,2.25)--(4.25,-0.25);
\draw [style=dashed] (3,3)--(4.25,1.75);
\draw [style=dashed] (3.75,4.25)--(4.25,3.75);
\draw [-latex] (0.75,0.25)--(1.25,0.25);
\draw [-latex] (1.25,0.25)--(1.75,0.75);
\draw [-latex] (1.75,0.75)--(2.25,0.75);
\draw [-latex] (2.25,0.75)--(2.25,1.25);
\draw [-latex] (2.25,1.25)--(2.75,1.75);
\draw [-latex] (2.75,1.75)--(3.25,1.75);
\node at (0,0) {$\bullet$};
\node at (4,2) {$\bullet$};
\node at (4.35,2) {$y$};
\node at (-0.45,-0.15) {$x$};
\node at (0.55,0.25) {$A$};
\node at (3.45,1.75) {$B$};
\end{tikzpicture}
\caption{Computing the cardinalities $|V_{k,l}(x)|$}\label{fig:C2count}
\end{figure}

Let $A$ be a chamber of the building containing $x$. There are $q^4r^2$ galleries in the building starting at $A$ and ending at a chamber in position $B$, and each of these end chambers contains a vertex in $V_{1,2}(x)$. Moreover, every vertex in $V_{1,2}(x)$ can be reached by such a gallery starting at some chamber containing $x$, and different starting chambers result in different end vertices in $V_{1,2}(x)$. Thus $|V_{1,2}(x)|=Kq^4r^2$, where $K$ is the number of chambers containing~$x$. The set of chambers containing $x$ is a spherical building of type $C_2$ (that is, a generalised quadrangle) with parameters $(q,r)$. Thus $K=(q+1)(r+1)(qr+1)$, and hence $N_{1,2}=(q+1)(r+1)(qr+1)q^4r^2$.

For each pair $k,l\geq 0$ define an operator $A_{k,l}$ acting on functions $f:V\to\CC$ by
$$
A_{k,l}f(x)=\frac{1}{N_{k,l}}\sum_{y\in V_{k,l}(x)}f(y).
$$
Every isotropic random walk $(X_n)_{n\geq 0}$ on $V$ has transition operator $A$ of the form
$$
A=\sum_{k,l\geq 0}a_{k,l}A_{k,l}
$$
with $a_{k,l}\geq 0$ and $\sum_{k,l\geq 0}a_{k,l}=1$. Explicitly, $a_{k,l}=\PP[X_{n+1}\in V_{k,l}(x)\mid X_n=x]=p(x,y)/N_{k,l}$ for any $y\in V_{k,l}(x)$. 

\begin{thm}\label{thm:recursionC2} The following formulae hold, where in each case the indices $m,n$ are required to be large enough to ensure that the indices appearing on the right are all at least~$0$.
\begin{align*}
A_{1,0}A_{0,1}&=A_{0,1}A_{1,0}\\
N_{1,0}A_{m,n}A_{1,0}&=rA_{m+1,n-2}+(q-1)(r+1)A_{m,n}+q^2r^2A_{m+1,n}+q^2rA_{m-1,n+2}+A_{m-1,n}\\
N_{1,0}A_{0,n}A_{1,0}&=(r+1)A_{1,n-2}+(q-1)(r+1)A_{0,n}+q^2r(r+1)A_{1,n}\\
N_{1,0}A_{m,0}A_{1,0}&=qr(q+1)A_{m-1,2}+q^2r^2A_{m+1,0}+(q-1)A_{m,0}+A_{m-1,0}\\
N_{1,0}A_{m,1}A_{1,0}&=q^2r^2A_{m+1,1}+q^2rA_{m-1,3}+A_{m-1,1}+(qr+q-1)A_{m,1}\\
N_{0,1}A_{m,n}A_{0,1}&=A_{m,n-1}+qr A_{m+1,n-1}+qA_{m-1,n+1}+q^2rA_{m,n+1}\\
N_{0,1}A_{0,n}A_{0,1}&=A_{0,n-1}+q^2rA_{0,n+1}+q(r+1)A_{1,n-1}\\
N_{0,1}A_{m,0}A_{0,1}&=(q+1)A_{m-1,1}+qr(q+1)A_{m,1}.
\end{align*}
\end{thm}

\begin{proof}
From the definition of the operators $A_{m,n}$ we have
\begin{align}\label{eq:productform}
A_{m,n}A_{s,t}f(x)=\sum_{u,v\geq 0}\left(\frac{1}{N_{u,v}}\sum_{y\in V_{u,v}(x)}\frac{N_{u,v}}{N_{m,n}N_{s,t}}|V_{m,n}(x)\cap V_{s,t}(y)|f(y)\right)
\end{align}
(see \cite[(3.1)]{Par:06a} for some intermediate steps). The formulae in the theorem follow by computing the cardinalities $|V_{m,n}(x)\cap V_{s,t}(y)|$ in the cases $(s,t)=(1,0)$ and $(s,t)=(0,1)$. We will give the calculation for $(s,t)=(0,1)$ and $m,n\geq 1$, leaving the remaining cases as an exercise. 

\begin{figure}[!h]
\centering
\begin{tikzpicture}[scale=0.35]
\draw (0,0)--(15,0);
\draw [style=dashed] (0,0)--(10,10);
\draw (10,2.5)--(10,7.5);
\draw (12,2.5)--(12,7.5);
\draw (14,2.5)--(14,7.5);
\draw (9.5,3)--(14.5,3);
\draw (9.5,5)--(14.5,5);
\draw (9.5,7)--(14.5,7);
\draw [style = dashed] (9.5,2.5)--(14.5,7.5);
\draw [style = dashed] (9.5,7.5)--(14.5,2.5);
\draw [style = dashed] (9.5,6.5)--(10.5,7.5);
\draw [style = dashed] (13.5,2.5)--(14.5,3.5);
\draw [style = dashed] (9.5,3.5)--(10.5,2.5);
\draw [style = dashed] (13.5,7.5)--(14.5,6.5);
\draw (4.75,5)--(5.25,5);
\draw (6.8,-0.2)--(7.2,0.2);
\node at (0,0) {$\bullet$};
\node at (12,5) {$\bullet$};
\node at (4.25,5) {$k$};
\node at (6.6,-0.75) {$l$};
\node [color=white] at (10,7) {$\bullet$};
\node [color=white] at (10,3) {$\bullet$};
\node [color=white] at (14,3) {$\bullet$};
\node [color=white] at (14,7) {$\bullet$};
\node at (10,7) {$\circ$};
\node at (10,3) {$\circ$};
\node at (14,3) {$\circ$};
\node at (14,7) {$\circ$};
\node at (12.25,5.75) {$y$};
\node at (-0.5,0) {$x$};
\draw [style=dotted] (7,0)--(9.5,2.5);
\draw [style=dotted] (5,5)--(9.5,5);
\draw (2,0)--(2,2);
\draw (2,2)--(4.5,2);
\draw (4,0)--(4,4);
\draw [style=dashed] (2,2)--(4,0);
\draw [style=dashed] (4,4)--(4.5,3.5);
\draw [style=dashed] (4,0)--(4.5,0.5);
\end{tikzpicture}
\caption{Computing the intersection cardinalities}\label{fig:C2recursion}
\end{figure}
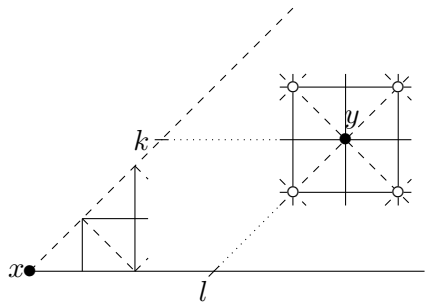

Suppose that $y\in V_{k,l}(x)$, with $k,l\geq 1$, and consider the intersection $V_{i,j}(x)\cap V_{0,1}(y)$. It is clear from Figure~\ref{fig:C2recursion} that if this intersection is nonempty then 
$$(i,j)\in\{(k,l-1),(k-1,l+1),(k+1,l-1),(k,l+1)\}$$ (these are the $4$ points marked with $\circ$ in Figure~\ref{fig:C2recursion}). Using some basic building theory we see that $|V_{k,l-1}(x)\cap V_{0,1}(y)|=1$, $|V_{k-1,l+1}(x)\cap V_{0,1}(y)|=q$, $|V_{k+1,l-1}(x)\cap V_{0,1}(y)|=qr$, and $|V_{k,l+1}(x)\cap V_{0,1}(y)|=q^2r$. Thus for large enough $m,n\geq 1$ we have
\begin{align*}
|V_{m,n}(x)\cap V_{0,1}(y)|&=\begin{cases}
1&\text{if $y\in V_{m,n+1}(x)$}\\
q&\text{if $y\in V_{m+1,n-1}(x)$}\\
qr&\text{if $y\in V_{m-1,n+1}(x)$}\\
q^2r&\text{if $y\in V_{m,n-1}(x)$}.
\end{cases}
\end{align*}
It follows from (\ref{eq:productform}) and Lemma~\ref{lem:N} that
\begin{align*}
A_{m,n}A_{0,1}&=\frac{1}{N_{0,1}}(q^2rA_{m,n+1}+qrA_{m+1,n-1}+qA_{m-1,n+1}+A_{m,n-1}).
\end{align*}
The remaining formulae are similar, with some care for small values of $m$ and $n$.
\end{proof}

Let $\scA$ be the linear span of $\{A_{m,n}\mid m,n\geq 0\}$ over $\CC$. 

\begin{lemma}\label{lem:partial}
The vector space $\scA$ is a commutative unital algebra over $\CC$, generated by $A_{1,0}$ and $A_{0,1}$. Moreover, $\scA$ is isomorphic to $\CC[X,Y]$ (the algebra of polynomials in commuting indeterminates~$X$ and~$Y$), with an isomorphism given by $X\mapsto A_{1,0}$ and $Y\mapsto A_{0,1}$. 
\end{lemma}

\begin{proof}
Let $\prec$ be the total order on $\NN^2$ given by $(k,l)\prec (m,n)$ if either $k+l<m+n$ or $k+l=m+n$ and $k<m$. An induction using this total order and the formulae in Theorem~\ref{thm:recursionC2} shows that for each $(m,n)\in\NN^2$ and $(k,l)\in\NN^2$ the product $A_{m,n}A_{k,l}$ is a linear combination of terms $A_{i,j}$ with $(i,j)\in\NN^2$. Thus $\scA$ is a unital algebra (with unit $A_{0,0}=I$). Moreover, for each $(m,n)\in\NN^2$ an induction shows that there is a positive number $c_{m,n}>0$ such that
\begin{align}\label{eq:poly}
A_{m,n}=c_{m,n}A_{1,0}^mA_{0,1}^n+\text{a linear combination of $A_{1,0}^kA_{0,1}^{l}$ with $(k,l)\prec (m,n)$}.
\end{align}
Thus $\scA$ is generated by $A_{1,0}$ and $A_{0,1}$, and hence is commutative as $A_{1,0}A_{0,1}=A_{0,1}A_{1,0}$. 

It follows that there is a surjective homomorphism $\psi:\CC[X,Y]\to\scA$ with $\psi(X)=A_{1,0}$ and $\psi(Y)=A_{0,1}$. Suppose that $z=\sum a_{k,l}X^kY^l\in\ker(\psi)$ is nonzero, and let $(m,n)\in\NN^2$ be maximal subject to $a_{m,n}\neq 0$. Then~(\ref{eq:poly}) implies that
$$
0=\psi(z)=c_{m,n}'A_{m,n}+\text{linear combination of terms $A_{k,l}$ with $(k,l)\prec (m,n)$}
$$
for some $c_{m,n}'\neq 0$, contradicting the linear independence of the operators $A_{k,l}$. Thus $\psi$ is injective, and so $\scA\cong \CC[X,Y]$. 
\end{proof}

Let $C_2$ be the group of signed permutations on two letters, acting on pairs of nonzero complex numbers $(z_1,z_2)$ permutations and inversions (for example, there is an element $\sigma\in C_2$ with $\sigma(z_1,z_2)=(z_2^{-1},z_1)$). This group of order~$8$ is the Weyl group of type~$C_2$.

\begin{thm}\label{thm:reps} For each pair $(z_1,z_2)$ of nonzero complex numbers there is a $1$-dimensional representation $\pi_{z_1,z_2}$ of $\scA$ given by 
\begin{align*}
\pi_{z_1,z_2}(A_{m,n})=\frac{(qr)^{-m}(q\sqrt{r})^{-n}}{(1+q^{-1})(1+r^{-1})(1+q^{-1}r^{-1})}\sum_{\sigma\in C_2}c(z_{\sigma(1)},z_{\sigma(2)})z_{\sigma(1)}^{m+n}z_{\sigma(2)}^n
\end{align*}
where
$$
c(z_1,z_2)=\frac{(1-q^{-1}z_1^{-1}z_2^{-1})(1-q^{-1}z_1^{-1}z_2)(1-r^{-1}z_1^{-2})(1-r^{-1}z_2^{-2})}
{(1-z_1^{-1}z_2^{-1})(1-z_1^{-1}z_2)(1-z_1^{-2})(1-z_2^{-2})}
$$
whenever $z_1,z_2,z_1^{-1},z_2^{-1}$ are pairwise distinct, and if $z_1,z_2,z_1^{-1},z_2^{-1}$ are not pairwise distinct then the formula for $\pi_{z_1,z_2}(A_{m,n})$ is obtained from the above formula by taking an appropriate limit. Moreover every $1$-dimensional representation $\pi$ of $\scA$ is of the form $\pi=\pi_{z_1,z_2}$ for some $z_1,z_2\in\CC^{\times}$, and $\pi_{z_1,z_2}=\pi_{z_1',z_2'}$ if and only if $(z_1',z_2')=\sigma(z_1,z_2)$ for some $\sigma\in C_2$. 
\end{thm}

\begin{proof} 
By Lemma~\ref{lem:partial} we have $\scA\cong\CC[X,Y]$. The $1$-dimensional representations of $\CC[X,Y]$ are precisely the evaluation maps $X\mapsto u$ and $Y\mapsto v$, and thus for each $(u,v)\in\CC^2$ there is a unique $1$-dimensional representation $\pi^{(u,v)}$ of $\scA$ determined by $\pi^{(u,v)}(A_{1,0})=u$ and $\pi^{(u,v)}(A_{0,1})=v$, and all $1$-dimensional representations are of this form. 

Let $(u,v)\in\CC^2$ and write $a_{m,n}=(qr)^m(q\sqrt{r})^n\pi^{(u,v)}(A_{m,n})$. Let $u'=(qr)^{-1}N_{1,0}u$ and $v'=(q\sqrt{r})^{-1}N_{0,1}v$. Applying $\pi$ to the formulae in Theorem~\ref{thm:recursionC2} gives:
\begin{align}
\label{eq:cc1}u'a_{m,n}&=a_{m+1,n-2}+(1-q^{-1})(1+r^{-1})a_{m,n}+a_{m+1,n}+a_{m-1,n+2}+a_{m-1,n}\\
\label{eq:cc2}u'a_{0,n}&=(1+r^{-1})(a_{1,n-2}+(1-q^{-1})a_{0,n}+a_{1,n})\\
\label{eq:cc3}u'a_{m,0}&=(1+q^{-1})a_{m-1,2}+a_{m+1,0}+(1-q^{-1})r^{-1}a_{m,0}+a_{m-1,0}\\
\label{eq:cc4}u'a_{m,1}&=a_{m+1,1}+a_{m-1,3}+a_{m-1,1}+(1+r^{-1}-q^{-1}r^{-1})a_{m,1}\\
\label{eq:cc5}v'a_{m,n}&=a_{m,n-1}+a_{m+1,n-1}+a_{m-1,n+1}+a_{m,n+1}\\
\label{eq:cc6}v'a_{0,n}&=a_{0,n-1}+a_{0,n+1}+(1+r^{-1})a_{1,n-1}\\
\label{eq:cc7}v'a_{m,0}&=(1+q^{-1})(a_{m-1,1}+a_{m,1}),
\end{align}
where in each case the indices $m,n$ are required to be large enough to ensure that the indices appearing on the right are all at least~$0$. From (\ref{eq:cc5}) we have
$$
a_{m+1,n-1}+a_{m-1,n+1}=v'a_{m,n}-a_{m,n-1}-a_{m,n+1}\quad\text{for all $m,n\geq 1$},
$$
and using this equation in (\ref{eq:cc1}) gives
\begin{align*}
u'a_{m,n}&=(1-q^{-1})(1+r^{-1})a_{m,n}+(a_{m+1,n-2}+a_{m-1,n})+(a_{m+1,n}+a_{m-1,n+2})\\
&=(1-q^{-1})(1+r^{-1})a_{m,n}+(v'a_{m,n-1}-a_{m,n-2}-a_{m,n})+(v'a_{m,n+1}-a_{m,n}-a_{m,n+2}),
\end{align*}
valid for all $m\geq 1$ and $n\geq 2$. A similar calculation using~(\ref{eq:cc6}) and (\ref{eq:cc2}) shows that the above formula also holds for $m=0$. By replacing $n$ by $n+2$ and rearranging we obtain
\begin{align}\label{eq:look}
a_{m,n+4}-v' a_{m,n+3}+\alpha a_{m,n+2}-v' a_{m,n+1}+a_{m,n}=0\quad\text{for all $m,n\geq 0$,}
\end{align}
where $\alpha=2+u'-(1-q^{-1})(1+r^{-1})$.  The auxiliary equation of this linear recurrence (in $n$) factorises as a product of two quadratics: $\lambda^4-v'\lambda^3+\alpha\lambda^2-v'\lambda+1=(\lambda^2-a\lambda+1)(\lambda^2-b\lambda+1)$, and so the roots of the auxiliary equation are of the form $z_1,z_1^{-1},z_2,z_2^{-1}$ for some numbers $z_1,z_2\in\CC^{\times}$. By Newton's identities we have 
\begin{align}
\label{eq:base1}u&=\frac{qr}{N_{1,0}}\left((1-q^{-1})(1+r^{-1})+(z_1+z_1^{-1})(z_2+z_2^{-1})\right)\\
\label{eq:base2}v&=\frac{q\sqrt{r}}{N_{0,1}}\left(z_1+z_1^{-1}+z_2+z_2^{-1}\right).
\end{align}
Writing $\pi_{z_1,z_2}=\pi^{(u,v)}$ whenever $(z_1,z_2)\in(\CC^{\times})^2$ and $(u,v)\in\CC^2$ are related as above, it follows from (\ref{eq:base1}) and (\ref{eq:base2}) that $\pi_{z_1,z_2}=\pi_{z_1',z_2'}$ if and only if $(z_1',z_2')=(z_{\sigma(1)},z_{\sigma(2)})$ for some $\sigma\in C_2$. We now verify that $\pi_{z_1,z_2}(A_{m,n})$ is given by the formula in the statement of the theorem.

Assuming for now that $z_1,z_1^{-1},z_2,z_2^{-1}$ are pairwise distinct, solving the recurrence~(\ref{eq:look}) gives
$$
a_{m,n}=C_{1,m}(z_1,z_2)z_1^n+C_{2,m}(z_1,z_2)z_2^n+C_{3,m}(z_1,z_2)z_1^{-n}+C_{4,m}(z_1,z_2)z_2^{-n}
$$
for suitable functions $C_{i,m}(z_1,z_2)$ (independent of $n$). Writing $C_{1,m}(z_1,z_2)=C_m(z_1,z_2)$, the invariance under the group~$C_2$ implies that 
$$
C_{2,m}(z_1,z_2)=C_m(z_2,z_1),\quad C_{3,m}(z_1,z_2)=C_m(z_1^{-1},z_2^{-1}),\quad C_{4,m}(z_1,z_2)=C_m(z_2^{-1},z_1^{-1}),
$$
and also that $C_m(z_1,z_2^{-1})=C_m(z_1,z_2)$. Thus for all $m,n\geq 0$ we have
\begin{align}\label{eq:temprec}
a_{m,n}=C_{m}(z_1,z_2)z_1^n+C_{m}(z_2,z_1)z_2^n+C_{m}(z_1^{-1},z_2^{-1})z_1^{-n}+C_{m}(z_2^{-1},z_1^{-1})z_2^{-n}.
\end{align}
Writing $C_m=C_m(z_1,z_2)$ it follows from (\ref{eq:temprec}) and (\ref{eq:cc5}) that
\begin{align}\label{eq:anotherrec}
z_1^{-1}C_{m+2}-(z_2+z_2^{-1})C_{m+1}+z_1C_m=0\quad\text{for all $m\geq 0$}
\end{align}
(we have used the fact that $v'=z_1+z_1^{-1}+z_2+z_2^{-1}$). The roots of the auxiliary equation of the recurrence~(\ref{eq:anotherrec}) are $z_1z_2$ and $z_1z_2^{-1}$, and these are distinct by hypothesis, and hence
$$
C_m=D(z_1,z_2)z_1^mz_2^m+D'(z_1,z_2)z_1^mz_2^{-m}\quad\text{for all $m\geq 0$}
$$
for suitable functions $D(z_1,z_2)$ and $D'(z_1,z_2)$ independent of $m$. Since $C_m(z_1,z_2)=C_m(z_1,z_2^{-1})$ we have $D'(z_1,z_2)=D(z_1,z_2^{-1})$, and thus by (\ref{eq:temprec}) we have
$$
a_{m,n}=\sum_{\sigma\in C_2}D(z_{\sigma(1)},z_{\sigma(2)})z_{\sigma(1)}^{m+n}z_{\sigma(2)}^m\quad\text{for all $m,n\geq 0$}.
$$

To compute $D(z_1,z_2)$ we proceed as follows: Using the recurrence formulae we obtain explicit formulae for $a_{0,0},a_{0,1},a_{0,2}$ and $a_{0,3}$ in terms of $z_1$ and $z_2$. In particular, $a_{0,0}=1$, and $a_{1,0}=qr u$ and $a_{0,1}=q\sqrt{r} v$ are given by (\ref{eq:base1}) and (\ref{eq:base2}). Then~(\ref{eq:cc6}) with $n=1$ gives $a_{0,2}=v'a_{0,1}-1-(1+r^{-1})a_{1,0}$. By (\ref{eq:cc6}) with $n=2$ we have $a_{0,3}=v'a_{0,2}-a_{0,1}-(1+r^{-1})a_{1,1}$, and $a_{1,1}$ is computed using~(\ref{eq:cc7}) with $m=1$ giving $a_{11}=v'(1+q^{-1})^{-1}a_{1,0}-a_{0,1}$. This gives the initial conditions of the recurrence~(\ref{eq:look}) with $m=0$, and thus the coefficients in~(\ref{eq:temprec}) (with $m=0$) can be computed, giving
$$
C_0(z_1,z_2)=\frac{(1-q^{-1}z_1^{-1}z_2^{-1})(1-q^{-1}z_1^{-1}z_2)(1-r^{-1}z_1^{-2})}{(1+q^{-1})(1+q^{-1}r^{-1})(1-z_1^{-1}z_2^{-1})(1-z_1^{-1}z_2)(1-z_1^{-2})}.
$$
From (\ref{eq:cc6}) and (\ref{eq:temprec}) we see that
$
C_1(z_1,z_2)=\frac{z_1(z_2+z_2^{-1})}{1+r^{-1}}C_0,
$
and thus the initial conditions of the recurrence~(\ref{eq:anotherrec}) are known. Thus we can solve for $D(z_1,z_2)$, and we find that $D(z_1,z_2)=c(z_1,z_2)/(1+q^{-1})(1+r^{-1})(1+q^{-1}r^{-1})$, with $c(z_1,z_2)$ as in the statement of the theorem (a computer algebra package is recommended for these calculations). 

Thus the formula for $\pi_{z_1,z_2}$ is verified in the case where $z_1,z_2,z_1^{-1},z_2^{-1}$ are pairwise distinct. Generally, from (\ref{eq:base1}), (\ref{eq:base2}), and the fact that $A_{1,0}$ and $A_{0,1}$ generate~$\scA$ we see that $\pi(A_{m,n})$ is a polynomial in $z_1,z_1^{-1},z_2,z_2^{-1}$. Thus in the case where $z_1,z_1^{-1},z_2,z_2^{-1}$ are not pairwise distinct we can obtain the formula for $\pi_{z_1,z_2}(A_{m,n})$ by taking an appropriate limit.
\end{proof}

Let $\TT=\{t\in\CC\mid |t|=1\}$, and let $dt$ denote normalised Haar measure on~$\TT$. Thus for integrable functions $f$ on $\TT$ we have $\int_{\TT}f(t)\,dt=\frac{1}{2\pi}\int_0^{2\pi}f(e^{i\theta})\,d\theta$.  If $A\in\scA$ and $t\in\TT^2$ we write $\widehat{A}(t)=\pi_{t}(A)$. The following theorem establishes the Plancherel formula for the algebra~$\scA$. 

\begin{thm}\label{thm:plancherelC2} We have
$$
\frac{1}{N_{k,l}}\int_{\TT^2}\widehat{A}_{k,l}(t)\widehat{A}_{m,n}(t)\,d\mu(t)=\delta_{(k,l),(m,n)},
$$
where $d\mu(t)$ is the measure
$$
d\mu(z)=K\frac{1}{|c(z_1,z_2)|^2}dt_1dt_2,\quad\text{with}\quad K=\frac{1}{8}(1+q^{-1})(1+r^{-1})(1+q^{-1}r^{-1}).
$$
\end{thm}

\begin{proof}
Since $A_{k,l}A_{m,n}=\sum_{i,j\geq 0}c_{(k,l),(m,n)}^{(i,j)}A_{i,j}$ and $c_{(k,l),(m,n)}^{(0,0)}=\delta_{(k,l),(m,n)}N_{k,l}$, it suffices to show that $\int_{\TT^2}\widehat{A}_{k,l}(t)\,d\mu(t)=\delta_{(k,l),(0,0)}$. Using the facts that $|c(t_1,t_2)|^2=c(t_{\sigma(1)},t_{\sigma(2)})c(t_{\sigma(1)}^{-1},t_{\sigma(2)}^{-1})$ for $(t_1,t_2)\in\TT^2$ and $\sigma\in C_2$, and that $\int_{\TT}f(t^{-1})\,dt=\int_{\TT}f(t)\,dt$ we have
\begin{align*}
\int_{\TT^2}\widehat{A}_{k,l}(t)d\mu(t)&=\frac{1}{8}(qr)^{-k}(q\sqrt{r})^{-l}\int_{\TT^2}\sum_{\sigma\in C_2}\frac{t_{\sigma(1)}^{k+l}t_{\sigma(2)}^{l}}{c(t_{\sigma(1)}^{-1},t_{\sigma(2)}^{-1})}\,dt_1dt_2\\
&=(qr)^{-k}(q\sqrt{r})^{-l}\int_{\TT}\bigg(\int_{\TT}\frac{t_{1}^{k+l}t_{2}^{l}}{c(t_{1}^{-1},t_{2}^{-1})}\,dt_1\bigg)dt_2.
\end{align*}
As a contour integral, the inner integral is
$$
\int_{\TT}\frac{t_1^{k+l}t_2^l}{c(t_1^{-1},t_2^{-1})}dt_1=\frac{1}{2\pi i}\int_{\Gamma}\frac{z_1^{k+l}t_2^l}{c(z_1^{-1},t_2^{-1})}\frac{dz_1}{z_1}
$$
where $\Gamma$ is the unit circle traversed once counterclockwise. The poles of the function $f(z_1)=1/c(z_1^{-1},t_2^{-1})$ are at $z_1=qt_2^{-1},qt_2,\pm\sqrt{r}$ and since $|t_2|=1$ and $q,r>1$ we see that $f(z_1)$ has no poles inside the contour~$\Gamma$, and so by residue calculus we deduce that
$$
\int_{\TT}\frac{t_1^{k+l}t_2^l}{c(t_1^{-1},t_2^{-1})}dt_1=\delta_{(k,l),(0,0)}\lim_{z_1\to 0}\frac{1}{c(z_1^{-1},t_2^{-1})}=\delta_{(k,l),(0,0)}\frac{1-t_2^2}{1-r^{-1}t_2^2}.
$$
Thus
$$
\int_{\TT^2}\widehat{A}_{k,l}(t)\,d\mu(t)=\delta_{(k,l),(0,0)}\int_{\TT}\frac{1-t_2^2}{1-r^{-1}t_2^2}\,dt_2=\delta_{(k,l),(0,0)}.\qedhere
$$
\end{proof}

\begin{thm}\label{thm:pn} Let $(X_n)_{n\geq 0}$ be an isotropic random walk on $\Sigma$ with transition operator~$A$. Then
$$
p^{(n)}(x,y)=\frac{1}{N_{k,l}^2}\int_{\TT^2}\widehat{A}(t)^n\widehat{A}_{k,l}(t)\,d\mu(t)\quad\text{if $y\in V_{k,l}(x)$}.
$$
\end{thm}

\begin{proof} We have $A^n=\sum_{i,j\geq 0}a_{i,j}^{(n)}A_{i,j}$ where $p^{(n)}(x,y)=a_{i,j}^{(n)}/N_{i,j}$ whenever $y\in V_{i,j}(x)$. Thus by Theorem~\ref{thm:plancherelC2} we have, for any pair $x,y$ with $y\in V_{k,l}(x)$,
$$
\int_{\TT^2}\widehat{A}(t)^n\widehat{A}_{k,l}(t)\,d\mu(t)=\sum_{i,j\geq 0}a_{i,j}^{(n)}\int_{\TT^2}\widehat{A}_{i,j}(t)\widehat{A}_{k,l}(t)\,d\mu(t)=N_{k,l}a_{k,l}^{(n)}=N_{k,l}^2p^{(n)}(x,y).\qedhere
$$
\end{proof}

The asymptotics for the $n$-step transition probabilities of the random walk (that is, the local limit theorem) can be extracted in a standard way from Theorem~\ref{thm:pn}. Let us simply illustrate this in an example. Consider the `simple random walk' on $V$ with transition operator $A_{0,1}$ (this is the random walk such that if $X_n=y$ in Figure~\ref{fig:C2recursion}, then $X_{n+1}$ is one of the $N_{0,1}=(q+1)(qr+1)$ vertices marked with $\circ$ in the figure, each chosen with equal probability~$1/N_{0,1}$). We can compute $\widehat{A}_{0,1}(t)$ from the general formula in Theorem~\ref{thm:reps}, however a shortcut is given by~(\ref{eq:base2}), giving
$$
\widehat{A}_{0,1}(t)=\frac{q\sqrt{r}}{(q+1)(qr+1)}\left(t_1+t_1^{-1}+t_2+t_2^{-1}\right).
$$
The simple random walk is periodic, with period~$2$, and so we consider $p^{(2n)}(x,x)$. Writing $e^{i\theta}=(e^{i\theta_1},e^{i\theta_2})$ we have
\begin{align*}
\widehat{A}_{0,1}(e^{i\theta})&=\frac{2q\sqrt{r}(\cos\theta_1+\cos\theta_2)}{(q+1)(qr+1)}\quad\text{and}\quad \frac{1}{|c(e^{i\theta})|^2}=\frac{4(\theta_1^2-\theta_2^2)^2\theta_1^2\theta_2^2}{(1-q^{-1})^4(1-r^{-1})^4}+O(\|\theta\|^4).
\end{align*}
Let $\rho=4q\sqrt{r}/((q+1)(qr+1))$ and $K'=K/(\pi^2(1-q^{-1})^4(1-r^{-1})^4)$. Some standard tricks from asymptotic analysis now give
\begin{align*}
p^{(2n)}(x,x)&=\frac{K}{4\pi^2}\int_{-\pi}^{\pi}\int_{-\pi}^{\pi}\frac{\widehat{A}_{0,1}(e^{i\theta})}{|c(e^{i\theta})|^2}\,d\theta_1d\theta_2\\
&\sim K'\rho^{2n}\int_{-\pi}^{\pi}\int_{-\pi}^{\pi}\left(\frac{\cos\theta_1+\cos\theta_2}{2}\right)^{2n}(\theta_1^2-\theta_2^2)^2\theta_1^2\theta_2^2\,d\theta_1d\theta_2\\
&\sim2K'\rho^{2n}\int_{-\epsilon}^{\epsilon}\int_{-\epsilon}^{\epsilon}\left(\frac{\cos\theta_1+\cos\theta_2}{2}\right)^{2n}(\theta_1^2-\theta_2^2)^2\theta_1^2\theta_2^2\,d\theta_1d\theta_2\\
&=\frac{2K'}{n^4}\rho^{2n}\int_{-\sqrt{n}\epsilon}^{\sqrt{n}\epsilon}\int_{-\sqrt{n}\epsilon}^{\sqrt{n}\epsilon}\left(\frac{\cos(\varphi_1/\sqrt{n})+\cos(\varphi_2/\sqrt{n})}{2}\right)^{2n}(\varphi_1^2-\varphi_2^2)^2\varphi_1^2\varphi_2^2\,d\varphi_1d\varphi_2\\
&\sim\frac{2K'}{n^4}\rho^{2n}\int_{-\infty}^{\infty}\int_{-\infty}^{\infty}e^{-(\varphi_1^2+\varphi_2^2)/2}(\varphi_1^2-\varphi_2^2)^2\varphi_1^2\varphi_2^2\,d\varphi_1d\varphi_2,
\end{align*}
and thus in conclusion we have
\begin{align*}
p^{(2n)}(x,x)&\sim\frac{6(q+1)(r+1)(qr+1)q^2r^2}{\pi(q-1)^4(r-1)^4}\left(\frac{4q\sqrt{r}}{(q+1)(qr+1)}\right)^{2n}n^{-4}.
\end{align*}

\section{Rank~2 Hecke algebras and the Feit-Higman Theorem}\label{app:rank2}

In this appendix we present the representation theory of rank $2$ spherical Hecke algebras and apply it to Theorems~\ref{thm:spherical1} and~\ref{thm:spherical2}. As a byproduct we arrive at a proof of the Feit-Higman Theorem (this proof is due to Kilmoyer and Solomon~\cite{KS:73}). Let $(W,S)$ be the Coxeter system of type $I_2(m)$. Write $S=\{s_1,s_2\}$, and write $q_{s_1}=q$ and $q_{s_2}=r$ (with $q=r$ if $m$ is odd). Let $T_i=T_{s_i}$ for $i=1,2$ be the generators of the abstract Hecke algebra $\scH$. The classification of the irreducible representations of $\scH$ is elementary:

\begin{prop}\label{prop:22}\emph{\cite[Theorem~8.3.1]{GP:00}}
In the above notation:
\begin{enumerate}
\item[\emph{(1)}] If $m$ is odd then the complete list of irreducible representations of $\scH$ is as follows: There are precisely $2$ $1$-dimensional irreducible representations, given by
\begin{align*}
\begin{aligned}
\rho_{\mathrm{triv}}(T_1)&=1\\
\rho_{\mathrm{triv}}(T_2)&=1
\end{aligned}\quad\text{and}\quad 
\begin{aligned}
\rho_{\mathrm{sgn}}(T_1)&=-q^{-1}\\
\rho_{\mathrm{sgn}}(T_2)&=-q^{-1}
\end{aligned}
\end{align*}
and precisely $(m-1)/2$ $2$-dimensional representations, given by
$$
\rho_j(T_1)=\frac{1}{q}\begin{bmatrix}-1&0\\
c_j&q\end{bmatrix}\quad\text{and}\quad\rho_j(T_2)=\frac{1}{q}\begin{bmatrix}q&c_j'\\
0&-1\end{bmatrix}\quad\text{for $1\leq j\leq (m-1)/2$},
$$
where $c_j$ and $c_j'$ are any numbers satisfying $c_jc_j'=4q\cos^2(\pi j / m)$. 
\item[\emph{(2)}] If $m$ is even then the complete list of irreducible representation of $\scH$ is as follows: There are preciesly $4$ $1$-dimensional representations, given by  
\begin{align*}
\begin{aligned}
\rho_{\mathrm{triv}}(T_1)&=1\\
\rho_{\mathrm{triv}}(T_2)&=1
\end{aligned}\quad \text{and}\quad
\begin{aligned}
\rho_{\mathrm{sgn}}(T_1)&=-q^{-1}\\
\rho_{\mathrm{sgn}}(T_2)&=-r^{-1}
\end{aligned}\quad\text{and}\quad
\begin{aligned}
\rho^1(T_1)&=1\\
\rho^1(T_2)&=-r^{-1}
\end{aligned}\quad\text{and}\quad\begin{aligned}
\rho^2(T_1)&=-q^{-1}\\
\rho^2(T_2)&=1.
\end{aligned}
\end{align*}
There are exactly $(m-2)/2$ $2$-dimensional representations, given by 
$$
\rho_j(T_1)=\frac{1}{q}\begin{bmatrix}-1&0\\
c_j&q\end{bmatrix}\quad\text{and}\quad \rho_j(T_2)=\frac{1}{r}\begin{bmatrix}r&c_j'\\
0&-1\end{bmatrix}\quad\text{for $1\leq j\leq (m-2)/2$},
$$
where $c_j$ and $c_j'$ are any numbers satisfying $c_jc_j'=q+r+2\sqrt{qr}\cos(2\pi j/m)$. 
\end{enumerate}
\end{prop}

\begin{proof}
It is a straightforward exercise to show that the claimed formulae produce representations (by checking that the defining relations~(\ref{eq:presentation2}) are satisfied). It is also easy to check that the representations are irreducible. To check that they are pairwise non-isomorphic one can compute the inner products and check~(\ref{eq:innerprod}), and finally to check we have all irreducible representations one uses the character theoretic fact that $\sum \dim(\rho)^2=|I_2(m)|=2m$, where the sum is over all irreducible representations of~$\scH$. See \cite[Theorem~8.3.1]{GP:00} for details.
\end{proof}

The representation theory described in Proposition~\ref{prop:22} allows us to be extremely precise for random walks on generalised polygons. Let us simply illustrate these arguments for the case of generalised quadrangles (that is, $m=4$):

\begin{cor} For the simple random walk on a generalised quadrangle with parameters $(q,r)$ we have
$$
p^{(n)}(o,o)=\frac{1+k_1\lambda_1^n+k_2\lambda_2^n+k_3\lambda_3^n+k_4(\lambda_+^n+\lambda_-^n)}{(q+1)(r+1)(qr+1)}
$$
and 
$$
\|\mu^{(n)}-u\|_{\mathrm{tv}}^2\leq \frac{1}{4}\left(k_1\lambda_1^{2n}+k_2\lambda_2^{2n}+k_3\lambda_3^{2n}+k_4(\lambda_+^{2n}+\lambda_-^{2n})\right),
$$
where the numbers $k_i$, $\lambda_i$, and $\lambda_{\pm}$ are given by $k_1=q^2r^2$, $k_2=r^2(qr+1)/(q+r)$, $k_3=q^2(qr+1)/(q+r)$, $k_4=qr(q+1)(r+1)/(q+r)$, $\lambda_1=-2/(q+r)$, $\lambda_2=(q-1)/(q+r)$, $\lambda_3=(r-1)/(q+r)$, and 
$
\lambda_{\pm}=(q+r-2\pm\sqrt{(q-r)^2+4(q+r)})/2(q+r).
$
\end{cor}

\begin{proof}
For the first statement, from Theorem~\ref{thm:spherical1} we have
$$
p^{(n)}(o,o)=\frac{1}{|\Delta|}\sum_{\rho\in\mathrm{Irrep}(\scH)}m_{\rho}\chi_{\rho}(T^n),
$$
where $T=\frac{q}{q+r}T_1+\frac{r}{q+r}T_2$ (since we are considering the simple random walk). We have $|\Delta|=(q+1)(r+1)(qr+1)$, and by Proposition~\ref{prop:22} there are $5$ irreducible representations of $\scH$ with respective multiplicities $m_{\mathrm{triv}}=1$, $m_{\mathrm{sgn}}=q^2r^2$, $m^1=r^2(qr+1)/(q+r)$, $m^2=q^2(qr+1)/(q+r)$, and $m_1=qr(q+1)(r+1)/(q+r)$ (computed using Theorem~\ref{thm:character}). We have $\chi_{\mathrm{triv}}(T^n)=1$, $\chi_{\mathrm{sgn}}(T^n)=(-2/(q+r))^n$, $\chi^1(T^n)=((q-1)/(q+r))^n$, $\chi^2(T)=((r-1)/(q+r))^n$, and a calculation of eigenvalues gives $\chi_1(T^n)=\lambda_+^n+\lambda_-^n$. The result follows, and the second statement follows similarly from Theorem~\ref{thm:spherical2}, noting that $T^*=T$. 
\end{proof}

By considering the formulae for the multiplicities of irreducible representations in the geometric representation we obtain a proof of the Feit-Higman Theorem and some of the divisibility conditions from Theorem~\ref{thm:param} (c.f. \cite{KS:73}):

\begin{proof}[Proof of the Feit-Higman Theorem and divisibility conditions]
Suppose that a finite thick generalised $m$-gon exists with parameters $(q,r)$. Let $\chi_j$ be the character of the representation $\rho_j$ from Proposition~\ref{prop:22}. By Theorem~\ref{thm:character} we have $\langle\chi_j,\chi_j\rangle\in\QQ$. On the other hand we can explicitly compute these inner products. Writing $\theta_j=2\pi j/m$, a tedious calculation gives 
$$
|\Delta|\langle\chi_j,\chi_j\rangle=\begin{cases}
2m+\frac{(q-1)^2m}{q(1-\cos\theta_j)}&\text{if $m$ is odd}\\
2m+\frac{\left(r(q-1)^2+q(r-1)^2\right)m}{2qr\sin^2\theta_j}+\frac{(q-1)(r-1)m\cos\theta_j}{\sqrt{qr}\sin^2\theta_j}&\text{if $m$ is even}.
\end{cases}
$$
If $m$ is odd, then the formulae force $\cos\theta_j$ to be rational, and this implies that $m=3$. 
If $m$ is even, then the above formulae imply that both $\sin^2(2\pi/m)$ and $\cos(2\pi/m)$ are rational (consider $\langle\chi_1,\chi_1\rangle+\langle\chi_{(m/2)-j},\chi_{(m/2)-j}\rangle$ to see that $\sin^2(2\pi/m)$ is rational). Together these facts imply that $m\in\{2,4,6,8\}$.

The divisibility conditions for the cases $m=4,6,8$ follow by computing the multiplicity of $\rho^1$ using Theorem~\ref{thm:character}. The facts that $\sqrt{qr}\in\ZZ$ and $\sqrt{2qr}\in\ZZ$ for hexagons and octagons (respectively) arise from the multiplicity of $\rho_1$. 
\end{proof}

\section{Spherical and affine Coxeter systems}\label{app:classification}

In this final appendix we list the Coxeter diagrams of the irreducible spherical and affine Coxeter systems. For the affine systems, the extra generator that is added to the spherical system is indicated by~$\circ$. 

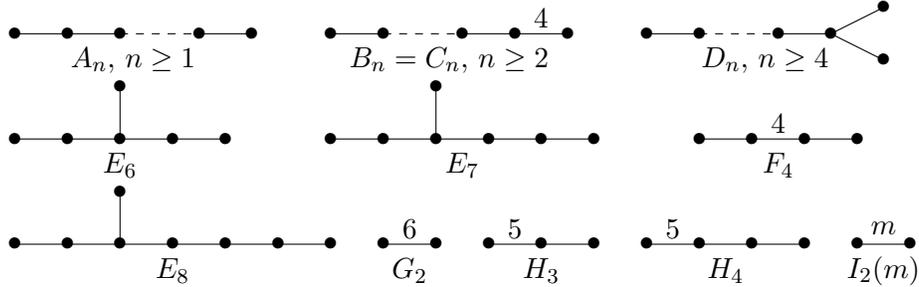
\begin{figure}[!h]
\centering
\begin{tikzpicture}[scale=0.7]
\node at (0,0) {$\bullet$};
\node at (1,0) {$\bullet$};
\node at (2,0) {$\bullet$};
\node at (3.5,0) {$\bullet$};
\node at (4.5,0) {$\bullet$};
\draw (0,0)--(2,0);
\draw (3.5,0)--(4.5,0);
\draw [style=dashed] (2,0)--(3.5,0);
\node at (2.25,-0.5) {$A_n,\,n\geq 1$};
\node at (6,0) {$\bullet$};
\node at (7,0) {$\bullet$};
\node at (8.5,0) {$\bullet$};
\node at (9.5,0) {$\bullet$};
\node at (10.5,0) {$\bullet$};
\draw (6,0)--(7,0);
\draw (8.5,0)--(10.5,0);
\node at (10,0.3) {$4$};
\draw [style=dashed] (7,0)--(8.5,0);
\node at (8.25,-0.5) {$B_n=C_n,\,n\geq 2$};
\node at (12,0) {$\bullet$};
\node at (13,0) {$\bullet$};
\node at (14.5,0) {$\bullet$};
\node at (15.5,0) {$\bullet$};
\node at (16.5,0.5) {$\bullet$};
\node at (16.5,-0.5) {$\bullet$};
\draw (12,0)--(13,0);
\draw (14.5,0)--(15.5,0);
\draw (15.5,0)--(16.5,0.5);
\draw (15.5,0)--(16.5,-0.5);
\draw [style=dashed] (13,0)--(14.5,0);
\node at (14.25,-0.5) {$D_n,\,n\geq 4$};
\node at (0,-2) {$\bullet$};
\node at (1,-2) {$\bullet$};
\node at (2,-2) {$\bullet$};
\node at (3,-2) {$\bullet$};
\node at (4,-2) {$\bullet$};
\node at (2,-1) {$\bullet$};
\draw (0,-2)--(4,-2);
\draw (2,-2)--(2,-1);
\node at (2,-2.5) {$E_6$};
\node at (6,-2) {$\bullet$};
\node at (7,-2) {$\bullet$};
\node at (8,-2) {$\bullet$};
\node at (9,-2) {$\bullet$};
\node at (10,-2) {$\bullet$};
\node at (11,-2) {$\bullet$};
\node at (8,-1) {$\bullet$};
\draw (6,-2)--(11,-2);
\draw (8,-2)--(8,-1);
\node at (8.5,-2.5) {$E_7$};
\node at (13,-2) {$\bullet$};
\node at (14,-2) {$\bullet$};
\node at (15,-2) {$\bullet$};
\node at (16,-2) {$\bullet$};
\draw (13,-2)--(16,-2);
\node at (14.5,-1.7) {$4$};
\node at (14.5,-2.5) {$F_4$};
\node at (0,-4) {$\bullet$};
\node at (1,-4) {$\bullet$};
\node at (2,-4) {$\bullet$};
\node at (3,-4) {$\bullet$};
\node at (4,-4) {$\bullet$};
\node at (5,-4) {$\bullet$};
\node at (6,-4) {$\bullet$};
\node at (2,-3) {$\bullet$};
\draw (0,-4)--(6,-4);
\draw (2,-4)--(2,-3);
\node at (3,-4.5) {$E_8$};
\node at (7,-4) {$\bullet$};
\node at (8,-4) {$\bullet$};
\draw (7,-4)--(8,-4);
\node at (7.5,-3.7) {$6$};
\node at (7.5,-4.5) {$G_2$};
\node at (9,-4) {$\bullet$};
\node at (10,-4) {$\bullet$};
\node at (11,-4) {$\bullet$};
\draw (9,-4)--(11,-4);
\node at (9.5,-3.7) {$5$};
\node at (10,-4.5) {$H_3$};
\node at (12,-4) {$\bullet$};
\node at (13,-4) {$\bullet$};
\node at (14,-4) {$\bullet$};
\node at (15,-4) {$\bullet$};
\draw (12,-4)--(15,-4);
\node at (12.5,-3.7) {$5$};
\node at (13.5,-4.5) {$H_4$};
\node at (16,-4) {$\bullet$};
\node at (17,-4) {$\bullet$};
\draw (16,-4)--(17,-4);
\node at (16.5,-4.5) {$I_2(m)$};
\node at (16.5,-3.7) {$m$};
\end{tikzpicture}
\caption{Irreducible spherical Coxeter systems}\label{fig:sphericalclassification}
\end{figure}

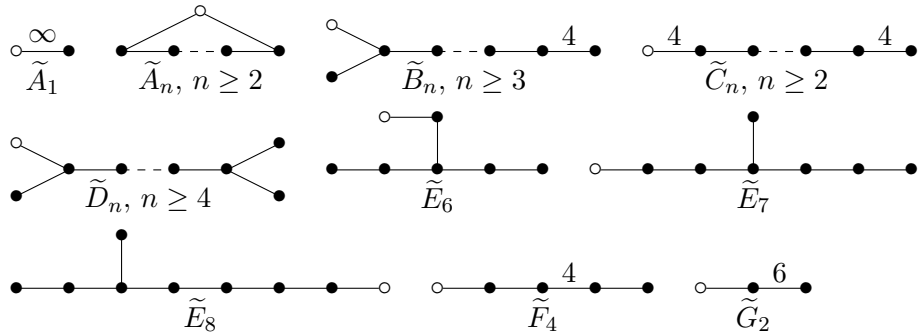
\begin{figure}[!h]
\centering
\begin{tikzpicture}[scale=0.7]
\draw (0,0)--(1,0);
\node [color=white] at (0,0) {$\bullet$};
\node at (0,0) {$\circ$};
\node at (1,0) {$\bullet$};
\node at (0.5,0.3) {$\infty$};
\node at (0.5,-0.5) {$\widetilde{A}_1$};
\node at (2,0) {$\bullet$};
\node at (3,0) {$\bullet$};
\node at (4,0) {$\bullet$};
\node at (5,0) {$\bullet$};
\draw (2,0)--(3,0);
\draw (4,0)--(5,0);
\draw [style=dashed] (3,0)--(4,0);
\draw (2,0)--(3.5,0.75)--(5,0);
\node [color=white] at  (3.5,0.75) {$\bullet$};
\node at (3.5,0.75) {$\circ$};
\node at (3.5,-0.5) {$\widetilde{A}_n,\,n\geq 2$};
\node at (6,-0.5) {$\bullet$};
\node at (7,0) {$\bullet$};
\node at (8,0) {$\bullet$};
\node at (9,0) {$\bullet$};
\node at (10,0) {$\bullet$};
\node at (11,0) {$\bullet$};
\draw (6,0.5)--(7,0)--(8,0);
\draw (6,-0.5)--(7,0);
\draw (9,0)--(11,0);
\node at (10.5,0.3) {$4$};
\draw [style=dashed] (8,0)--(9,0);
\node [color=white] at (6,0.5) {$\bullet$};
\node at (6,0.5) {$\circ$};
\node at (8.5,-0.5) {$\widetilde{B}_n,\,n\geq 3$};
\node at (13,0) {$\bullet$};
\node at (14,0) {$\bullet$};
\node at (15,0) {$\bullet$};
\node at (16,0) {$\bullet$};
\node at (17,0) {$\bullet$};
\draw (12,0)--(14,0);
\draw (15,0)--(17,0);
\node at (16.5,0.3) {$4$};
\node at (12.5,0.3) {$4$};
\draw [style=dashed] (14,0)--(15,0);
\node [color=white] at (12,0) {$\bullet$};
\node at (12,0) {$\circ$};
\node at (14.25,-0.5) {$\widetilde{C}_n,\,n\geq 2$};
\node at (0,-2.75) {$\bullet$};
\node at (1,-2.25) {$\bullet$};
\node at (2,-2.25) {$\bullet$};
\node at (3,-2.25) {$\bullet$};
\node at (4,-2.25) {$\bullet$};
\node at (5,-1.75) {$\bullet$};
\node at (5,-2.75) {$\bullet$};
\draw (0,-1.75)--(1,-2.25)--(2,-2.25);
\draw (0,-2.75)--(1,-2.25);
\draw [style=dashed] (2,-2.25)--(3,-2.25);
\draw (3,-2.25)--(4,-2.25)--(5,-1.75);
\draw (4,-2.25)--(5,-2.75);
\node [color=white] at (0,-1.75) {$\bullet$};
\node at (0,-1.75) {$\circ$};
\node at (2.5,-2.75) {$\widetilde{D}_n,\,n\geq 4$};
\node at (6,-2.25) {$\bullet$};
\node at (7,-2.25) {$\bullet$};
\node at (8,-2.25) {$\bullet$};
\node at (9,-2.25) {$\bullet$};
\node at (10,-2.25) {$\bullet$};
\node at (8,-1.25) {$\bullet$};
\draw (6,-2.25)--(10,-2.25);
\draw (8,-2.25)--(8,-1.25);
\draw (8,-1.25)--(7,-1.25);
\node [color=white] at (7,-1.25) {$\bullet$};
\node at (7,-1.25) {$\circ$};
\node at (8,-2.75) {$\widetilde{E}_6$};
\node at (12,-2.25) {$\bullet$};
\node at (13,-2.25) {$\bullet$};
\node at (14,-2.25) {$\bullet$};
\node at (15,-2.25) {$\bullet$};
\node at (16,-2.25) {$\bullet$};
\node at (17,-2.25) {$\bullet$};
\node at (14,-1.25) {$\bullet$};
\draw (14,-2.25)--(14,-1.25);
\draw (11,-2.25)--(17,-2.25);
\node [color=white] at (11,-2.25) {$\bullet$};
\node at (11,-2.25) {$\circ$};
\node at (14,-2.75) {$\widetilde{E}_7$};
\node at (0,-4.5) {$\bullet$};
\node at (1,-4.5) {$\bullet$};
\node at (2,-4.5) {$\bullet$};
\node at (3,-4.5) {$\bullet$};
\node at (4,-4.5) {$\bullet$};
\node at (5,-4.5) {$\bullet$};
\node at (6,-4.5) {$\bullet$};
\node at (2,-3.5) {$\bullet$};
\draw (0,-4.5)--(7,-4.5);
\draw (2,-4.5)--(2,-3.5);
\node [color=white] at (7,-4.5) {$\bullet$};
\node at (7,-4.5) {$\circ$};
\node at (3.5,-5) {$\widetilde{E}_8$};
\node at (9,-4.5) {$\bullet$};
\node at (10,-4.5) {$\bullet$};
\node at (11,-4.5) {$\bullet$};
\node at (12,-4.5) {$\bullet$};
\draw (8,-4.5)--(12,-4.5);
\node at (10.5,-4.2) {$4$};
\node [color=white] at (8,-4.5) {$\bullet$};
\node at (8,-4.5) {$\circ$};
\node at (10,-5) {$\widetilde{F}_4$};
\node at (14,-4.5) {$\bullet$};
\node at (15,-4.5) {$\bullet$};
\draw (13,-4.5)--(15,-4.5);
\node at (14.5,-4.2) {$6$};
\node [color=white] at (13,-4.5) {$\bullet$};
\node at (13,-4.5) {$\circ$};
\node at (14,-5) {$\widetilde{G}_2$};
\end{tikzpicture}
\caption{Irreducible affine Coxeter systems}\label{fig:affineclassification}
\end{figure}

\end{appendix}

\medskip

\noindent James Parkinson\newline
School of Mathematics and Statistics\newline
University of Sydney\newline
Carslaw Building, F07\newline
NSW, 2006, Australia\newline
\texttt{jamesp@maths.usyd.edu.au}

\end{document}